\documentclass[a4paper, reqno,12pt]{amsart}

\usepackage{tikz, tikz-3dplot}
\usepackage{amsmath, amsthm, amssymb, graphics }

\theoremstyle{plain}
\newtheorem{te}{Theorem}[section]
\newtheorem{lem}[te]{Lemma}

\newtheorem{lemma}[te]{Lemma}
\newtheorem{theorem}[te]{Theorem}
\newtheorem{corollary}[te]{Corollary}

\newtheorem{pr}[te]{Proposition}
\newtheorem{de}[te]{Definition}

\newtheorem{qu}[te]{Question}

\theoremstyle{remark}

\newtheorem*{ack*}{Acknowledgment}

\def\w{{\bf w}}

\def\n{{\bf n}}
\def\b{{\bf b}}
\def\t{{\bf t}}

\def\0{{\bf 0}}

\def\I{{\mathbb I}}

\def\T{{\mathbb T}}
\def\R{{\mathbb R}}

\def\C{{\mathbb C}}
\def\S{{\mathbb S}}
\def\Z{{\mathbb Z}}
\def\P{{\mathbb P}}

\def\W{{\mathbb W}}

\def\nint{\mathop{\diagup\kern-13.0pt\int}}

\def\les{{\;\lessapprox}\;}
\def\dist{{\operatorname{dist}\,}}

\def\Nc{{\mathcal N}}
\def\Tc{{\mathcal T}}

\def\Qc{{\mathcal Q}}
\def\Pc{{\mathcal P}}

\begin{document}
\author{Hongki Jung}
\address{Department of Mathematics, Indiana University,  Bloomington IN}
\email{jung11@iu.edu}

\begin{abstract}
We prove a sharp $l^{10}(L^{10})$ decoupling for the moment curve in $\R^3$. The proof involves a two-step decoupling combined with new incidence estimates for planks, tubes and plates.
\end{abstract}	

\title[A sharp $L^{10}$ decoupling for the twisted cubic]{A sharp $L^{10}$ decoupling for the twisted cubic}
\maketitle

\section{Introduction}
Let $\Gamma_d=\{(t,t^2,\dots, t^d):t\in[0,1]\}$ be the  moment curve in $\R^d$. For $\delta<1$, we introduce its isotropic $\delta$-neighborhood  (essentially $\Gamma_d+B_d(0,\delta)$)
$$\Nc_{\Gamma_d}(\delta)=\{(t,t^2+s_2,\ldots,t^d+s_d):\;s_2^2+\ldots+s_{d}^2\le \delta^2\}.$$
For $\sigma<1$, let $\I_{\sigma}$ be the partition of $[0,1]$ into intervals $I$ of length $\sigma$. For each function $F:\R^d\to\C$, we denote the Fourier projection of $F$ to the infinite strip $I\times \R^{d-1}$ by $\Pc_I F$,
$$\Pc_I F(x)=\int_{I\times \R^{d-1}}\widehat{F}(\xi)e(\xi\cdot x)d\xi.$$
We will consider functions $F$ with Fourier transform supported in $\Nc_{\Gamma_d}(R^{-1})$, and intervals $I\in\I_{R^{-\alpha}}$, with $\frac1d\le \alpha\le 1$. Note that $\Nc_{\Gamma_d}(R^{-1})$ is partitioned by the sets
$$\Nc_{I}(\delta)=\{(t,t^2+s_2,\ldots,t_d+s_d):\;t\in I\text{ and }s_2^2+\ldots+s_{d}^2\le \delta^2\}.$$
So $P_{I}F$ is in fact the Fourier projection of $F$ onto $\Nc_{I}(\delta)$. The following decoupling program for curves was initiated  in \cite{BD3}.
\begin{qu}[$l^p(L^p)$ decoupling at frequency scale $R^{-\alpha}$]
	What is the largest $p_{d,\alpha}$ such that the following $l^p(L^p)$ decoupling holds for all $2\le p\le p_{d,\alpha}:$
	for each $F:\R^d \to \C$ with Fourier transform supported in $\Nc_{\Gamma_d}(R^{-1})$ we have
	\begin{equation}
	\label{hhuihurh}
	\|F\|_{L^{p}(\R^d)}\lesssim_\epsilon R^{\alpha(\frac12-\frac{1}{p})+\epsilon}(\sum_{I\in \I_{R^{-\alpha}}}\|\Pc_I F\|_{L^{p}(\R^d)}^{p})^{\frac{1}{p}}.
	\end{equation}
\end{qu}
Until recently,  the most relevant case for applications, and the most sought after, was the one corresponding to  $\alpha=\frac1d$. The case $d=2$ has been solved in \cite{BD}, showing that $p_{2,\frac12}=6$.  After some preliminary progress in \cite{BD3}, the case $d\ge 3$ has also been solved in \cite{BDG}, showing that $p_{d,\frac1d}=d(d+1)$. The impact of this result was enormous, as it immediately led to the resolution of the Main Conjecture in Vinogradov's Mean Value Theorem in number theory.

Perhaps the second most important scale is $\alpha=\frac12$. This case has played an important role in Bourgain's progress \cite{Bo13} on the Lindel\"of hypothesis regarding the growth of the Riemann zeta on the critical line. The results for this scale are very sparse, and typically not sharp. Since $\alpha=\frac12$ will be our main concern here,  we will call $p_d$ the critical exponent  $p_{d,\frac12}$. Using a simple projection argument (the so-called cylindrical decoupling), it is easy to see that $p_d$ is non-decreasing in $d$.

In \cite{BD} it was proved  that $p_2=6$.
Testing this inequality with exponential sums shows that $p_d\le 4d-2$, see for example Section 9 in \cite{DGS}. In particular, $p_3\le 10$, $p_4\le 14$, $p_5\le 18$.

The following computation for higher values of $d$ is due to Bourgain, \cite{Pri}.
Let $E$ be the extension operator associated with $\Gamma_d$ and let $\textbf{1}$ be a smooth approximation of $1_{[0,1]}$. Then, a result in \cite{Bra} shows that for each $p\ge 2$

$$\|E\textbf{1}\|_{L^p(B(0,R))}\gtrsim R^{\frac{d}{p}-\sigma_{p,d}}$$ with
$$\sigma_{p,d}=\min_{2\le k\le d}\{\frac1k+\frac{k^2-k-2}{2kp}\}.$$
On the other hand, if $J\in\I_{R^{-1/2}}$,
$$\|E_J\textbf{1}\|_{L^p(B(0,R))}\sim R^{\frac{2d-1}{2p}-\frac12}.$$
So $l^p(L^p)$ decoupling can only have a chance to  hold if $\frac{2d-1}{2p}-\frac12+\frac14\ge \frac{d}p-\sigma_{p,d}$, or $\sigma_{p,d}\ge \frac14+\frac1{2p}$. This does not give anything new if $d=3,4$. But for $d\ge 5$, it implies that
$$\frac14+\frac1{2p}\le \min_{5\le k\le d}\{\frac1k+\frac{k^2-k-2}{2kp}\},$$
or
$$p\le \min_{5\le k\le d}\frac{2(k^2-2k-2)}{k-4}.$$
Since the smallest value for $\frac{2(k^2-2k-2)}{k-4}$ when $k\ge 5$ is 22 (for both $k=6$ and $k=7$), and since $p_d$ is non decreasing, we conclude that
$$p_d\le 22, \;\;\;\text{for all }d.$$
This leads to a surprising conclusion. While $p_d$ is nondecreasing in $d$, it does not grow to $\infty$.
\smallskip

There are two types of methods in the literature, let us call them {\em soft} and {\em hard},  that have been used to produce $l^p(L^p)$ decoupling for curves for $\alpha>\frac1d$. Both methods first derive  $l^p(L^p)$  decouplings in the multilinear setting, and then use a standard  multilinear-to-linear reduction to get linear decouplings. The {\em soft} method proves the multilinear decoupling for $\Gamma_d$ by reducing it to a known decoupling for a higher dimensional manifold in $\R^d$. This approach originates in \cite{BD3}, where a bilinear argument was used for $\Gamma_3$. The sum of two separated pieces of $\Gamma_3$ is a surface in $\R^3$ with nonzero Gaussian curvature. The critical exponent for the decoupling of such a surface (also at scale $\alpha=\frac12$) was proved to be $4$. This immediately shows that $p_3\ge 8=4\times 2$ for $\Gamma_3$.
Another example is the result in \cite{Bo13}. Using a bilinear approach, the inequality $p_4\ge 12$ was proved there, as the main step towards improving the state of the art on the Lindel\"of hypothesis. The reproof of $p_4\ge 12$ in \cite{BD13} clarifies the nature of the argument, and of the exponent 12. Sums of two separated pieces of $\Gamma_4$ create a surface in $\R^4$, whose critical exponent as far as decoupling at scale $\alpha=\frac12$ is concerned, is 6. So again, $12=6\times 2$. The trilinear result at $p=\frac{14}{3}$ in \cite{DGS} can be used to prove that $p_5\ge 14$, \cite{DG}. Similar applications of the soft method for larger values of $d$ appear in \cite{Oh}.

The soft method is probably never sharp, as there is loss of information and structure in turning curves into higher dimensional manifolds. We illustrate this by proving our main result, the sharp estimate $p_3=10$. We find it also likely that $p_4$ is greater than $12$, and also that $p_5$ is greater than $14$. It is likely that our methods here could be adapted, not without additional significant effort, to prove such estimates.
\smallskip

We now state the main theorem of this paper.
\begin{te}
Assume that function $F:\R^3 \to \C$ has the Fourier transform supported in $\Nc_{\Gamma_3}(R^{-1})$. Then for $2\leq p\leq 10$, we have
\begin{equation}
\label{te:main}
\|F\|_{L^{p}(\R^3)}\lesssim_\epsilon R^{\frac12(\frac12-\frac{1}{p})+\epsilon}(\sum_{J\in \I_{R^{-1/2}}}\|\Pc_J F\|_{L^{p}(\R^3)}^{p})^{\frac{1}{p}}.
\end{equation}
\end{te}
Combined with the aforementioned inequality $p_{3}\le 10$ proved in \cite{DGS}, this shows the sharp estimate $p_{3}=10$. As another measure of the sharpness of our result, we mention that the projection argument mentioned earlier shows that $l^2(L^p)$ decoupling
$$\|F\|_{L^{p}(\R^d)}\lesssim_\epsilon R^{\epsilon}(\sum_{I\in \I_{R^{-\alpha}}}\|\Pc_I F\|_{L^{p}(\R^d)}^{2})^{\frac{1}{2}}$$
cannot hold if $\alpha>\frac13$ and $p>6$. Indeed, it suffices to test this inequality with functions $F$ Fourier supported in $\Nc_{[0,R^{-1/3}]}(R^{-1})$, and to note that this curved tube is essentially planar.
\smallskip

Our approach is an application of the {\em hard} method introduced in \cite{DGH}.
Recall that the general context is about finding $p_{d,\alpha}$, with $\frac1d\le \alpha\le 1$. The index $\alpha=\frac1d$ is special, since it allows the efficient use of (the canonical version for $\Gamma_d$ of) parabolic rescaling. Because of this, the scale $\alpha=\frac1d$ is called {\em canonical scale}, and decoupling at this scale is sometimes called {\em canonical decoupling}. The main result in \cite{BDG} settles the canonical decoupling for $\Gamma_d$, by making repeated use of parabolic rescaling. This tool however becomes much less efficient in the case $\alpha>\frac1d$, referred to as {\em small cap decoupling}, due to the size $R^{-\alpha}\ll R^{-1/d}$ of the frequency intervals we decouple into.

The approach in \cite{DGH} circumvents this problem, by introducing  a two-step decoupling argument, with a significant emphasis on wave packet decompositions and the incidence geometry of their spatial supports. Parabolic rescaling is only used in the multilinear-to-linear reduction, but not in the main body of the argument. This method has led in \cite{DGH} to sharp results for special functions $F$ with frequency support near $\Gamma_3$. In particular, inequality \eqref{te:main} was proved there in the case when $F$ is an exponential sum that is periodic in a certain direction.

At large scales, our argument here is structured similarly to the one in \cite{DGH}. There is a decoupling into larger intervals $I$ of canonical scale $R^{-1/3}$, and then a decoupling of each $I$ into subintervals $J$ of length $R^{-1/2}$. There are however two major new difficulties that are consequences of us  considering arbitrary functions $F$ in \eqref{te:main}. The first has to  do with the complicated nature of each term $\Pc_J F$ in our setup. The result in \cite{DGH} is about the case when each of these functions is a single exponential wave. The (local versions of the) quantities $\|\Pc_J F\|_{L^{p}(\R^3)}$ in this simplified context are easily computable.

The second issue is the complete lack of periodicity in our context. Wave packets in \cite{DGH} have a lot of structure, due to periodicity  in the $x$-direction of the exponential sum.

To address these issues, our argument introduces a few innovations. We need to investigate  the complex nature of $\Pc_J F$ by introducing a second wave packet decomposition at scale $R^{-1/2}$, in addition to the one at scale $R^{-1/3}$. There will be two multi-layer pigeonholing arguments in Section \ref{Sec8}, one for each decomposition. The lack of periodicity increases the number of rounds of pigeonholing needed in each sequence. The connection between the  parameters arising throughout these pigeonholing steps is mainly captured by Proposition \ref{te:246decoupling}. The function $F$ is split into components that are ``uniform" at both scales. Decoupling each $I$ into intervals $J$ will involve $L^2$ orthogonality, the $l^2(L^6)$ canonical decoupling for the parabola, and the $l^4(L^4)$ small cap decoupling from \cite{DGH}.

Most of the work goes into proving a refined decoupling into intervals $I$. The main component of this part is Proposition \ref{te:Plankinci}. This is an incidence estimate between Vinogradov planks that satisfy a certain spacing assumption. The proof uses new estimates for incidences of tubes and plates, proved in Corollary \ref{co:halfwellspace} and Theorem \ref{te:L4plate}. The proof of the corollary uses the method in the recent papers \cite{GSW} and \cite{GWZ}. On the other hand, Theorem \ref{te:L4plate} is a new $L^4$ estimate for plates, that we prove using a similar $L^4$ estimate for planks, Theorem \ref{te:ell4}. The proof of this latter result is purely geometric, it exploits the size of the intersections of planks. $L^4$ is optimal in this context.

Similar to \cite{DGH}, we first prove the trilinear version of the desired decoupling, see Theorem \ref{rejhuvtopcekw}. The only advantage we get from working in the trilinear setup, is the access to the $L^6$ trilinear reverse square function estimate for $\Gamma_3$, used in the proof of Theorem \ref{ejfyer7f0-p-i98t90-}. However, unlike the argument from \cite{DGH}, ours makes no use of either bilinear or trilinear incidence estimates for boxes. All incidence estimates we use for Vinogradov tubes, plates and planks are linear. We caution that while Vinogradov boxes live in $\R^3$, they only have one degree of freedom (essentially their orientation is dictated by the restricted family of tangents/normals to $\Gamma_3$). Because of this, their incidences have {planar} complexity. Our results are (somewhat complicated) manifestations of the planar Kakeya phenomenon, as opposed to genuinely $\R^3$ results.
Nevertheless, we believe that  the new incidence estimates are of independent interest, and will serve as tools in future literature.

It seems possible that our inequality \eqref{te:main} will help with progress on various problems concerned with the $L^{10}$ moment of exponential sums. For example, such moments were investigated in \cite{Bo34} using soft decoupling techniques. There the equality  $10=\frac{10}{3}\times 3$ is exploited in conjunction with the $L^{10/3}$ decoupling for hypersurfaces in $\R^4$, but the results following this approach are suboptimal.

\bigskip

\begin{ack*}
I would like to thank my advisor Ciprian Demeter for his guidance and constant support throughout the completion of the project.
\end{ack*}

\bigskip

\section{Notation}
\label{Preli}

We write $X\lesssim Y$ to denote $X\leq CY$ for some constant $C$. Also $X\sim Y$ denotes $X\lesssim Y$ and $X\gtrsim Y$. We write $\lesssim_{par}$ if the constant $C$ has dependence on the parameter $par$. Whenever we  encounter quantities $X,Y$ depending on the scale parameter $R$, the notation $X\lessapprox Y$ stands for $X\lesssim (\log R)^C Y$ for some $C=O(1)$.

We define $\chi(x)=(1+|x|)^{-100}$.

We will denote the twisted cubic $\Gamma_3$ simply by $\Gamma$. For each $I=[c,c+\delta]\in \I_{{\delta}}$, we denote the unit tangent, normal and binormal vectors at $\Gamma(c)$ (the Frenet frame) by $\textbf{t}(I), \textbf{n}(I), \textbf{b}(I)$ respectively. We will define numerous boxes in $\R^3$ using these directions.

A polytope $B$ (typically the intersection of multiple rectangular boxes) is called an almost rectangular box if $C^{-1}R \subset B \subset CR$ for some genuinely rectangular box $R$ and $C=O(1)$. In our paper all rectangular boxes $R$ will have the three axes oriented along a Frenet frame. The dimensions $(d_1,d_2,d_3)$ of $B$ will be understood to be those of $R$, with respect to the $\t,\n,\b$ axes. Of course, they are well defined only within a multiplicative factor $O(1)$. This slight imprecision is harmless in our argument.
We will not distinguish between  two boxes $B_1$ and $B_2$ if $C^{-1}B_2 \subset B_1 \subset CB_2$ for $C=O(1)$. The dimensions of various boxes $B$ associated with an interval $I$ will be in such a way that the orientation of  $B$ remains essentially unchanged when specified with respect to the   Frenet frame at any $\Gamma(t)$ with $t\in I$.

We will focus on four types of boxes: $cube$, $plank$, $plate$ and $tube$. A cube is a box with identical dimensions $d_1,d_2,d_3$. A plank  has three considerably different dimensions $d_1,d_2,d_3$. A box is called plate if the lengths of two edges are comparable and significantly longer than the remaining edge. On the contrary, a tube has  two edges with comparable length while the remaining edge has considerably longer length. We will not distinguish between tubes and cylinders.

We will encounter planks and plates as part of various decompositions both on the spatial and on the frequency side. Typically, cubes and tubes will live on the spatial side.
Each of these boxes will be associate with an interval in $[0,1]$, that determines their orientation via the associated Frenet frame. We will call a collection of such boxes separated if the boxes corresponding to each interval (these boxes are mutually parallel) have bounded overlap.

Let $B_1$ be a $(d_1,d_2,d_3)$-box and $B_2$ be a $(d_1',d_2',d_3')$-box. The boxes $B_1$ and $B_2$ are said to be dual if they have the same orientation and satisfy $d_id_i'= 1$, for each $i=1,2,3$. We sometimes  write $B^*$ for a dual box of $B$.

Throughout this paper, we will use the family of linear maps $T_{\sigma,c}$ which is defined by
$$T_{\sigma, c}(w_1,w_2,w_3)=(w_1',w_2',w_3')$$
\begin{equation*}
\begin{cases}
w_1'=\frac{w_1}{\sigma}\\
w_2'=\frac{w_2-2cw_1}{\sigma^2}\\
w_3'=\frac{w_3-3cw_2+3c^2w_1}{\sigma^3}.
\end{cases}
\end{equation*}
We write the family of the inverse transpose linear maps $A_{\sigma,c}=(T_{\sigma, c}^{-1})^T$ by
$$A_{\sigma, c}(x_1,x_2,x_3)=(x_1',x_2',x_3')$$
\begin{equation*}
\begin{cases}
x_1'=\sigma(x_1+2cx_2+3c^2x_3)\\
x_2'=\sigma^2(x_2+3cx_3)\\
x_3'=\sigma^3x_3.
\end{cases}
\end{equation*}

\section{A partition for frequency Vinogradov planks}
\label{Ch:parti}

Recall that $\mathbb{I}_{R^{-{1/3}}}$ is the partition of $[0,1]$ into intervals $I$ of length $R^{-\frac{1}{3}}$. For each $I$ we define the anisotropic neighborhood of the arc $\Gamma_I$ by $$\Gamma_I(R^{-1/3}):=\{(\xi_1,\xi_2,\xi_3):\xi_1\in I, |\xi_2-\xi_1^2|\leq R^{-2/3},|\xi_3-3\xi_1\xi_2+2\xi_1^3|\leq R^{-1}\}.$$
Note that the sets $\Gamma_I(R^{-1/3})$ with $I\in \mathbb{I}_{R^{-{1/3}}}$ cover
$\Gamma(R^{-1/3})=\Gamma_{[0,1]}(R^{-1/3})$.

Each $\Gamma_I(R^{-1/3})$ is an almost rectangular $(R^{-\frac13},R^{-\frac23},R^{-1})$-plank with respect to the  $\textbf{t}(I)$, $\textbf{n}(I)$, $\textbf{b}(I)$ axes. The orientation of the plank is characterized by the direction of the long $R^{-\frac13}$-edge and by the normal direction of the wide $(R^{-\frac13},R^{-\frac23})$-face. Using this fact, we replace almost rectangular boxes $\Gamma_I(R^{-1/3})$ by parallelepiped planks $\theta_I$ with dimensions $(R^{-\frac13},R^{-\frac23},R^{-1})$. We can check that the plank $\theta_I$ satisfies $\frac{1}{10} \theta_I \subset \Gamma_I(R^{-\frac13}) \subset 10 \theta_I$.

\begin{de} [Frequency Vinogradov planks at scale $R$]
Let $I=[c,c+R^{-\frac13}]$, $c\in [0,1]$. The  $(R^{-\frac13},R^{-\frac23},R^{-1})$-plank $\theta_I$, and any of its translates, will be called a frequency Vinogradov plank (at scale $R$) associated with $I$. $\theta_I$ has the long edge pointing the direction $\textbf{t}(I)$ and the wide $(R^{-\frac13},R^{-\frac23})$-face with normal vector $\textbf{b}(I)$. The translate centered at the origin will be  denoted by $\tilde{\theta}_I$ or $\tilde{\theta}_c$.
\end{de}

In the recent paper \cite{GWZ}, the authors showed Kakeya-type estimate for the frequency planks associated with the cone in $\R^3$. In that argument, the union of the planks centered at the origin was partitioned by families of small (punctured) planks.
 We show that the frequency Vinogradov planks centered at the origin can also be partitioned by the families of small planks with the same properties. We will use this partition in the next section.
\smallskip

Let us describe the $base$ plank $\tilde{\theta}_0$ and the Vinogradov planks $\tilde{\theta}_I$ centered at the origin as
$$\tilde{\theta}_0 :=\{ (x,y,z)\in \mathbb{R}^3 : |x|\leq R^{-\frac{1}{3}},|y|\leq R^{-\frac{2}{3}},|z|\leq R^{-1} \}$$
$$\tilde{\theta}_I :=\{ w \in \mathbb{R}^3: |w_1|\leq R^{-\frac13},|w_2-2c w_1|\leq R^{-\frac23}, |w_3-3c w_2+3c^2w_1|\leq R^{-1}\}.$$
From this representation, we see that every Vinogradov plank $\tilde{\theta}_I$ is the image of base plank $\tilde{\theta}_0$ under the linear transformation $T_{1,-c}(\tilde{\theta}_0)$. Each frequency Vinogradov plank is a translate of $\tilde{\theta}_I$.

In the case of the cone, two planks can be transformed to one another by planar rotation. The rotation induces a family of circles covering the union of planks. These circles are used to localize the overlap of the planks. In the case of Vinogradov planks, the linear transformation $T_{1,-c}$ has a similar role as that of rotation. We devise a family of curves associated with the linear map $T_{1,-c}$, covering the union of Vinogradov planks.

Let us denote the union of the planks $\cup_{I\in \I_{R^{-1/3}}}\tilde{\theta}_I$ by $\Omega$. We define the family of curves $C_{(x,y,z)}=\{C_{(x,y,z)}(s)\in \R^3:s\in[0,1]\}$ associated with $(x,y,z)\in \tilde{\theta}_0$ by $$C_{(x,y,z)}(s):=T_{1,-s}(x,y,z)=(x,y+2sx,z+3sy+3s^2x).$$
Recalling that $T_{1,-c}$ maps $\tilde{\theta}_0$ onto $\tilde{\theta}_{c}$ for any $c\in [0,1]$, we can observe  $$\Omega \subset \cup_{(x,y,z)\in \tilde{\theta}_0} C_{(x,y,z)}.$$
On the other hand, any point on $C_{(x,y,z)}$ is contained in (a slight enlargement of) the union of Vinogradov planks $\Omega$. Let $s\in [0,1]$ and $|s-c|\leq R^{-\frac13}$. Then $C_{(x,y,z)}(s)\in 7\tilde{\theta}_{c}$. Therefore $\cup_{(x,y,z)\in \tilde{\theta}_0} C_{(x,y,z)}$ is essentially equal to $\Omega$.

For $0\leq s \leq 1$ and dyadic $R^{-\frac13}\leq\sigma\leq 1$, we define the small plank $\Theta(\sigma,s)$ at scale $\sigma$ by
$$\Theta(\sigma,s)=\{ w \in \mathbb{R}^3: |w_1|\leq R^{-\frac13}\sigma^2,|w_2-2sw_1|\leq R^{-\frac23}\sigma, |w_3-3sw_2+3s^2w_1|\leq R^{-1}\}.$$
Let us call $\Theta(\sigma,0)$ the $base$ plank at scale $\sigma$ so that the other small planks can also be described by $\Theta(\sigma,s)=T_{1,-s}(\Theta(\sigma,0))$.

Note that two small planks $\Theta(\sigma,s)$ and $\Theta(\sigma,s')$ are distinguishable only if $|s-s'|\gtrsim R^{-\frac13}\sigma^{-1}$. If $|s-s'|> (\frac74)R^{-\frac13}\sigma^{-1}$, then $\Theta(\sigma,s)\cap \Theta(\sigma,s')\cap \{ (\frac47)R^{-\frac13}\sigma^2\leq w_1\leq R^{-\frac13}\sigma^2\}$ is empty so that the two planks are distinct. On the other hand, if $|s-s'|\leq R^{-\frac13}\sigma^{-1}$, then $\Theta(\sigma,s)\subset 7\Theta(\sigma,s')$. 

For dyadic scales $R^{-\frac13}\leq \sigma\leq 1$, let $\textbf{CP}_{\sigma}$ be the set of $\sim\sigma R^{\frac13}$ small planks $\Theta(\sigma,s')$ with $s'$ evenly spaced on the interval $[0,1]$.
Similarly as before, we have
$$\cup_{\Theta\in\textbf{CP}_{\sigma}}\Theta(\sigma,s') \subset \cup_{(x,y,z)\in \Theta(\sigma,0)}C_{(x,y,z)}.$$
Also, for $0\leq s \leq 1$, we demand $\textbf{CP}_\sigma$ to include $\Theta(\sigma,s')$ such that $|s-s'|\leq R^{-\frac13}\sigma^{-1}$ and
$C_{(x,y,z)}(s)\in 7\Theta(\sigma,s').$
Therefore $\cup_{\Theta\in\textbf{CP}_{\sigma}}\Theta$ is essentially the same as $\cup_{(x,y,z)\in \Theta(\sigma,0)}C_{(x,y,z)}$.

We associate each small plank $\Theta(\sigma,s')$ with $\tilde{\theta}_c$ satisfying $|c-s'|\leq R^{-\frac13}\sigma^{-1}$. For each $\Theta$, let $\textbf{S}_{\Theta}$ be the set of $\tilde{\theta}$ associated with $\Theta$. Then each $\tilde{\theta}_c$ is included in $\sim O(1)$ many $\textbf{S}_{\Theta}$ and if $\tilde{\theta}_c\in \textbf{S}_{\Theta}$, then $\Theta\subset 7\tilde{\theta}_c$.

For dyadic $R^{-\frac13}\leq \sigma \leq 1$, let $\Omega_{\leq \sigma}=\cup_{\Theta\in\textbf{CP}_{\sigma}}\Theta$ and set up $\Omega_{\leq 1} = \Omega$. Let $\Omega_{\sigma}=\Omega_{\leq\sigma}\setminus\Omega_{\leq\sigma/2}$ and $\Omega_{R^{-1/3}}=\Omega_{\leq R^{-1/3}}$. Then we can partition the union of Vinogradov planks by
$$\Omega=\bigsqcup_{R^{-1/3}\leq \sigma\leq 1}\Omega_{\sigma}.$$

The following lemmas are the analogues of Lemma 4.1 and 4.2 from \cite{GWZ} for frequency Vinogradov planks centered at the origin. We first analyze the intersection of $\tilde{\theta}_c$ and $\Omega_\sigma$.

\begin{lemma}
\label{le:parti1}
Assume that $w\in \tilde{\theta}_c\cap\Omega_\sigma$ and $\tilde{\theta}_c\in \textbf{S}_{\Theta}$ for some $\Theta(\sigma,s')\in \textbf{CP}_\sigma$. Then $w\in 49\Theta$.
\end{lemma}

\begin{proof}
Recall that $\tilde{\theta}_c\in \textbf{S}_{\Theta}$ implies $|s'-c|\leq R^{-\frac13}\sigma^{-1}$ and $\Theta(\sigma,c)\subset 7\Theta(\sigma,s')$. We claim that $\tilde{\theta}_c\cap \Omega_\sigma$ is contained in $7\Theta(\sigma,c)$. This claim directly shows that $w\in 7\Theta(\sigma,c)\subset 49\Theta(\sigma,s')$. Since $\Omega_{\sigma}$ is covered by $\cup_{(x,y,z)\in \Theta(\sigma,0)}C_{(x,y,z)}$, we have $$(\tilde{\theta}_{c}\cap \Omega_{\sigma})\subset (\tilde{\theta}_{c}\cap( \cup_{(x,y,z)\in\Theta(\sigma,0)}C_{(x,y,z)})).$$
Hence we are able to write $w=C_{(x,y,z)}(s)$ for some $(x,y,z)\in\Theta(\sigma,0)$ and $s\in [0,1]$. The condition $w\in \tilde{\theta}_c\cap \Omega_\sigma$ is equal to the following inequalities
\begin{equation}
\label{ineq:1}
\begin{cases}
|w_1|=|x|\leq R^{-1/3}\\
|w_2-2cw_1|=|y+2x(s-c)|\leq R^{-2/3}\\
|w_3-3cw_2+3c^2w_1|=|z+3y(s-c)+3x(s-c)^2|\leq R^{-1}.
\end{cases}
\end{equation}
We check $w \in 7 \Theta(\sigma,c)$, or equivalently, check $C_{(x,y,z)}(s) $ satisfies the following inequalities
\begin{equation}
\label{ineq:2}
\begin{cases}
|w_1|=|x|\leq R^{-1/3}\sigma^2\\
|w_2-2cw_1|=|y+2x(s-c)|\leq 7 R^{-2/3}\sigma\\
|w_3-3cw_2+3c^2w_1|=|z+3y(s-c)+3x(s-c)^2|\leq R^{-1}.
\end{cases}
\end{equation}
The first inequality is trivial because $|x|\leq R^{-1/3}\sigma^2$ and the last inequality is the same as the last one in (\ref{ineq:1}). Let us prove the second inequality of (\ref{ineq:2}). If $|y(s-c)|\geq \frac12|x(s-c)^2|$ then $|s-c|\leq |2y/x|$ and the second inequality follows. If $|y(s-c)|\leq \frac12|x(s-c)^2|$, then by using the triangle inequality to the last inequality of (\ref{ineq:1}), $|s-c|\leq 2|R^{-1}/x|^{1/2}$ holds. This implies the second inequality of (\ref{ineq:2}).
\end{proof}

Next, we investigate the intersections between  $49\Theta \cap \Omega_\sigma $ for $\Theta\in \textbf{CP}_\sigma$.

\begin{lemma}
\label{le:parti2}
For $w\in \Omega_{\sigma}$, the number of $\Theta\in\textbf{CP}_{\sigma}$ such that $w\in 49\Theta$ is bounded by a constant $C$.
\end{lemma}

\begin{proof}
Let $(x,y,z)\in \Theta(\sigma,0) \setminus\Theta(\sigma/2,0)$. Then $\Omega_\sigma=\Omega_{\leq \sigma}\setminus\Omega_{\leq \sigma/2}$ is covered by $\cup_{(x,y,z)}C_{(x,y,z)}$.
Recall that the union of $\sim \sigma R^{1/3}$ planks $\Theta(\sigma,s')\in \textbf{CP}_{\sigma}$ essentially covers $C_{(x,y,z)}$.
We claim that for any $\Theta(\sigma,s')\in \textbf{CP}_\sigma$, the fraction of $C_{(x,y,z)}$ contained in $ 49 \Theta(\sigma,s')$ is $\sim\sigma^{-1}R^{-1/3}$.

The arc length of the curve $C_{(x,y,z)}$ for $s\in[a,b]$ is
 $$\int_{a}^{b}\sqrt{(2x)^2+(3y+6xs)^2}ds \sim (|x|+|y|)(b-a)$$
and the length of $C_{(x,y,z)}$ is $\sim (|x|+|y|)$.

Let us first focus on the case $|x|\leq (1/4)R^{-1/3}\sigma^2$. Note that in this $|x|$ range, $|y|\sim R^{-2/3}\sigma$. We divide into two subcases, and assume $R^{-2/3}\sigma\leq|x|\leq (1/4)R^{-1/3}\sigma^2$. The length of $C_{(x,y,z)}$ is $\sim |x|$. In order to find the arc length of $C_{(x,y,z)}\cap 49 \Theta(\sigma,s')$, we evaluate the range of $s$ satisfying the condition $C_{(x,y,z)}(s)\in 49 \Theta(\sigma,s')$ which can be described by the following inequalities
\begin{equation}
\label{ineq:3}
\begin{cases}
|x|\leq R^{-1/3}\sigma^2\\
|y+2x(s-s')|\leq 49 R^{-2/3}\sigma\\
|z+3y(s-s')+3x(s-s')^2|\leq 49R^{-1}.
\end{cases}
\end{equation}
If $\frac12 |y(s-s')|\geq |x(s-s')^2|$, the third inequality shows that $|s-s'|\lesssim R^{-\frac13}\sigma^{-1}$. Otherwise, suppose $\frac12 |y(s-s')|\leq |x(s-s')^2|$ holds. Combining this condition with the second inequality we have $|s-s'|\sim |\frac{y}{x}|$. The third inequality boils down to
$$|\frac{y}{x}+(s-s')|\lesssim \frac{R^{-1}}{|y|}\sim R^{-\frac13}\sigma^{-1}.$$
Therefore, the suitable range of $s$ such that $C_{(x,y,z)}(s)\in49\Theta(\sigma,s')$ has the length $\sim R^{-\frac13}\sigma^{-1}$. The corresponding arc length of $C_{(x,y,z)}\cap 49\Theta(\sigma,s')$ is $\sim R^{-1/3}\sigma^{-1}|x|$. We see that the fraction of $C_{(x,y,z)}$ contained in $49\Theta(\sigma,s')$ is $\lesssim R^{-1/3}\sigma^{-1}$.

For the second subcase, let $|x|\leq R^{-2/3}\sigma$. The length of $C_{(x,y,z)}$ is $\sim |y|$. By the same computation as in the former case, the range of suitable $s$ has the length $\sim R^{-\frac13}\sigma^{-1}$ and the arc length of $C_{(x,y,z)}\cap 49 \Theta(\sigma,s')$ is $\sim R^{-1/3}\sigma^{-1}|y|$. Thus the fraction of $C_{(x,y,z)}$ contained in $49\Theta(\sigma,s')$ is $\lesssim R^{-1/3}\sigma^{-1}$.

Lastly, assume that $(1/4)R^{-1/3}\sigma^2\leq|x|\leq R^{-1/3}\sigma^2$. The length of $C_{(x,y,z)}$ is $\sim |x|$. From the second inequality of (\ref{ineq:3}), we have $|s-s'|\lesssim R^{-2/3}\sigma/|x|\sim R^{-1/3}\sigma^{-1}$. Therefore the arc length of $C_{(x,y,z)}\cap 49 \Theta(\sigma,s')$ is $\sim R^{-1/3}\sigma^{-1}|x|$, and the fraction of $C_{(x,y,z)}$ contained in $49\Theta$ is $\lesssim R^{-1/3}\sigma^{-1}$.

Using the aforementioned claim, let us prove the lemma. Suppose for contradiction that one point $w\in C_{(x,y,z)}$ is contained in $\gg1$ number of $49\Theta(\sigma,s')$. Then any other point $w'$ on the same curve $C_{(x,y,z)}$ is also contained in $\gg1$ number of $343\Theta(\sigma,s')$. By the claim, the union of $\sim \sigma R^{\frac13}$ small planks $\Theta\in \textbf{CP}_\sigma$ fails to cover the curve $C_{(x,y,z)}$. This contradicts the fact that $\cup_{\Theta\in\textbf{CP}_\sigma}\Theta$ essentially covers $C_{(x,y,z)}$.

\end{proof}

\section{Incidence estimate for well-spaced Vinogradov tubes}
\label{Ch:well}

In this section, we show an incidence estimate for well-spaced tubes in $\R^3$ associated with $\Gamma$. We first define spatial Vinogradov planks and tubes.

\begin{de} (Spatial Vinogradov planks at scale $R$)\\
Let $I=[c,c+R^{-\frac13}]\in \I_{R^{-1/3}}$. Let the  $(R^{\frac13},R^{\frac23},R)$-box $P$ has the long edge pointing in the direction $\textbf{b}(I)$ and its wide $(R^{-\frac13},R^{-\frac23})$-face has normal direction $\textbf{t}(I)$. We call such a box a spatial Vinogradov plank at scale $R$, associated with the interval $I\in \I_{R^{-1/3}}$. We will write the plank as $P_I$, $P_c$, or simply as $P$.
\end{de}
We note that  frequency and spatial Vinogradov planks associated with a given $I$ are boxes dual to each other. We mention that  throughout the rest of the paper we will also encounter rescaled spatial Vinogradov planks. These will be planks of the form $LP$, for some scalar $L>0$ and $P$ as above.

\begin{de}[Vinogradov tubes at scale $R$]
	\label{oiufufihvjdkvnjkgbj }
Let $I=[c,c+R^{-\frac13}]\in \I_{R^{-1/3}}$. Assume the $(R^{\frac23},R^{\frac23},R)$-tube has the long edge pointing in the direction $\textbf{b}(I)$. We call it a Vinogradov tube associated with $I$, and denote it by $T_I,T_c$, or simply  as $T$. Families of Vinogradov tubes will typically be denoted by $\T$. The collection of all Vinogradov tubes in $\T$ associated with $I$ will then  be denoted by  $\T(I)$.
\end{de}
The orientation of tube is characterized by its long axis. Therefore, we can describe the general Vinogradov tubes $T_I$ as translates of the following tube at the origin
$$\{ x \in \mathbb{R}^3: |x_1+2cx_2+3c^2x_3|\leq R^{2/3},|x_2+3cx_3|\leq R^{2/3}, |x_3|\leq R\}.$$
Each plank $P_I$ sits naturally inside some tube $T_I$.
If we combine $R^{\frac13}$ many planks $P_I$, consecutively placed in the $\textbf{t}(I)$ direction, the result is a Vinogradov tube $T_I$.

The dual box of tube $T_I$ centered at the origin is
$$\{ w \in \mathbb{R}^3: |w_1|\leq R^{-\frac23},|w_2-2c w_1|\leq R^{-\frac23}, |w_3-3c w_2+3c^2w_1|\leq R^{-1}\}.$$
This is a plate which is the truncation of plank $\tilde{\theta}_I$ to the set $\{w\in \R^3:|w_1|\leq R^{-\frac23}\}$. We will denote it by  $\overline{\theta}_I$. We use this observation to prove the following incidence result for well-spaced Vinogradov tubes. The notation $\lessapprox$ stands for $\lesssim (\log R)^C$ with $C=O(1)$.

\begin{te}
\label{te:inci}
Let $\T$ be a collection of separated $(R^{\frac{2}{3}},R^{\frac{2}{3}},R)$-Vinogradov tubes $T$ in $[-R,R]^3$, satisfying the following well-spacing conditions:
\\
\\
(S1) For each $I\in \I_{R^{-1/3}}$, we cover $[-R,R]^3$ with $(R^{\frac{2}{3}},R,R)$-plates $S_I$ with  normal vector $\t(I)$. Note that each $T\in\T_I$ sits inside one such plate $S_I$.

We assume that each such  $S_I$ contains at most $N$  tubes $T\in \T(I)$.\\ \\
(S2) For each $I\in \I_{R^{-1/3}}$, we cover $[-R,R]^3$ with $(R^{\frac{2}{3}},R^{\frac{5}{6}},R)$-boxes $B_I$ with the axes $\textbf{t}(I)$, $\textbf{n}(I)$, $\textbf{b}(I)$. Note that each $B_I$ sits inside a unique $S_I$.

We assume that each such $B_I$ contains at most $N_1$  tubes $T\in \T(I)$, for some $N_1\le N$.

\smallskip

Let $1\leq r \leq R^{\frac13}$ and $q$ be the cube with side length $\sim R^{\frac{2}{3}}$. We say the cube $q$ is $r$-rich if $q$ intersects at least $r$ tubes $T\in\T$. Let $\Qc_r$ be a collection of pairwise disjoint $r$-rich cubes. Then there is a dyadic scale $R^{-\frac{1}{3}}\leq\sigma\leq1$ and an integer $M_{\sigma}$ with the following properties:
\begin{equation}\label{eqn:Q}
      |\Qc_r|\lessapprox \begin{cases}
    \frac{|\T|M_{\sigma}\sigma R^{\frac{1}{3}}}{r^2}  &\text{if } \sigma\geq R^{-\frac{1}{6}} \\
        \frac{|\T| M_{\sigma}\sigma^3 R^{\frac{2}{3}}}{r^2} &\text{if } \sigma< R^{-\frac{1}{6}}\\
\end{cases}\end{equation}
\begin{equation}\label{eqn:r}
r\lessapprox \begin{cases}
R^{\frac{1}{3}}M_{\sigma}\sigma^2 &\text{if } \sigma\geq R^{-\frac{1}{6}}\\
R^{\frac{2}{3}}M_{\sigma}\sigma^4 &\text{if } \sigma< R^{-\frac{1}{6}}\\
\end{cases}
\end{equation}
and
\begin{equation}\label{eqn:M}
    M_{\sigma}\lesssim \begin{cases}
    \sigma^{-1}\min(N_1,\sigma^{-1}) &\text{if } \sigma\geq R^{-\frac{1}{6}}\\
    \sigma^{-3}R^{-\frac{1}{3}}\min(N_1\sigma^{-1}R^{-\frac{1}{6}},N) &\text{if } \sigma< R^{-\frac{1}{6}}.\\
    \end{cases}
\end{equation}

\end{te}


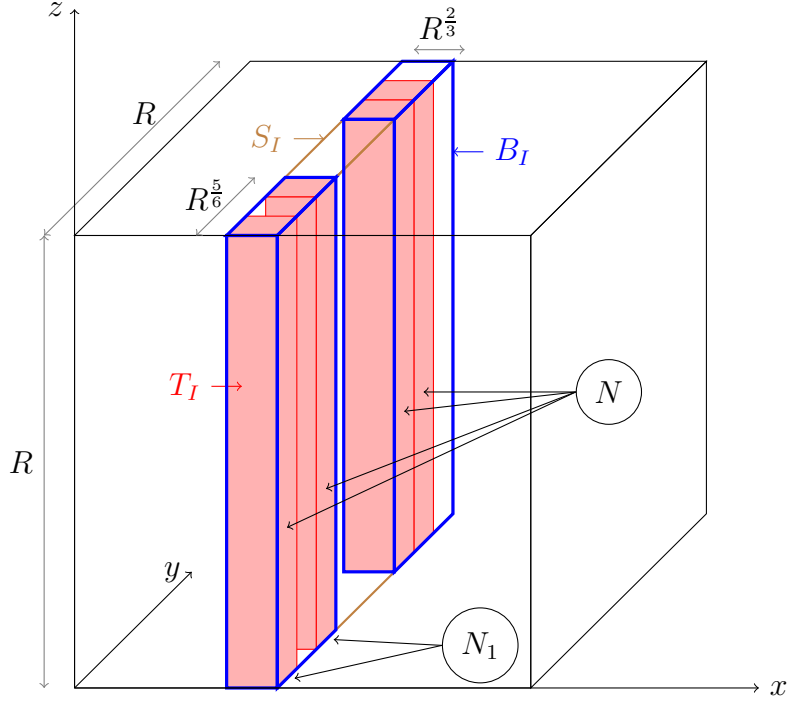
\textbf{Figure 1}
\begin{center}
\begin{tikzpicture}[scale=2,roundnode/.style={circle, draw=black, minimum size=0.1}]

\pgfmathsetmacro{\originonex}{2}
\pgfmathsetmacro{\originoney}{0}
\pgfmathsetmacro{\originonez}{5}

\coordinate (origin) at (-\originonex, -\originoney, 0);

\pgfmathsetmacro{\Sx}{1/3}
\pgfmathsetmacro{\Sy}{3}
\pgfmathsetmacro{\Sz}{1}

\pgfmathsetmacro{\Tx}{1/3}
\pgfmathsetmacro{\Ty}{3}
\pgfmathsetmacro{\Tz}{1/3}

\coordinate (TTorigin) at ($(origin)+(\Tx,\Ty,-2*\Tz)$);

\draw[red, fill=red!30 ] (TTorigin) -- ++(-\Tx, 0,0)-- ++ (0,-\Ty, 0) -- ++ (\Tx, 0, 0)-- cycle;
\draw[red, fill=red!30 ] (TTorigin) -- ++ (0,0, -\Tz) -- ++ (0,-\Ty, 0) -- ++ (0,0, \Tz)--cycle;
\draw[red, fill=red!30 ] (TTorigin) -- ++ (-\Tx, 0,0) -- ++ (0,0,-\Tz) -- ++ (\Tx, 0,0) -- cycle;

\coordinate (Torigin) at ($(origin)+(\Tx,\Ty,0)$);

\draw[red, fill=red!30 ] (Torigin) -- ++(-\Tx, 0,0)-- ++ (0,-\Ty, 0) -- ++ (\Tx, 0, 0)-- cycle;
\draw[red, fill=red!30 ] (Torigin) -- ++ (0,0, -\Tz) -- ++ (0,-\Ty, 0) -- ++ (0,0, \Tz)--cycle;
\draw[red, fill=red!30 ](Torigin) -- ++ (-\Tx, 0,0) -- ++ (0,0,-\Tz) -- ++ (\Tx, 0,0) -- cycle;

\draw[red, ->] ($(Torigin)-(\Tx+0.1,\Ty/3,0)$) -- ++(0.2, 0, 0);
\node[left, red] at ($(Torigin)-(\Tx+0.1,\Ty/3,0)$) {$T_I$};

\pgfmathsetmacro{\Tx}{1/3}
\pgfmathsetmacro{\Ty}{3}
\pgfmathsetmacro{\Tz}{1/3}

\coordinate (TTTTorigin) at ($(origin)+(\Tx,\Ty,-\Tz-2*\Sz)$);

\draw[red, fill=red!30 ] (TTTTorigin) -- ++(-\Tx, 0,0)-- ++ (0,-\Ty, 0) -- ++ (\Tx, 0, 0)-- cycle;
\draw[red, fill=red!30 ] (TTTTorigin) -- ++ (0,0, -\Tz) -- ++ (0,-\Ty, 0) -- ++ (0,0, \Tz)--cycle;
\draw[red, fill=red!30 ] (TTTTorigin) -- ++ (-\Tx, 0,0) -- ++ (0,0,-\Tz) -- ++ (\Tx, 0,0) -- cycle;

\coordinate (TTTorigin) at ($(origin)+(\Tx,\Ty,-2*\Sz)$);

\draw[red, fill=red!30 ] (TTTorigin) -- ++(-\Tx, 0,0)-- ++ (0,-\Ty, 0) -- ++ (\Tx, 0, 0)-- cycle;
\draw[red, fill=red!30 ] (TTTorigin) -- ++ (0,0, -\Tz) -- ++ (0,-\Ty, 0) -- ++ (0,0, \Tz)--cycle;
\draw[red, fill=red!30 ](TTTorigin) -- ++ (-\Tx, 0,0) -- ++ (0,0,-\Tz) -- ++ (\Tx, 0,0) -- cycle;

\pgfmathsetmacro{\Phix}{1/3}
\pgfmathsetmacro{\Phiy}{3}
\pgfmathsetmacro{\Phiz}{3}

\coordinate (Phio) at ($(origin)+(\Phix,\Phiy,0)$); 

\draw[thick, brown] (Phio) -- ++(-\Phix, 0,0)-- ++ (0,-\Phiy, 0) -- ++ (\Phix, 0, 0)-- cycle;
\draw[thick, brown] (Phio) -- ++ (0,0, -\Phiz) -- ++ (0,-\Phiy, 0) -- ++ (0,0, \Phiz)--cycle;
\draw[thick, brown]  (Phio) -- ++ (-\Phix, 0,0) -- ++ (0,0,-\Phiz) -- ++ (\Phix, 0,0) -- cycle;

\draw[brown , ->] ($(Phio)-(\Phix,0,5*\Phiz/9)-(0.2,0,0)$) --++(0.2,0,0);
\node[left, brown] at ($(Phio)-(\Phix,0,5*\Phiz/9)-(0.2,0,0)$) {$S_I$};

\pgfmathsetmacro{\Sx}{1/3}
\pgfmathsetmacro{\Sy}{3}
\pgfmathsetmacro{\Sz}{1}

\coordinate (Sorigin) at ($(origin)+(\Sx,\Sy,0)$);

\draw[blue, very thick] (Sorigin) -- ++(-\Sx, 0,0)-- ++ (0,-\Sy, 0) -- ++ (\Sx, 0, 0)-- cycle;
\draw[blue, very thick] (Sorigin) -- ++ (0,0, -\Sz) -- ++ (0,-\Sy, 0) -- ++ (0,0, \Sz)--cycle;
\draw[blue, very thick] (Sorigin) -- ++ (-\Sx, 0,0) -- ++ (0,0,-\Sz) -- ++ (\Sx, 0,0) -- cycle;

\pgfmathsetmacro{\Sx}{1/3}
\pgfmathsetmacro{\Sy}{3}
\pgfmathsetmacro{\Sz}{1}

\coordinate (SSorigin) at ($(origin)+(\Sx,\Sy,-2*\Sz)$);

\draw[blue, very thick] (SSorigin) -- ++(-\Sx, 0,0)-- ++ (0,-\Sy, 0) -- ++ (\Sx, 0, 0)-- cycle;
\draw[blue, very thick] (SSorigin) -- ++ (0,0, -\Sz) -- ++ (0,-\Sy, 0) -- ++ (0,0, \Sz)--cycle;
\draw[blue, very thick] (SSorigin) -- ++ (-\Sx, 0,0) -- ++ (0,0,-\Sz) -- ++ (\Sx, 0,0) -- cycle;

\draw[blue, ->] ($(SSorigin)-(0,\Sy/5,\Sz)+(0.2,0,0)$) -- ++(-0.2, 0, 0);
\node[right, blue] at ($(SSorigin)-(0,\Sy/5,\Sz)+(0.2,0,0)$) {$B_I$};

\pgfmathsetmacro{\Qx}{3}
\pgfmathsetmacro{\Qy}{3}
\pgfmathsetmacro{\Qz}{3}

\coordinate (Qorigin) at ($(origin)+(2,3,0)$);

\draw[black] (Qorigin) -- ++(-\Qx, 0,0)-- ++ (0,-\Qy, 0) -- ++ (\Qx, 0, 0)-- cycle;
\draw[black]  (Qorigin) -- ++ (0,0, -\Qz) -- ++ (0,-\Qy, 0) -- ++ (0,0, \Qz)--cycle;
\draw[black]  (Qorigin) -- ++ (-\Qx, 0,0) -- ++ (0,0,-\Qz) -- ++ (\Qx, 0,0) -- cycle;

\draw[black, ->] ($(SSorigin)-(-1,2*\Sy/3,\Sz/2)$) --($(TTorigin)-(0,2*\Ty/3,\Tz/2)$);
\draw[black, ->] ($(SSorigin)-(-1,2*\Sy/3,\Sz/2)$) --($(Torigin)-(0,2*\Ty/3,\Tz/2)$);
\draw[black, ->] ($(SSorigin)-(-1,2*\Sy/3,\Sz/2)$) --($(TTTTorigin)-(0,2*\Ty/3,\Tz/2)$);
\draw[black, ->] ($(SSorigin)-(-1,2*\Sy/3,\Sz/2)$) --($(TTTorigin)-(0,2*\Ty/3,\Tz/2)$);
\node [roundnode]  [right] at ($(SSorigin)-(-1,2*\Sy/3,\Sz/2)$) {$N$};

\draw[black, ->] ($(Torigin)-(-0.8,\Ty,\Tz+0.4)$) --($(TTorigin)-(-0.05,\Ty,\Tz/2)$);
\draw[black, ->] ($(Torigin)-(-0.8,\Ty,\Tz+0.4)$) --($(Torigin)-(-0.05,\Ty,\Tz/2)$);
\node[roundnode][right] at ($(Torigin)-(-0.8,\Ty,\Tz+0.4)$) {$N_1$};

\coordinate(ss) at (0.2, 0, 0);
\coordinate (m1) at ($(Qorigin)+(-\Qx, 0, 0) - (ss)$);
\coordinate (m2) at ($(Qorigin)+(-\Qx, 0, -\Qz) - (ss)$);
\draw[gray, <->] (m1)--(m2);
\node[left] at ($(m1)! 0.7 !(m2) $)  {$R$};

\coordinate(ss) at (0.2, 0, 0);
\coordinate (m3) at ($(Qorigin)+(-\Qx, -\Qy, 0) - (ss)$);
\coordinate (m4) at ($(Qorigin)+(-\Qx, 0, 0) - (ss)$);
\draw[gray, <->] (m3)--(m4);
\node[left] at ($(m3)! 0.5 !(m4) $)  {$R$};

\coordinate(sss) at (0.2, 0, 0);
\coordinate (m5) at ($(Sorigin)+(-\Sx, 0, 0) - (sss)$);
\coordinate (m6) at ($(Sorigin)+(-\Sx, 0, -\Sz) - (sss)$);
\draw[gray, <->] (m5)--(m6);
\node[left] at ($(m5)! 0.7 !(m6) $)  {$R^{\frac{5}{6}}$};

\coordinate(ssss) at (0, 0, 0.2);
\coordinate (m9) at ($(SSorigin)+(-\Sx, 0, -\Sz) - (ssss)$);
\coordinate (m10) at ($(SSorigin)+(0,0, -\Sz) - (ssss)$);
\draw[gray, <->] (m9)--(m10);
\node[above] at ($(m9)! 0.5 !(m10) $)  {$R^{\frac{2}{3}}$};


\coordinate (corigin) at ($(origin)+(-1,0,0)$);

\draw[->] (corigin)--++(4.5,0,0);
\node[right] at ($(corigin)+(4.5,0,0)$) {$x$};
\draw[->] (corigin) -- ++ (0,4.5, 0);
\node[left] at ($ (corigin) + (0, 4.5, 0)$) {$z$};
\draw[->] (corigin)--++ (0, 0, -2);
\node[left] at ($ (corigin)+(0,0,-2)$) {$y$};
\end{tikzpicture}
\end{center}
\bigskip

We show two immediate consequences. Both should be compared with the easy estimate
\begin{equation}
\label{iofhyruiygtuyg}
|\Qc_r|\lesssim \frac{|\T|^2}{r^2}
\end{equation}
that holds in the bilinear setting (when each cube is intersected by at least $r$ tubes from two transverse families with cardinality $|\T|$).

We will only use the second corollary throughout the remainder of the paper. The first corollary is similar to Corollary 5.5 of \cite{DGH}. Since we believe it might be useful in future work, we record it here alongside a short argument.
\begin{corollary}[Refined Kakeya estimates for high incidences]
Suppose $\T$ satisfies the well-spacing conditions (S1) and (S2) in the previous theorem. 

Assume that for some  $\epsilon>0$, we have that  $r\geq C(\epsilon)R^{\epsilon}N_1R^{\frac{1}{6}}$, for some large enough $C(\epsilon)$. Then
$$|\Qc_r|\lessapprox \frac{|\T|N_1 R^{\frac{1}{3}}}{r^2}.$$
\end{corollary}
\begin{proof}
Let $\sigma$ be the one guaranteed by the previous theorem. We first show that $\sigma\ge  R^{-1/6}$.	
 Assume for contradiction that $\sigma< R^{-1/6}$. Combining the lower bound of $r$, the second inequalities in (\ref{eqn:r}) and (\ref{eqn:M}) lead to
$$C(\epsilon)R^{\epsilon}N_1R^{\frac{1}{6}}\leq r\leq (\log R)^{C}R^{\frac{2}{3}}M_{\sigma}\sigma^4 \leq C_2(\log R)^{C_1}N_1R^{\frac{1}{6}}$$ which boils down to $C(\epsilon)R^\epsilon\leq C_2(\log R)^{C_1}$. If we choose $C(\epsilon)$ large enough, the inequality fails for all $R\geq 1$.

Since $\sigma\geq R^{-\frac16}$, we can use the first inequalities in (\ref{eqn:Q}) and (\ref{eqn:M}) to get the desired result.

\end{proof}

Note that, subject to (S2),  the collection $\T$ has size at most $|\T_{max}|=N_1R^{\frac56}$, when each $B_I$ for each $I$ contains $\sim N_1$ tubes. For such collections of tubes, the inequality in the corollary above reads
$$|\Qc_r|\les \frac{|\T|^2}{Wr^2},$$
where $W=R^{1/2}$ is the number of boxes $B_I$ inside $[-R,R]^3$, for a fixed $I$. Note the $W$-gain over the estimate in  \eqref{iofhyruiygtuyg}.

The lower bound on $r$ is necessary for the $W$-gain. Indeed, if $N_1$ tubes are randomly chosen for each box $B_I$, for each interval $I$, we get a family with $|\T|\sim N_1R^{5/6}$. We expect $|\Qc_r|\sim R$ for $r\sim N_1R^{1/6}$, as ``most" cubes $q$ will intersect roughly the same number of tubes in the random case.  Note that $ \frac{|\T|N_1 R^{\frac{1}{3}}}{r^2}\sim R^{5/6}$ is much smaller than $|\Qc_r|$.

\medskip

The next result produces a similar $W$-gain over  \eqref{iofhyruiygtuyg}, this time with the smaller $W=R^{1/3}$, and only assuming the statistics (S1) inside the plates $S_I$. Note that there is no need for a lower bound on $r$ this time.

\begin{corollary}
\label{co:halfwellspace}
Suppose $\T$ satisfies the well-spacing condition (S1) in Theorem \ref{te:inci}. Then for each $r\ge 1$ we have
$$|\Qc_r|\lessapprox \frac{|\T|N R^{\frac{1}{3}}}{r^2}.$$
\end{corollary}
\begin{proof}
We apply the previous Theorem  \ref{te:inci} with $N_1=N$.  The cases $\sigma\ge  R^{-\frac{1}{6}}$ and $\sigma < R^{-\frac{1}{6}}$ will immediately follow by applying the first and second part of  (\ref{eqn:Q}) and (\ref{eqn:M}), respectively.
\end{proof}

We mention that this corollary will be used twice in the remaining part of the paper, namely is steps 4 and 10 of the proof of Proposition \ref{te:Plankinci}.

In order to prove the Theorem \ref{te:inci}, we first observe the simple geometry of the Vinogradov tubes centered at the origin. The following lemma is a variation of Lemma 7.3 in \cite{DGH}. There is a major difference  between small angle and large angle.

\begin{lemma}[Union of tubes]
\label{le:BCT}

Let $R^{-\frac13}\leq\sigma \leq 1$. Let $J\subset [0,1]$ be an interval of length $R^{-\frac13}\sigma^{-1}$. For each $I\in \I_{R^{-1/3}}(J)$, let $T_I$ be a Vinogradov tube associated with $I$ and centered at the origin.
\\
\\
(small angle) If $\sigma\geq R^{-\frac16}$, the box $U_\sigma$ centered at the origin with dimensions $(R^{\frac23},R^{\frac23}\sigma^{-1},R)$ with respect to the axes $(\textbf{t}(J),\textbf{n}(J),\textbf{b}(J))$ contains all the tubes $T_I$.
\\
\\
(large angle) If $\sigma< R^{-\frac16}$, the box $U_\sigma$ centered at the origin with dimensions $(R^{\frac13}\sigma^{-2},R^{\frac23}\sigma^{-1},R)$ and orientation associated with $J$ contains all the tubes $T_I$.
\end{lemma}

\begin{proof}
Let $J=[a,a+R^{-\frac13}\sigma^{-1}]$. Noting that the image $A_{1,a}(T_I)$ is also the Vinogradov tube centered at the origin, it is enough to prove the case when $J=[0,R^{-\frac13}\sigma^{-1}]$. In this case, the axes $(\t(J),\n(J),\b(J))$ of the box $U_\sigma$ is the same as $x,y,z$ axes.

Let $I=[c,c+R^{-\frac13}]$. Each $T_I$ has the dimensions $(R^{\frac23},R^{\frac23},R)$ with respect to axes $\t(c),\n(c),\b(c)$. Recall that the $(x,y,z)$-coordinates of these axes are given by
$$t(c)=(1,2c,3c^2)=(1,O(R^{-\frac13}\sigma^{-1}),O(R^{-\frac23}\sigma^{-2}))$$ $$n(c)=(-2c-9c^3,-c^4+1,3c+6c^3)=(O(R^{-\frac13}\sigma^{-1}),O(1),O(R^{-\frac13}\sigma^{-1}))$$ and $$b(c)=(3c^2,-3c,1)=(O(R^{-\frac23}\sigma^{-2}),O(R^{-\frac13}\sigma^{-1}),1).$$
Therefore, each point inside $T_I$ has $(x,y,z)$-coordinates of the form
$$(O(R^{\frac23}),O(R^{\frac13}\sigma^{-1}),O(\sigma^{-2}))+(O(R^{\frac13}\sigma^{-1}),O(R^{\frac23}),O(R^{\frac13}\sigma^{-1}))+(O(R^{\frac13}\sigma^{-2}),O(R^{\frac23}\sigma^{-1}),O(R)).$$
For the small angle case $\sigma\geq R^{-\frac16}$, the $x$ coordinate is bounded by the first term $O(R^{\frac23})$. For the large angle case $\sigma< R^{-\frac16}$, the $x$ coordinate is bounded by the last term $O(R^{\frac13}\sigma^{-2})$. This shows the required dimensions of the box $U_\sigma$.
\end{proof}

We prove Theorem \ref{te:inci} using Plancherel's identity and plank partitioning. The argument is inspired by that from \cite{DGH}.

\begin{proof}(of Theorem \ref{te:inci}) Let us first prove the $|\Qc_r|$ estimate (\ref{eqn:Q}). Let $v_T$ be a positive smooth approximation of $\textbf{1}_T$ for each $T\in\T$ so that its Fourier transform $\widehat{v_T}$ is supported in the dual box $\overline{\theta}_I$ centered at the origin. We define the function
$$g(x)=\sum_{T\in\T} v_T(x) $$
whose Fourier transform is supported in $\cup_{I\in \I_{R^{-1/3}}}\overline{\theta}_I$. We denote such union of planks by $\overline{\Omega}$. Similarly, for $s\in [0,1]$ and dyadic $R^{-\frac{1}{3}}\leq\sigma\leq1$, let us denote the truncation of $\Omega_{\leq\sigma}$, $\Omega_{\sigma}$,  $\Theta(\sigma,s)$ to the infinite strip $\{w\in \R^3:|w_1|\leq R^{-\frac23}\}$ by $\overline{\Omega}_{\leq\sigma}$, $\overline{\Omega}_{\sigma}$,  $\overline{\Theta}(\sigma,s)$ respectively.
Let $\{\eta_{\sigma}\}$ be a smooth partition of unity of $\overline{\Omega}$, so that each $\eta_{\sigma}$ is supported in $\overline{\Omega}_\sigma$ and satisfies $\sum_{\sigma}\eta_{\sigma}\equiv1$
on $\overline{\Omega}$. Since $\widehat{g}$ is supported in $\overline{\Omega}$, $\widehat{g}=\widehat{g}\sum_{\sigma}\eta_{\sigma}$ and we decompose $g$ by the dyadic scales $\sigma$
$$g=\sum_{R^{-1/3}\leq\sigma\leq 1}g*\check{\eta_{\sigma}}.$$

Let $\overline{\textbf{CP}}_\sigma$ be the set of $\sim R^{\frac13}\sigma$ truncated small planks $\overline{\Theta}(\sigma,s')$ with $s'$ evenly spaced in $[0,1]$. Recall that $\cup_{\overline{\Theta}\in \overline{\textbf{CP}}_\sigma}\overline{\Theta}$ essentially covers $\overline{\Omega}_{\leq\sigma}$. For each $J=[s',s'+R^{-\frac13}\sigma^{-1}]$, let us tile $[-R,R]^3$ by the dual boxes of $\overline{\Theta}(\sigma,s')$. These boxes are the translates of $U_\sigma$ with the axes associated with $J$ as in Lemma \ref{le:BCT}. For fixed scale $\sigma$, we see that each $T$ is contained in $O(1)$ many boxes $U_\sigma$.

Let $m$ be a dyadic parameter, $1\leq m\leq \sigma^{-2}$ if $\sigma\geq R^{-\frac{1}{6}}$ and $1\leq m\leq \sigma^{-4}R^{-\frac{1}{3}}$ if $\sigma < R^{-\frac{1}{6}}$. We denote by $\textbf{U}_{\sigma,m}$ the family of $U_\sigma$ containing $\sim m$ tubes $T\in \T$. Let $\T_{\sigma,m}$ be the set of tubes $T$ contained in one of $U_\sigma\in\textbf{U}_{\sigma,m}$. We define the function $g_{\sigma,m}$ by
$$g_{\sigma,m}=\sum_{I\in\I_{R^{-1/3}}}\sum_{T\in\T_{\sigma,m}(I)}v_T*\check{\eta_\sigma}$$
and decompose $g$ by
$$g=\sum_\sigma g*\check{\eta_\sigma}=\sum_\sigma \sum_m g_{\sigma,m}.$$

There are $\lesssim(\log R)^2$ choices for the dyadic parameters $\sigma$ and $m$. We pick $\sigma$ and $m$, denote such $m$ by $M_\sigma$ where
$$g(x)\lessapprox |g_{\sigma,M_{\sigma}}(x)|$$ for all $x$ in a subset $E$ of $\cup_{q\in \Qc_r}q$ with comparable volume $R^2|\Qc_r|\lessapprox |E|$. We apply Chebyshev's inequality and Plancheral's identity to $g$
\begin{align*}
r^2|\Qc_r|R^2&\lessapprox \int_{E}|g|^2\les \int_{\R^3}|g_{\sigma,M_{\sigma}}(x)|^2dx\\
&=\int_{\R^3}|\sum_{I\in\I_{R^{-1/3}}}\sum_{T\in\T_{\sigma,M_{\sigma}}(I)}v_T*\check{\eta_\sigma}(x)|^2dx\\
&=\int_{\R^3}|\sum_{I\in\I_{R^{-1/3}}}\sum_{T\in\T_{\sigma,M_{\sigma}}(I)}\widehat{v_T}\eta_\sigma(w)|^2dw.
\end{align*}
For each $I\in\I_{R^{-1/3}}$, we define the function $g_{\overline{\theta}_I}$ by
$$g_{\overline{\theta}_I}:=\sum_{T\in\T_{\sigma,M_\sigma}(I)}v_T*\check{\eta_\sigma}$$
where $\widehat{g_{\overline{\theta}}}$ is supported in $\overline{\theta}$. Noting that $g_{\sigma,M_{\sigma}}=\sum_I g_{\overline{\theta}_I}$, we simplify the last equation by
\begin{align*}
\int|\sum_{I\in \I_{R^{-1/3}}}\sum_{T\in \T_{\sigma,M_\sigma}(I)}  \widehat{v_T}\eta_\sigma(w)|^2dw
&=\int |\sum_{\overline{\theta}}\widehat{g_{\overline{\theta}}}(w)|^2dw.
\end{align*}

In general, the functions $\widehat{g_{\overline{\theta}}}$ are neither positive nor pairwise disjoint, so we manage the integration carefully.
We associate each $\overline{\Theta}(\sigma,s')$ with $\overline{\theta}_c$ with $|c-s'|\leq R^{-\frac13}\sigma^{-1}$ for evenly $\sim
R^{-\frac13}\sigma^{-1}$ spaced $s'$. Moreover, we demand each $\overline{\theta}_c$ to be associated with only one $\overline{\Theta}(\sigma,s')$.
For each $\overline{\Theta}(\sigma,s')$, let $\textbf{S}_{\overline{\Theta}}$ be the set of $\overline{\theta}_c$ associated with $\overline{\Theta}=\overline{\Theta}(\sigma,s')$. Thus we can partition the family $\{\overline{\theta}\}$ by the subfamilies $\textbf{S}_{\overline{\Theta}}$. Using this observation, we can rewrite the last equation by
\begin{align*}
\int |\sum_{\overline{\theta}}\widehat{g_{\overline{\theta}}}(w)|^2dw
&=\int |\sum_{\overline{\Theta}\in \overline{\textbf{CP}}_\sigma}\sum_{\overline{\theta}\in\textbf{S}_{\overline{\Theta}}}\widehat{g_{\overline{\theta}}}(w)|^2dw.
\end{align*}
Recall that the integrand has nonzero value only if $w\in \overline{\Omega}_\sigma$. We apply the triangle inequality, Lemma \ref{le:parti1} and Lemma \ref{le:parti2}
\begin{align*}
\int |\sum_{\overline{\Theta}\in \overline{\textbf{CP}}_\sigma}\sum_{\overline{\theta}\in\textbf{S}_{\overline{\Theta}}}\widehat{g_{\overline{\theta}}}(w)|^2dw
&\leq \int (\sum_{\substack{\overline{\Theta}\in \overline{\textbf{CP}}_\sigma\\}}|\sum_{\overline{\theta}\in\textbf{S}_{\overline{\Theta}}}\widehat{g_{\overline{\theta}}}(w)|)^2dw\\
&\lesssim \int (\sum_{\substack{\overline{\Theta}\in \overline{\textbf{CP}}_\sigma\\w\in49\overline{\Theta}}}|\sum_{\overline{\theta}\in\textbf{S}_{\overline{\Theta}}}\widehat{g_{\overline{\theta}}}(w)|)^2dw\\
&\lesssim \int \sum_{\substack{\overline{\Theta}\in \overline{\textbf{CP}}_\sigma\\w\in49\overline{\Theta}}}|\sum_{\overline{\theta}\in\textbf{S}_{\overline{\Theta}}}\widehat{g_{\overline{\theta}}}(w)|^2dw\\
&\lesssim \int \sum_{\substack{\overline{\Theta}\in \overline{\textbf{CP}}_\sigma\\}}|\sum_{\overline{\theta}\in\textbf{S}_{\overline{\Theta}}}\eta_{49\overline{\Theta}}\widehat{g_{\overline{\theta}}}(w)|^2dw
\end{align*}
where $\eta_{49\overline{\Theta}}$ is a smooth bump essentially supported in $49\overline{\Theta}$. We use Plancheral's identity again
\begin{align*}
\int \sum_{\substack{\overline{\Theta}\in \overline{\textbf{CP}}_\sigma\\}}|\sum_{\overline{\theta}\in\textbf{S}_{\overline{\Theta}}}\eta_{49\overline{\Theta}}\widehat{g_{\overline{\theta}}}(w)|^2dw
= \sum_{\substack{\overline{\Theta}\in \overline{\textbf{CP}}_\sigma\\}}\int |\sum_{\overline{\theta}_I\in\textbf{S}_{\overline{\Theta}}}\sum_{T\in\T_{\sigma,M_\sigma}(I)}(v_T*\check{\eta_\sigma})*\check{\eta_{49\overline{\Theta}}}|^2dx.
\end{align*}
Note that each summand $(v_T*\check{\eta_\sigma})*\check{\eta_{49\overline{\Theta}}}$ is essentially supported on $U\in \textbf{U}_{\sigma,M_\sigma}$ such that $T\subset U$ and $U//\overline{\Theta}^*$. We define the function $g_U$ essentially supported on $U\in\textbf{U}_{\sigma,M_\sigma}$ by $$g_{U}=\sum_{\substack{T\in\T_{\sigma,M_\sigma}\\T\subset U}} (v_T*\check{\eta_\sigma})*\check{\eta_{49\overline{\Theta}}}.$$
Each $g_U$ can be estimated by
\begin{equation}\label{eqn:g}
\|g_U\|_{\infty}\lesssim \begin{cases}
M_\sigma\sigma &\text{if }\sigma\geq R^{-\frac{1}{6}}\\
M_\sigma\sigma^3R^{\frac{1}{3}} &\text{if }\sigma < R^{-\frac{1}{6}}.
\end{cases}
\end{equation}
We achieve the $|\Qc_r|$ estimate using the spatial (almost) orthogonality
\begin{align*}
\sum_{\substack{\overline{\Theta}\in \overline{\textbf{CP}}_\sigma\\}}\int |\sum_{\overline{\theta}_I\in\textbf{S}_{\overline{\Theta}}}\sum_{T\in\T_{\sigma,M_\sigma}(I)}(v_T*\check{\eta_\sigma})*\check{\eta_{49\overline{\Theta}}}|^2dx
&\lesssim \sum_{\substack{\overline{\Theta}\in \overline{\textbf{CP}}_\sigma\\}} \sum_{U//\overline{\Theta}^*}\int |g_U|^2\\
&\lesssim \sum_{U\in\textbf{U}_{\sigma,M_\sigma}}\int |g_U|^2\\
&\lesssim \sum_{U\in\textbf{U}_{\sigma,M_{\sigma}}}|U|\|g_U\|_{\infty}^2\\
&\lesssim \begin{cases}
|\T|M_{\sigma}\sigma R^{7/3} &\text{if } \sigma\geq R^{-\frac{1}{6}}\\
|\T|M_{\sigma}\sigma^3 R^{8/3} &\text{if } \sigma< R^{-\frac{1}{6}}.
\end{cases}
\end{align*}
Next, let us prove the $r$ estimate (\ref{eqn:r}). For any $x\in E$, note that there are $\lesssim R^{1/3}\sigma $ boxes $U$ passing through $x$. Using this observation and the $\|g_{U}\|_{\infty}$ estimate (\ref{eqn:g}), we have
\begin{align*}
r\lesssim g(x)\lessapprox|g_{\sigma,M_{\sigma}}|(x)&\leq\sum_{U\in\textbf{U}_{\sigma,M_{\sigma}}}|g_U(x)|\\
&\lessapprox R^{1/3}\sigma\|g_U\|_{\infty}\\
&\lesssim \begin{cases} R^{\frac{1}{3}}M_{\sigma}\sigma^2 &\text{if } \sigma\geq R^{-\frac{1}{6}}\\
R^{\frac{2}{3}}M_{\sigma}\sigma^4 &\text{if } \sigma < R^{-\frac{1}{6}}.
\end{cases}
\end{align*}

Finally, we prove the $M_{\sigma}$ estimate (\ref{eqn:M}). We separate the proof into two cases. Let us pick $U\in\textbf{U}_{\sigma,M_{\sigma}}$ with the orientation associated with some $J$ of length $R^{-\frac13}\sigma^{-1}$.

We start with the case $\sigma\geq R^{-\frac{1}{6}}$. By the spacing condition (S2), the box $U$ contains at most $\min(N_1,\sigma^{-1})$ tubes $T_I\in\T(I)$ from each direction $I\subset J$. There are $\lesssim \sigma^{-1}$ contributing directions by Lemma \ref{le:BCT}. Therefore, $U$ contains at most $\sigma^{-1}\min(N_1,\sigma^{-1})$ tubes $T\in\T$.

Next, we assume the case $\sigma < R^{-\frac{1}{6}}$. For $I\subset J$, the box $U$ intersects $\lesssim \frac{R^{\frac{1}{3}}\sigma^{-2}}{R^{\frac{2}{3}}}$ plates $S_I$. Each intersection $U\cap S_I$ contains at most $\frac{R^{\frac{2}{3}}\sigma^{-1}}{R^{\frac{5}{6}}}$ boxes $B_I$. Thus the spacing condition (S2) implies $U\cap S_I$ contains at most
$N_1\frac{R^{\frac{2}{3}}\sigma^{-1}}{R^{\frac{5}{6}}}$ tubes $T_I$. Also, (S1) implies $U\cap S_I$ contains at most $N$ tubes $T_I$. There are $\lesssim \sigma^{-1}$ contributing directions by Lemma \ref{le:BCT}. We conclude that $U$ contains at most $\sigma^{-1} \frac{R^{\frac{1}{3}}\sigma^{-2}}{R^{\frac{2}{3}}} \min(N_1\frac{R^{\frac{2}{3}}\sigma^{-1}}{R^{\frac{5}{6}}},N)$ tubes $T\in \T$. Combining the two cases we have the following estimate
$$M_{\sigma}\lesssim \begin{cases}
    \sigma^{-1}\min(N_1,\sigma^{-1}) &\text{if } \sigma\geq R^{-\frac{1}{6}}\\
    \sigma^{-3}R^{-\frac{1}{3}}\min(N_1\sigma^{-1}R^{-\frac{1}{6}},N) &\text{if } \sigma < R^{-\frac{1}{6}}.
    \end{cases}$$
This finishes the proof.
\end{proof}

\section{An $L^4$ inequality for Vinogradov planks}
In this section we introduce another Kakeya type estimate for the Vinogradov planks. It represents one of the main innovations of this paper.

Let us first prove the following ``base case",  when each Vinogradov plank $P_I$ is centered at the origin.
\begin{te}
\label{te:ell4}
Let $\delta<1$.  Let $\mathbb{I}_{\delta}$ be the partition of the of $[0,1]$ into intervals $I$ of length $\delta$. Let $P_I$ be the  Vinogradov plank of dimension $(\delta^{-1},\delta^{-2},\delta^{-3})$ associated with $I$ and centered at the origin.
We have
\begin{equation}
\label{jhdhdhuyh}
\|\sum_{I\in \mathbb{I_\delta}}1_{P_I}\|_{4}^4\lesssim (\log \delta^{-1})^2 \|\sum_{I\in \mathbb{I_\delta}}1_{P_I}\|_{1}.
\end{equation}
\end{te}

\begin{proof} Let us write $I_j=[c_j,c_j+\delta]\in \I_{\delta}$ and $D_j=c_j-c_1$ for $j=1,2,3,4$.
A simple observation shows that
\begin{align*}
\|\sum\limits_{I\in \mathbb{I_\delta}}1_{P_I}\|_{4}^4 &=\int \sum_{I_1\in \I_{\delta}}\sum_{I_2\in \I_{\delta}}\sum_{I_3\in \I_{\delta}}\sum_{I_4\in \I_{\delta}}1_{P_{I_1}}1_{P_{I_2}}1_{P_{I_3}}1_{P_{I_4}}\\
&\lesssim \sum\limits_{I_1\in \mathbb{I_\delta}} \sum\limits_{I_2 : |D_2|<1} \sum\limits _{I_3:|D_3|\le |D_2|} \sum\limits_{I_4:|D_4|\le |D_3|} |P_{I_1}\cap P_{I_2}\cap P_{I_3}\cap P_{I_4}|.
\end{align*}

We first estimate the volume $|P_{I_1}\cap P_{I_2}\cap P_{I_3}\cap P_{I_4}|$.
Recall the linear map $A_{1,c}$ on $\mathbb{R}^3$ introduced earlier.
If $J=[c,c+\delta]$, then $A_{1,c}(P_J)$ is a Vinogradov plank centered at the origin, associated with the interval $J-c$. Noting that $\det A_{1,c}=1$, we can assume $I_1=[0,\delta]$, $I_2=[D_2,D_2+\delta]$, $I_3=[D_3,D_3+\delta]$ and $I_4=[D_4,D_4+\delta]$. Also, it is harmless to assume that $D_j\geq 0$ for $j=2,3,4$.

In this setting, $P_{I_1}\cap P_{I_2}$ is an almost rectangular  box of dimensions roughly  $(\delta^{-1},\frac{\delta^{-1}}{D_2},\frac{\delta^{-2}}{D_2})$ centered at the origin, whose the long side has the direction $\textbf{t}(I_1) \times \textbf{t}(I_2)//(0,\frac{3D_2}{2},1)$. In the same way, $P_{I_1}\cap P_{I_3}$ is a box of dimensions $\sim (\delta^{-1},\frac{\delta^{-1}}{D_3},\frac{\delta^{-2}}{D_3})$ with the long side pointing in the direction $(0,\frac{3D_3}{2},1)$. Because these two boxes are subsets of  $P_{I_1}$, we see that the intersection $(P_{I_1}\cap P_{I_2}) \cap (P_{I_1}\cap P_{I_3})$ is an almost rectangular box with the length of the  edge parallel to $x$-axis equal to $\delta^{-1}$.

Let us evaluate the length of the other sides. We project the intersection to the $yz$-plane. The image of projection is an intersection of two tubes with dimensions $(\frac{\delta^{-1}}{D_2},\frac{\delta^{-2}}{D_2})$, $(\frac{\delta^{-1}}{D_3},\frac{\delta^{-2}}{D_3})$ respectively and the angle difference is $\sim D_2-D_3$. Using planar geometry, the intersection is an almost rectangle with width $\frac{\delta^{-1}}{D_2}$ and length $\sim \frac{\delta^{-1}}{D_3(D_2-D_3)}$ with the long edge pointing in the direction $(\frac{3D_2}{2},1)$. Therefore, $P_{I_1}\cap P_{I_2}\cap P_{I_3}$ is an almost rectangular box of dimensions $(\delta^{-1},\frac{\delta^{-1}}{D_2},\frac{\delta^{-1}}{D_3(D_2-D_3)})$.

We estimate $|P_{I_1}\cap P_{I_2}\cap P_{I_3}\cap P_{I_4}|\leq |P_{I_1}\cap P_{I_2}\cap P_{I_3}|$ so that
$$|P_{I_1}\cap P_{I_2}\cap P_{I_3}\cap P_{I_4}| \lesssim  \delta^{-1}\times \frac{\delta^{-1}}{D_2+\delta}\times\frac{\delta^{-1}}{(D_3+\delta) (D_2-D_3+\delta)} .$$

Using the observation we write
\begin{align*}
\|\sum\limits_{I\in \mathbb{I_\delta}}1_{P_I}\|_{4}^4 &\lesssim \sum\limits_{I_1\in \mathbb{I_\delta}} \sum\limits_{I_2 : D_2<1} \sum\limits _{I_3:D_3\le D_2} \sum\limits_{I_4:D_4\le D_3} |P_{I_1}\cap P_{I_2}\cap P_{I_3}\cap P_{I_4}| \\
&\lesssim \sum\limits_{I_1\in \mathbb{I_\delta}} \sum\limits_{I_2 : D_2<1} \sum\limits _{I_3:D_3\le D_2} \sum\limits_{I_4:D_4\le D_3} \delta^{-1}\times\frac{\delta^{-1}}{D_2+\delta}\times\frac{\delta^{-1}}{(D_3+\delta)(D_2-D_3+\delta)}\\
&\lesssim \sum\limits_{I_1\in \mathbb{I_\delta}} \sum\limits_{I_2 : D_2<1} \sum\limits _{I_3:D_3\le D_2} \delta^{-1}(D_3+\delta)\times \delta^{-1}\times\frac{\delta^{-1}}{D_2+\delta}\times\frac{\delta^{-1}}{(D_3+\delta)(D_2-D_3+\delta)}\\
&=\sum\limits_{I_1\in \mathbb{I_\delta}} \sum\limits_{I_2 : D_2<1} \sum\limits _{I_3:D_3\le D_2} \frac{\delta^{-5}}{(D_2+\delta)(1+\delta^{-1}(D_2-D_3))}\\
&\lesssim \sum\limits_{I_1\in \mathbb{I_\delta}} \sum\limits_{I_2 : D_2<1} \frac{\delta^{-5}\log (1+\delta^{-1} D_2)}{D_2+\delta}\\
&\lesssim \sum\limits_{I_1\in \mathbb{I_\delta}} (\log \delta^{-1})^2 \delta^{-6}\\
&\lesssim  (\log \delta^{-1})^2\|\sum\limits_{I\in \mathbb{I_\delta}}1_{P_I}\|_{1}.
\end{align*}

\end{proof}
H\"older's inequality shows that  \eqref{jhdhdhuyh} holds true with the exponent 4 replaced with any $1\le p\le 4$. The restriction theorem for $\Gamma_3$ combined with the standard randomization argument proves the case $p=\frac72$.  The example with one plank for each direction centered at the origin shows that $p$ cannot be taken larger than 4.

There are  variants of the previous theorem for planks  with different spacing conditions. Let us show one example in the following corollary. We write $T_I$ for a Vinogradov tube of dimensions $(\delta^{-2},\delta^{-2},\delta^{-3})$ with the long edge in the direction $\textbf{b}(I)$. Also $\P(I)$ and $\T(I)$ will stand for the subsets of $\P$ and $\T$ containing all planks $P_I\in\P$ and all tubes $T_I\in\T$ associated with $I\in\I_\delta$, respectively. Note that each $T_I$ can contain at most $\delta^{-1}$  Vinogradov planks $P_I$.

\begin{corollary}
\label{Cor5.2}	
	Let $\P$ be a collection  of separated Vinogradov planks satisfying  the following two spacing conditions. For each $I\in \I_{\delta}$, there are at most $M$  planks $P_I\in\mathbb{P}(I)$. Moreover, there are at most $N$ parallel planks $P_I$ inside each $T_I$.

Then we have $$\|\sum_{P\in\P}1_{P_I}\|_{4}^4\lesssim (\log \delta^{-1})^2 MN \|\sum_{P\in\P}1_{P_I}\|_1.$$
\end{corollary}

\begin{proof} A simple observation shows that
\begin{align*}
\|\sum_{P\in \P}1_{P_I}\|_{4}^4 &\lesssim \sum_{P_{I_1}\in\P} \sum\limits_{I_2 : D_2<1} \sum_{I_3:D_3\leq D_2} \sum\limits_{I_4:D_4 \leq D_3}\sum\limits_{P_{I_2}\in\mathbb{P}({I_2})}\sum\limits_{P_{I_3}\in\mathbb{P}({I_3})}\sum\limits_{P_{I_4}\in\mathbb{P}({I_4})} |P_{I_1}\cap P_{I_2}\cap P_{I_3}\cap P_{I_4}|.
\end{align*}

The last sum can be estimated by
\begin{align*}
\sum_{P_{I_4}\in\mathbb{P}({I_4})} |P_{I_1}\cap P_{I_2}\cap P_{I_3}\cap P_{I_4}|&=|P_{I_1}\cap P_{I_2}\cap P_{I_3}\cap (\cup_{P_{I_4}\in\P(I_4)} P_{I_4})|\\
&\leq |P_{I_1}\cap P_{I_2}\cap P_{I_3}|.
\end{align*}

For each $I_1$, $I_2$, $I_3$, we can assume that $P_{I_1}\cap P_{I_2}\cap P_{I_3}$ is nonempty, otherwise the summand is $0$. In the case of $P_{I_1}\cap P_{I_2} \cap P_{I_3}$ nonempty, we can bound the volume $|P_{I_1}\cap P_{I_2} \cap P_{I_3}|$ by $\delta^{-1}\times \frac{\delta^{-1}}{D_2}\times\frac{\delta^{-1}}{D_3 (D_2-D_3)}$ as in the Theorem \ref{te:ell4}. Also, for fixed $I_3$ and nonempty $P_{I_1}\cap P_{I_2}$, there exists a unique $T_{I_3}$ such that $(P_{I_1}\cap P_{I_2}) \subset T_{I_3}$. Combining these facts with the last inequality, the last two sums become

\begin{align*}
\sum\limits_{P_{I_3}\in\mathbb{P}({I_3})}\sum\limits_{P_{I_4}\in\mathbb{P}({I_4})} |P_{I_1}\cap P_{I_2}\cap P_{I_3}\cap P_{I_4}| &\leq \sum\limits_{P_{I_3}\subset T_{I_3}}|P_{I_1}\cap P_{I_2}\cap P_{I_3}|\\
&\leq N\delta^{-1}\times \frac{\delta^{-1}}{D_2+\delta}\times\frac{\delta^{-1}}{(D_3+\delta) (D_2-D_3+\delta)}.
\end{align*}

Using this observation we find

\begin{align*}
\|\sum_{P\in \P}1_{P_I}\|_{4}^4 &\lesssim \sum_{P_{I_1}\in\P} \sum\limits_{I_2 : D_2<1} \sum_{I_3:D_3 \leq D_2} \sum\limits_{I_4:D_4 \leq D_3}\sum\limits_{P_{I_2}\in\mathbb{P}({I_2})}N \delta^{-1}\times \frac{\delta^{-1}}{D_2+\delta}\times\frac{\delta^{-1}}{(D_3+\delta)(D_2-D_3+\delta)}\\
&\lesssim \sum_{P_{I_1}\in\P} \sum\limits_{I_2 : D_2<1} \sum_{I_3:D_3\leq D_2} \sum\limits_{I_4:D_4 \leq D_3}MN \delta^{-1}\times \frac{\delta^{-1}}{D_2+\delta}\times\frac{\delta^{-1}}{(D_3+\delta) (D_2-D_3+\delta)}\\
&\lesssim (\log \delta^{-1})^2 MN \|\sum_{P\in\P}1_{P_I}\|_1.
\end{align*}

\end{proof}

\section{Incidence estimates for Vinogradov plates}
\label{Ch:improvedplates}
We use the Kakeya estimate for  planks from the last section to show incidence estimates for the Vinogradov plates.

\begin{de} (Vinogradov plates)\\
\label{ fjrgop cpreopit}	
	For each $I\in \I_{\delta}$, the $(\delta^{-1},\delta^{-2},\delta^{-2})$-plate $S_I$ is a Vinogradov plate associated with $I$ if it has the normal vector $\textbf{t}(I)$.
\end{de}
The orientation of each plate is characterized by its normal vector. Therefore, if $I=[c,c+\delta]$,  each $S_I$ can be described as a translation of the following Vinogradov plate centered at the origin
$$S=\{ x \in \mathbb{R}^3: |x_1+2cx_2+3c^2x_3|\leq \delta^{-1},|x_2+3cx_3|\leq \delta^{-2}, |x_3|\leq \delta^{-2}\} . $$
We first state the standard  $L^2$ estimate for  plate incidences, whose proof appears in Proposition 6.4 in \cite{DGH}.
\begin{te}[$L^2$ Kakeya for plates]
\label{te:L2plate}
    Let $\S$ be a collection of separated $(\delta^{-1},\delta^{-2},\delta^{-2})$-Vinogradov plates in $[-\delta^{-2},\delta^{-2}]^3$. Assume that there are at most $N$ plates $S_I$ for each $I$.

    Let $\Qc_r$ be a collection of pairwise disjoint cubes $q$ with side length $\sim \delta^{-1}$ that intersect at least $r$ plates $S\in\S$. Then for each $r\ge 1$ we have $$\|\sum_{S\in\S}1_S\|_2^2\lessapprox N(\sum_{S\in\S}|S|)$$
    and we get the incidence estimate
    $$|\Qc_r|\lessapprox \frac{|\S|N\delta^{-2}}{r^{2}} . $$
\end{te}

While we will use this incidence result in Step 7 of the proof of Proposition \ref{te:Plankinci}, we also state it in order  to serve as a comparison with the next result. There is a better estimate for $|\Qc_r|$ if $r\geq N^{\frac{1}{2}}\delta^{-\frac12}$ in the same setting. This better estimate will be used  twice in the remaining part of the paper, namely is steps 5 and 11 of the proof of Proposition \ref{te:Plankinci}.

\begin{te}[$L^4$ Kakeya for plates]
\label{te:L4plate}
Let $\S$ be a collection of separated $(\delta^{-1},\delta^{-2},\delta^{-2})$-Vinogradov plates in $[-\delta^{-2},\delta^{-2}]^3$. Assume that there are at most $N$ plates $S_I$ for each $I$.

 Then for each $r\ge 1$ we have
    $$|\Qc_r|\lessapprox \frac{|\S|N^2\delta^{-3}}{r^{4}}.$$
\end{te}

\begin{proof}
We extend each  plate $S_I$ in the  $\textbf{b}(I)$ direction to get a Vinogradov plank $P_I$ containing $S_I$. Each $P_I$ has  dimensions $(\delta^{-1},\delta^{-2},\delta^{-3})$ with respect to the axes $(\textbf{t}(I),\textbf{n}(I),\textbf{b}(I))$. Let us denote the collection of Vinogradov planks corresponding to $\S$ by $\P$ . Each $r$-rich cube with respect to $\S$ will also be $r$-rich with respect to $\P$.

Using Corollary \ref{Cor5.2} we write
 \begin{align*}
r^4|\Qc_r|(\delta^{-1})^3\leq\|\sum_{S\in\S}1_S\|_4^4 &\leq\|\sum_{P\in\P}1_P\|_4^4\\
&\lessapprox N^2\|\sum_{P\in\P}1_P\|_1\\
&\lesssim N^2 |\P| \delta^{-6}\\
&=N^2 |\S| \delta^{-6}.
 \end{align*}

\end{proof}

\section{A few Preliminaries}
In this section, we record a few simple geometric facts about  Vinogradov planks and plates centered at the origin, along the lines of Lemma 7.3 and Lemma 6.3 in \cite{DGH}. We also briefly introduce two wave packet decompositions and two decouplings that will be useful in the next section.
\label{Ch:Afew}

\subsection{Intersections and unions of planks and plates}

\begin{lem}[Vinogradov planks]
\label{le:planksimplegeo}
Let $R^{-\frac13} \leq \sigma \leq 1 $.
Let $J\subset[0,1]$ be an interval of length $\sigma$. Assume that for each $I\in\I_{R^{-1/3}}(J)$, the spatial $(R^{\frac13},R^{\frac23},R)$- Vinogradov plank $P_I$ is centered at the origin.\\\\
(Intersection) The intersection of all such planks $P_I$ is essentially an $(R^{\frac13},R^{\frac13}/\sigma,R^{\frac13}/\sigma^2)$-plank with respect to axes $(\textbf{t}(J),\textbf{n}(J),\textbf{b}(J))$, and centered at the origin.\\\\
(Union) The box $B$ centered at the origin with dimensions $(R\sigma^2,R\sigma,R)$ with respect to axes $(\textbf{t}(J),\textbf{n}(J),\textbf{b}(J))$ contains all such planks $P_I$.
\end{lem}
\begin{proof} Let $J=[a,a+\sigma]$. Using the linear map $A_{1,a}$, both parts can be reduced to the case $J=[0,\sigma]$. In this case, the axes $\textbf{t}(J),\textbf{n}(J),\textbf{b}(J)$ are the same as the $x,y,z$ axes.

For the intersection part, we can recall the proof of Lemma \ref{te:ell4}. We can see that the intersection of all planks $P_I$ is an $(R^{\frac13},R^{\frac13}/\sigma,R^{\frac13}/\sigma^2)$-box and considering the eccentricity, the box has the orientation associated with $J$.

For the union part, let $I=[c,c+R^{-\frac13}]$. Each $P_I$ has the dimensions $(R^{\frac13},R^{\frac23},R)$ with respect to axes $\t(c),\n(c),\b(c)$. Recall that the $(x,y,z)$-coordinates of these axes are given by
$$t(c)=(1,2c,3c^2)=(1,O(\sigma),O(\sigma^2))$$ $$n(c)=(-2c-9c^3,-c^4+1,3c+6c^3)=(O(\sigma),O(1),O(\sigma))$$ and $$b(c)=(3c^2,-3c,1)=(O(\sigma^2),O(\sigma),1).$$
Therefore, each point inside $P_I$ has $(x,y,z)$-coordinates of the form
$$O(R^{1/3})(1,O(\sigma),O(\sigma^2))+O(R^{2/3})(O(\sigma),O(1),O(\sigma))+O(R)(O(\sigma^2),O(\sigma),1).$$
This is easily seen to be $(O(R\sigma^2), O(R\sigma),O(R))$, as desired.
\end{proof}

\begin{lem}[Vinogradov plates]
\label{le:platesmallangle}
Let $\delta<1$.\\\\
(Intersection) Let $J\subset [0,1]$ be an interval of length $\delta^{\frac12}$. Assume that for each $I\in\I_{\delta}(J)$, the  $(\delta^{-1},\delta^{-2},\delta^{-2})$-Vinogradov plates $S_I$ is centered at the origin. Then the intersection of all the plate $S_I$ is essentially a $(\delta^{-1},\delta^{-\frac32},\delta^{-2})$-plank centered at the origin, such that the $\delta^{-2}$-edge is pointing in direction $\textbf{b}(J)$ and the  $(\delta^{-\frac32},\delta^{-2})$-wide face has  normal direction $\textbf{t}(J)$.\\\\
(Union) Let $\delta \leq \sigma \leq 1$ and let $J$ be an interval of length $\sigma$. Assume that for each $I\in\I_{\delta}(J)$, the $(\delta^{-1},\delta^{-2},\delta^{-2})$-Vinogradov plate $S_I$ is centered at the origin. Then the fat plate of dimensions $(\delta^{-2}\sigma,\delta^{-2},\delta^{-2})$ with the normal vector $\textbf{t}(J)$ contains all the plates $S_I$.

\end{lem}
\begin{proof}
For the intersection part we can find the proof in Lemma 6.3 of \cite{DGH}. For the union part, we assume that $J=[0,\sigma]$. Let $I=[c,c+\delta]$ and we do the same computation as in the last lemma. The details are left to the reader.
\end{proof}

\subsection{Wave packet decompositions}
\label{WPDs}
Both decompositions in this section are based on Exercise 2.7 in \cite{Book} and Theorem 7.4 in \cite{DGH}. Let $R>1$.
\medskip

Recall that the isotropic $R^{-1}$-neighborhood of the cubic moment curve $\Gamma$ in $\R^3$ is given by
$$\Nc_\Gamma(R^{-1})=\{(t,t^2+s_2,t^3+s_3):t\in [0,1], \; (s_2^2+s_3^2)^\frac12\leq R^{-1}\}.$$
For $J\in \I_{R^{-1/2}}$, we denote by $\Nc_J(R^{-1})$ the intersection of $\Nc_\Gamma(R^{-1})$ and the infinite strip $J\times \R^2$. Then, each set $\Nc_J(R^{-1})$ is an almost rectangular box of dimensions $(R^{-\frac12},R^{-1},R^{-1})$ with the long side pointing in the direction $\t(J)$. Let us consider a partition of $\Nc_\Gamma(R^{-1})$ by the  family  $\Nc_J(R^{-1})$. Then we have the following wave packet decomposition at scale $R^{-\frac12}$. We mention that this decomposition was not needed in \cite{DGH}, since that paper is only concerned with the case when each $F_J$ is one exponential wave.

\begin{te}[Wave packet decomposition at scale $R^{-\frac12}$]
\label{te:WPD1}	
Assume that the function $F$ has  Fourier transform supported in $\Nc_\Gamma(R^{-1})$. There is a decomposition
$$F=\sum_{J\in \I_{R^{-1/2}}}\Pc_{J} F=\sum_{W\in\W(F)}F_W$$ where $\W(F)$ is a collection of separated Vinogradov $(R^{\frac12},R,R)$-plates, such that
\\
\\
(W1)\;\; Each $\widehat{F_W}$ is supported on  $2\Nc_J(R^{-1})$ for some $J\in\I_{R^{-1/2}}$, and $W$ is associated with $J$. We denote by $\W_J(F)$ the corresponding plates, so $\Pc_J F=\sum_{W\in\W_J}F_W$.
\\
\\
(W2)\;\; $F_W$ is spatially concentrated near $W$, in the sense that for each $M\ge 1$
$$|F_W(x,y,z)|\lesssim_M\|F_W\|_\infty\chi_P^M(x,y,z).$$
Moreover, for each $p\ge 1$
$$\|F_W\|_p\sim \|F_W\|_\infty |W|^{1/p}.$$
\\
\\
(W3)\;\; for each $p\ge 2$ and each $\W_1\subset \W_2\subset \W_J(F)$ such that $\|F_W\|_\infty\sim const$ for $W\in\W_1$, we have
$$\|\sum_{W\in\W_1}F_W\|_{L^p(\R^3)}\lesssim \|\sum_{W\in\W_2}F_W\|_{L^p(\R^3)}.$$
\\
\\
(W4)\;\; for each $p\ge 2$ and each $\W_1\subset  \W_J(F)$ such that $\|F_W\|_\infty\sim const$ for $W\in\W_1$, we have
$$\|\sum_{W\in\W_1}F_W\|_{L^p(\R^3)}\sim (\sum_{W\in\W_1}\|F_W\|^p_{L^p(\R^3)})^{1/p}.$$
\end{te}
A somewhat simplified representation of $F_W$ is
$$F_W(x,y,z)\approx  A_W1_W(x,y,z)e((x,y,z)\cdot (a,a^2,a^3))$$
where $a$ is some point in $J$ and  $A_W\in\C$ is a constant with  $|A_W|=\|F_W\|_\infty$.\\

\bigskip
Let us consider another decomposition. Recall that we defined the anisotropic neighborhood $\Gamma(R^{-\frac13})$ of the cubic moment curve $\Gamma$ as follows.
$$\Gamma(R^{-\frac13})=\{w\in \R^3:\;w_1\in [0,1],\;|w_2-w_1^2|\le R^{-\frac23},\; |w_3-3w_1w_2+2w_1^3|\le R^{-1}\}$$
For each of $I\in \I_{R^{-1/3}}$, we write the intersection of $\Gamma(R^{-\frac13})$ and the infinite strip $I\times \R^2$ as $\Gamma_I(R^{-\frac13})$. Then each $\Gamma_{I}(R^{-\frac13})$ is an almost rectangular box with dimensions $(R^{-\frac13},R^{-\frac23},R^{-1})$ with respect to axes $(\t(I),\n(I),\b(I))$. This is essentially a frequency Vinogradov plank that we denote by  $\theta_I$. Let us consider a partition of $\Gamma(R^{-1})$ by these  $\theta_I$. Then we have the following wave packet decomposition at scale $R^{-\frac13}$.

\begin{te}[Wave packet decomposition at scale $R^{-\frac13}$]
\label{te:WPD2}	
Assume that $F$ has  Fourier transform supported in the anisotropic neighborhood $\Gamma(R^{-\frac13})$. There is a decomposition
$$F=\sum_{\theta_I}\Pc_{\theta_I} F=\sum_{P\in\P(F)}F_P$$ where $\P(F)$ is a collection of spatial Vinogradov $(R^{\frac13},R^{\frac23},R)$-planks, such that
\\
\\
(P1)\;\; Each $\widehat{F_P}$ is supported on  $2\theta_I$ for some $I\in\I_{R^{-1/3}}$, and $P$ is associated with $I$. We denote by $\P_I(F)$ the corresponding planks, so $\Pc_I F=\sum_{P\in\P_I}F_P$.
\\
\\
(P2)\;\; $F_P$ is spatially concentrated near $P$, in the sense that for each $M\ge 1$
$$|F_P(x,y,z)|\lesssim_M\|F_P\|_\infty\chi_P^M(x,y,z).$$
Moreover, for each $p\ge 1$
$$\|F_P\|_p\sim \|F_P\|_\infty |P|^{1/p}.$$
\\
\\
(P3)\;\; for each $p\ge 2$ and each $\P_1\subset \P_2\subset \P_I(F)$ such that $\|F_P\|_\infty\sim const$ for $P\in\P_1$, we have
$$\|\sum_{P\in\P_1}F_P\|_{L^p(\R^3)}\lesssim \|\sum_{P\in\P_2}F_P\|_{L^p(\R^3)}.$$
\\
\\
(P4)\;\; for each $p\ge 2$ and each $\P_1\subset  \P_I(F)$ such that $\|F_P\|_\infty\sim const$ for $P\in\P_1$, we have
$$\|\sum_{P\in\P_1}F_P\|_{L^p(\R^3)}\sim (\sum_{P\in\P_1}\|F_P\|^p_{L^p(\R^3)})^{1/p}.$$
\end{te}
A fair enough representation of $F_P$ is
$$F_P(x,y,z)\approx A_P 1_P(x,y,z)e((x,y,z)\cdot (c,c^2,c^3))$$
where $c$ is some (any!) point in $I$ and $A_P\in\C$ is a constant such that  $A_P\sim \|F_P\|_\infty$.
\smallskip

\subsection{Decouplings}
The following refinement of the $l^{12}(L^{12})$-decoupling for the cubic moment curve was proved in Theorem 7.5 of \cite{DGH}.
\begin{te}[Refined decoupling]
\label{te:L12refined}	
Assume the function $F:\R^3\to \C$ has the Fourier transform supported in $\Gamma({R^{-1/3}})$.
	Let $\Qc$ be a collection of pairwise disjoint $R^{1/3}$-cubes $q$ in $\R^3$.
	Assume that each $q$ intersects at most $M$ fat planks $R^\Delta P$ with $P\in\P(F)$, for some $M\ge 1$ and $\Delta>0$.
	
	Then for each $2\le p\le 12$ and $\epsilon>0$ we have	
	
	$$\|F\|_{L^p(\cup_{q\in\Qc}\chi_q)}\lesssim_{\Delta,\epsilon}R^{\epsilon}M^{\frac12-\frac1p}(\sum_{P\in\P(F)}\|F_P\|^p_{L^p(\R^3)})^{\frac1p}.$$
\end{te}	
Next, we also record the following flat decoupling that will be used to prove the  trilinear-linear reduction in the next section. We can refer to Proposition 2.4 of \cite{DGH}.
\begin{te}[Flat decoupling]
\label{te:flatdec}
Let $B$ be a rectangular box in $\R^n$, and let $B_1,B_2,...,B_L$ be a partition of $B$ into congruent boxes that are translates of each other.\\
Then for each $2\leq p\leq \infty$ and each $F\in L^p(\R^n)$ we have
$$\|\Pc_B F\|_{L^p(\R^n)}\lesssim L^{1-\frac2p}(\sum_{i=1}^L\|\Pc_{B_i}F\|_{L^p(\R^n)}^p)^{\frac1p}.$$
\end{te}

\section{Proof of the main theorem}
\label{Sec8}
\bigskip

Let us recall the main theorem.
\begin{te}
\label{te:mainrecall}
Assume that function $F:\R^3 \to \C$ has the Fourier transform supported in $\Nc_\Gamma(R^{-1})$. Then for $2\leq p\leq 10$, we have
$$\|F\|_{L^{p}(\R^3)}\lesssim_\epsilon R^{\frac12(\frac12-\frac{1}{p})+\epsilon}(\sum_{J\in \I_{R^{-1/2}}}\|\Pc_J F\|_{L^{p}(\R^3)}^{p})^{\frac{1}{p}}.$$
\end{te}
The proof will follow in several steps. First, we prove the trilinear-to-linear reduction by the Bourgain-Guth method \cite{BG}. We show that the following trilinear estimate implies the main theorem when $p=10$. Then, applying the decoupling interpolation in Exercise 9.21 of \cite{Book}, we can prove the main theorem for the full range $2\leq p\leq 10$.
\subsection{Trilinear to linear reduction} For each function $F:\R^3\to \C$ let us write $F_1=\Pc_{[0,1/6]}F$, $F_2=\Pc_{[1/3,1/2]}F$, $F_3=\Pc_{[2/3,1]}F$.
\begin{theorem}
\label{rejhuvtopcekw}	
Assume that we have the following estimate for all functions $F:\R^3 \to \C$ with $\widehat{F}$ supported in $\Nc_{R^{-1}}(\Gamma)$,
\begin{equation}
\label{eqn:triglobal}
\|(F_1F_2F_3)^{\frac13}\|_{L^{10}(\R^3)}\lesssim_\epsilon R^{\frac12(\frac12-\frac{1}{10})+\epsilon}(\sum_{J\in \I_{R^{-1/2}}} \|\Pc_J F\|_{L^{10}(\R^3)}^{10})^{\frac{1}{10}}.
\end{equation}
Then, this estimate implies  Theorem \ref{te:mainrecall}.
\end{theorem}

\begin{proof}
Let $K\sim \log R$ and $l$ be the number such that $K^l= R^{\frac13}$. Let $I$ be dyadic interval contained in $[0,1]$. We write $I_1\nsim I_2 \nsim I_3$ if intervals $I_1,I_2,I_3$ are pairwise non-adjacent. We have the following elementary inequality with fixed constant $C=O(1)$
$$|F(x)|\leq C \max_{|I|=K^{-1}}|\Pc_I F(x)|+K^C \max_{\substack{I_1, I_2, I_3 \in \I_{K^{-1}}\\I_1\nsim I_2 \nsim I_3}}|\Pc_{I_1}F \Pc_{I_2}F \Pc_{I_3}F|^{\frac13} . $$
We iterate the first term $l$ times, raise to power $10$ and integrate over $\R^3$
\begin{align*}
    \|F\|_{L^{10}(\R^3)}^{10}&\lesssim C^l\sum_{I\in \I_{R^{-1/3}}}\int_{\R^3}|\Pc_I F(x)|^{10} dx\\
&+ C^l K^C \sum _{\substack{R^{-1/3}\lesssim \Delta \lesssim 1\\ \Delta\in K^{\Z}}} \sum_{I :|I| = \Delta} \max_{\substack{I_1,I_2,I_3 \in \I_{\Delta K^{-1}}(I)\\ I_1\nsim I_2\nsim I_3}} \int_{\R^3}|\Pc_{I_1} F(x) \Pc_{I_2} F(x) \Pc_{I_3} F(x)|^{\frac{10}{3}} dx.
\end{align*}
The first sum with $|I|=R^{-1/3}$ can be estimated by the flat decoupling Theorem \ref{te:flatdec},
$$\|\Pc_{I}F\|_{L^{10}(\R^3)}^{10}\lesssim R^{(\frac12-\frac13)(1-\frac{2}{10})10} \sum _{\substack{J\in \I_{R^{-1/2}}(I)}} \|\Pc_I F\|_{L^{10}(\R^3)}^{10}.$$
Note that $C^l\lesssim_\epsilon R^\epsilon$ and the exponent of $R$ satisfies $(\frac12-\frac13)(1-\frac{2}{10}) < (\frac12)(\frac12-\frac{1}{10})$. Thus the first sum estimate is safe enough.

Next we focus on the second sum.  For each $I=[c,c+\Delta]$, we define the affine transformation $\tilde{T}_{\Delta, c}(w_1,w_2,w_3)=(w_1',w_2',w_3')$ by
\begin{equation*}
\begin{cases}
w_1'=\frac{w_1-c}{\Delta}\\
w_2'=\frac{w_2-2cw_1+c^2}{\Delta^2}\\
w_3'=\frac{w_3-3cw_2+3c^2w_1-c^3}{\Delta^3}.
\end{cases}
\end{equation*}
Let $A_{\Delta,c}(x_1,x_2,x_3)=(x_1',x_2',x_3')$ be the following linear map
\begin{equation*}
\begin{cases}
x_1'=\Delta(x_1+2cx_2+3c^2x_3)\\
x_2'=\Delta^2(x_2+3cx_3)\\
x_3'=\Delta^3x_3.
\end{cases}
\end{equation*}
Let $G$ be the function defined by $\widehat{G}=\widehat{\Pc_I F}\circ \tilde{T}_{\Delta,c}^{-1}$ and let $G_i$ be $\widehat{G_i}=\widehat{\Pc_{I_i} F}\circ \tilde{T}_{\Delta,c}^{-1}$ for $i=1,2,3$. Then $\widehat{G}$ is supported in $\Nc_{\Gamma}(\frac{R^{-1}}{\Delta^3})$ and the supports of $\widehat{G_1}, \widehat{G_2}, \widehat{G_3}$ are pairwise non-adjacent neighborhood of $\Gamma$ with length $\frac{1}{K}$. Note that for $G$ and similarly for $G_1, G_2, G_3$, we have
$$|G(x)|=|\Pc_I F(A_{\Delta,c}^{-1}x)||\det A_{\Delta,c}^{-1}|.$$
Using this fact, we can observe the following
\begin{align*}
\int_{\R^3}|\Pc_{I_1}F \Pc_{I_2}F \Pc_{I_3}F|^{\frac{10}{3}}(x)dx &=\int_{\R^3}|\Pc_{I_1}F \Pc_{I_2}F\Pc_{I_3}F|^{\frac{10}{3}}(A_{\Delta,c}^{-1}x)|\det A_{\Delta,c}^{-1}|dx\\
&=\int_{\R^3} |G_1 G_2 G_3|^{\frac{10}{3}}(x)|\det A_{\Delta,c}^{-1}||\det A_{\Delta,c}|^{10}dx\\
    &\lesssim_\epsilon (R\Delta^3)^{(\frac12)(\frac12-\frac{1}{10})10+\epsilon}\sum_{\tilde{J}\in \I_{(R\Delta^3)^{-1/2}}} \|\Pc_{\tilde{J}} G\|_{L^{10}(\R^3)}^{10}|\det A_{\Delta,c}|^{9}\\
    &= (R\Delta^3)^{(\frac12)(\frac12-\frac{1}{10})10+\epsilon}\sum_{\tilde{J}\in \I_{\frac{\Delta}{(R\Delta^3)^{1/2}}}(I)} \|\Pc_{\tilde{J}} F\|_{L^{10}(\R^3)}^{10}.
\end{align*}
For each interval $\tilde{J}\in \I_{\frac{\Delta}{(R\Delta^3)^{1/2}}}(I)$, we apply the flat decoupling Theorem \ref{te:flatdec} to decouple further to the intervals $J$ of length $R^{-\frac12}$
$$\|\Pc_{\tilde{J}}F\|_{L^{10}(\R^3)}^{10} \lesssim (\Delta^{-\frac12})^{(1-\frac{2}{10})10}\sum_{J\in \I_{R^{-1/2}}(\tilde{J})}\|\Pc_J F\|_{L^{10}(\R^3)}^{10} . $$
We combine the last two inequalities with $C^l K^C\lesssim_\epsilon R^\epsilon$, and conclude the second sum estimate by
\begin{align*}
    R^{\frac12(\frac12-\frac{1}{10})10+\epsilon} \sum _{\substack{R^{-1/3}\lesssim \Delta \lesssim 1\\ \Delta\in K^{\Z}}}\Delta^{(\frac{1}{5})10} \sum_{J\in \I_{R^{-1/2}}}\|\Pc_J F\|_{L^{10}(\R^3)}^{10}\lesssim_\epsilon R^{\frac12(\frac12-\frac{1}{10})10+\epsilon} \sum_{J\in \I_{R^{-1/2}}}\|\Pc_J F\|_{L^{10}(\R^3)}^{10}.
\end{align*}
This finishes the proof of trilinear to linear reduction.
\end{proof}

\medskip
\subsection{Wave packet decomposition}
For a given ball $B_R$ of radius $R$ centered at $c$ we define the weight function on $B$ by
$$w_B(x)=\frac{1}{(1+\frac{|x-c|}{R})^{300}}.
$$
It will suffice to prove the following local version.
\begin{te}
Assume that the function $F:\R^3\to \C$ has the Fourier transform supported in $\Nc_\Gamma(R^{-1})$. Then we have
\begin{equation}
\label{eqn:trilocal}
\|(F_1F_2F_3)^{\frac13}\|_{L^{10}([-R,R]^3)}\lesssim_\epsilon R^{\frac12(\frac12-\frac{1}{10})+\epsilon}(\sum_{J\in \I_{R^{-1/2}}} \|\Pc_J F\|_{L^{10}(w_{B_R})}^{10})^{\frac{1}{10}}.
\end{equation}
\end{te}

\medskip

We focus on proving the local trilinear version (\ref{eqn:trilocal}).
It suffices to assume that $F=F_1+F_2+F_3$.

We will perform two rounds of wave packet decomposition at scales $R^{-\frac12}$ and $R^{-\frac13}$. It is worth recalling that $\Nc_\Gamma(R^{-1})$ is a subset of $\Gamma(R^{-1/3})$, so functions with spectrum in the first set are subject to both decompositions from Section \ref{WPDs}.

Several steps of partitioning and pigeonholing will follow after the wave packet decompositions.
\\\\

\medskip

\textbf{Wave packet decomposition at scale $R^{-1/2}$. The first pigeonholing sequence}

\smallskip

We can write
    $$F=\sum_{J\in\I_{R^{-1/2}}}\Pc_J F.$$

By splitting $F$ into two parts, we may assume that there are no neighboring intervals $J$ in the sum. Then, we decompose $F$ into wave packets at scale $R^{-\frac12}$ as in Theorem \ref{te:WPD1}
$$F=\sum_{W\in \W_J(F)}F_W=\sum_{W\in \W(F)}F_W.$$
The integration domain is $[-R,R]^3$, so we only consider the wave packets associated with plates $W\subset [-R,R]^3$. We can replace the domain $[-R,R]^3$ with $\R^3$ for the rest of the argument, and understand $\W(F)$ as consisting of only the plates $W\subset [-R,R]^3$. Recall that $W\in\W_J(F)$ is a Vinogradov  $(R^{\frac12},R,R)$-plate with  normal vector $\textbf{t}(J)$.

\smallskip

We partition the set $\W(F)$ by using dyadic parameters  $w, n, X, m, l, Y$, into $O((\log R)^C)$ many collections $\W^{(i)}(F)$. We will always use $I$ to denote an element of $\I_{R^{-1/3}}$ and $J$ to denote an element of $\I_{R^{-1/2}}$. Also, we will drop the $F$ dependence and simply write $\W$ for $\W(F)$.
\\

1. Parameter $w$
\\

Let $w$ be a dyadic parameter in the range $[R^{-1000}\max_{W\in \W}\|F_W\|_\infty, \max_{W\in \W}\|F_W\|_\infty]$. We partition the set of plates $\W$ into subcollections $\W_w$ so that within each $\W_w$ we have  $\|F_W\|_\infty\sim w$. There are $O(\log R)$ many such subcollections. We discard all wave packets with weight $w<R^{-1000}\max_{W\in \W}\|F_W\|_\infty$, as they contribute negligibly to the $L^p$ norm of $F$.

Therefore, the function $F$ can be written as a sum of wave packets that we keep and a small error term whose contribution is negligible. We fix the parameter $w$  and apply the next pigeonholing step to $\W_w$. We will not change notation when we move to a new step, so we will continue to call $\W_w$ as $\W$.
\\

2. Parameters $n,X$: definition of heavy $J$
\\

For each $I\in \I_{R^{-1/3}}$, we tile $[-R,R]^3$ with $(R^{\frac23},R,R)$-fat plates $\Pi_I$ with normal vector $\textbf{t}(I)$. Reasoning as in  Lemma \ref{le:platesmallangle}, we can assume that each $W\in \W_J$ with $J\subset I$ is uniquely contained in some $\Pi_I$. Each fat plate $\Pi_I$ contains at most $\sim R^\frac16$ parallel plates $W\in\W_J$, for each $J\subset I$. Also,  there are $\sim R^{\frac13}$  parallel fat plates $\Pi_I$ inside $[-R,R]^3$.

For fixed  dyadic parameters $1\leq n \leq R^{\frac16}$ and $1\leq X \leq R^{\frac13}$, we call an interval $J$ (subinterval of some $I$) {\em heavy}, if there are $\sim X$ boxes $\Pi_I$ each containing $\sim n$ plates $W_J$.

Of course, a given $J$ may be heavy with respect to more than one pair $(X,n)$. But at the end of this step, we fix this pair. In the next step, we only consider the  heavy $J$ with respect to this pair. Also, $\W_J$ will next refer to the $\sim nX$ plates $W_J$ that contribute to it being heavy with respect to this pair. All other plates from $\W$ --both those in $\W_J$ for non-heavy $J$, and those in $\W_J$ for a heavy $J$  but not among the special $\sim nX$ ones-- will be discarded.

It is worth pointing out that the same $J$ may contribute to more than one collection  $\W^{(i)}(F)$, but always with different plates for each collection.
\\

3. Parameter $m$: definition of heavy $I$
\\

We partition $\I_{R^{-1/3}}$ into $O(\log R)$ many collections, where each interval $I$ in the collection contains $\sim m$ heavy $J\subset I$. We fix the dyadic number $1\leq m\leq R^{\frac16}$ and call $\I_{heavy}$ the collection of intervals $I$ corresponding to this $m$. All other intervals $I$ will be discarded.
\medskip

 We record the following lower bound for the main theorem
\begin{align}
\label{ineq:reduction}
\begin{split}
R^{\frac12(\frac12-\frac{1}{10})+\epsilon}(\sum_{J\in \I_{R^{-1/2}}}\|\Pc_J F\|_{L^{10}(\R^3)}^{10})^{\frac{1}{10}}&=R^{\frac12(\frac12-\frac{1}{10})+\epsilon}(\sum_{I\in \I_{R^{-1/3}}}\sum_{J\in \I_{R^{-1/2}}(I)}\|\Pc_J F\|_{L^{10}(\R^3)}^{10})^{\frac{1}{10}}\\
   &\gtrsim R^{\frac12(\frac12-\frac{1}{10})+\epsilon}(|\I_{heavy}|m w^{10} nX |W|)^{\frac{1}{10}}.
\end{split}
\end{align}
Also, recall that the volume of the plate $W$ is $|W|\sim R^{\frac52}.$
\\

4. Parameters $l,Y$: definition of heavy $\Pi_I$
\\

Let $I\in\I_{heavy}$. We will say that a heavy $J\subset I$ contributes to the fat plate $\Pi_I$ if $\Pi_I$ is one of the $\sim X$ boxes that contains $\sim n$ plates $W_J$.

Each fat plate $\Pi_I$ can be contributed by at most $m$ heavy  intervals $J\subset I$. Let
$$1\le l\le m$$ be a dyadic parameter.
We split the family of fat plates $\Pi_I$ into $O(\log R)$ collections according to the number $\sim l$ of heavy intervals $J$ contributing to $\Pi_I$. We fix $l$, and call all corresponding $\Pi_I$ {\em heavy}.

Also, for  a fixed dyadic parameter
\begin{equation*}
1\leq Y \leq R^{\frac16}
\end{equation*}
 we only retain those $I\in \I_{heavy}$, for which there are $\sim Y$ heavy $\Pi_I$. We discard all other $I$, and continue to call the smaller collection $\I_{heavy}$. Also, for each $I$ in this new collection, we only keep those $\sim Y$ heavy boxes $\Pi_I$ that contribute. For each $J$, we only keep those $W_J$ that sit inside one of the selected $\Pi_I$. We note the following simple inequality
 \begin{equation}
 \label{fsddppf-=0pg}
     lY\leq mX.
 \end{equation}
 The pigeonholing is over.
 We denote by $\W^{(i)}$ the collection of these surviving plates.
\medskip

We write  $$F^{(i)}=\sum_{W\in\W^{(i)}}F_W,$$ so  $F$ differers from  $\sum_{i\les 1}F^{(i)}$ by a negligible error term.

Since $F=F_1+F_2+F_3$,  $\W(F)$ splits as a disjoint union of $\W(F_1),\W(F_2),\W(F_3)$. Thus, we also have the partition for $1\le j\le 3$
$$\W(F_j)=\bigcup_i(\W^{(i)}(F)\cap \W(F_j))=\bigcup_i\W^{(i)}(F_j).$$

We can decompose each function $F_1$, $F_2$, $F_3$ as a sum of $O((\log R)^C)$ many restricted functions $F_1^{(i_1)}$, $F_2^{(i_2)}$, $F_3^{(i_3)}$ associated with the families of plates $W\in \W^{(i)}(F_1)$,  $\W^{(i)}(F_2)$, $\W^{(i)}(F_3)$, respectively. It follows that

$$
\|(F_1F_2F_3)^{1/3}\|_{L^p(\R^3)}\lessapprox \sup_{i_1,i_2,i_3}\|(F_1^{(i_1)}F_2^{(i_2)}F_3^{(i_3)})^{1/3}\|_{L^p(\R^3)}.
$$

It remains to estimate each term corresponding to a tuple $(i_1,i_2,i_3)$.
To keep the notation simpler, we will consider the case $i_1=i_2=i_3=i$.

Let us also call $w, n, X, m, l, Y$ the parameters associated with $F^{(i)}$.
Let us call $g$ the restricted function $F^{(i)}$, $\W^{(i)}$ as $\W$ and $\W^{(i)}_J$ as $\W_J$. We also call $g_j=F_j^{(i)}$.

Before we move on to the next wave packet decomposition, let us re-evaluate our goal. Recalling \eqref{eqn:trilocal},  we need to prove that
$$
\|(g_1g_2g_3)^{\frac13}\|_{L^{10}(\R^3)}\lesssim_\epsilon R^{\frac12(\frac12-\frac{1}{10})+\epsilon}(\sum_{J\in \I_{R^{-1/2}}} \|\Pc_J F\|_{L^{10}(\R^3)}^{10})^{\frac{1}{10}}.
$$
In light of \eqref{ineq:reduction}, this will follow if we prove that
\begin{equation}
\label{eqn:pigeonhole1}
\|(g_1g_2g_3)^{\frac13}\|_{L^{10}(\R^3)}\lesssim_\epsilon R^{\frac12(\frac12-\frac{1}{10})+\epsilon}(|\I_{heavy}|m w^{10} nX |W|)^{\frac{1}{10}}.
\end{equation}

\medskip

\textbf{Wave packet decomposition at scale $R^{-1/3}$. The second pigeonholing sequence}

\smallskip

 Note that the function $g$ continues to have  Fourier transform supported on (a slight enlargement of) $\Nc_\Gamma(R^{-1})$, due to (W1) in Theorem \ref{te:WPD1}.

We can  write (recall that the sum is in fact over $I\in\I_{heavy}$)
    $$g=\sum_{I\in\I_{R^{-1/3}}}\Pc_I g.$$

We may assume that there are no neighboring intervals  $I$ in the sum. We decompose $g$ into wave packets at scale $R^{-\frac13}$, as in Theorem \ref{te:WPD2}
$$g=\sum_{P\in \P(g)}g_P.$$
Each $P$ is a $(R^{\frac13},R^{\frac23},R)$-Vinogradov plank in $[-R,R]^3$. We will partition the set of planks $\P(g)$ into $O((\log R)^C)$ many collections $\P^{(j)}(g)$, according to the dyadic parameters $A, N, Z_1, Z_2$.
\\

1. Parameter $A$
\\

 We partition the set of planks $\P(g)$ into $O(\log R)$ many significant collections $\P_A$,  with $\|g_P\|_\infty \sim A$ for all $P\in\P_A$. The dyadic parameter $A$ satisfies $$R^{-1000}\max_{P\in \P(g)}\|g_P\|_\infty\le A\le  \max_{P\in \P(g)}\|g_P\|_{\infty}.$$
We fix such a collection $\P_A$,  and move to the next step.  Due to Schwartz tail considerations, we may assume that each plank $P_I\in \P_A$ (associated with some $I$) lies inside a heavy fat box $\Pi_I$, produced by the previous pigeonholing sequence.
\\

2. Parameter $N$: definition of heavy $\tau$ \\

We tile each heavy $\Pi_I$ by parallel $(R^{\frac12},R^{\frac23},R)$-boxes $\tau$. Each plank $P_I\in\P_A$ is uniquely contained in one of the boxes $\tau$ and each box $\tau$ can contain at most $R^{\frac16}$ planks $P_I$. Let us partition the family of $\tau$ according to the dyadic parameter $1\leq N\leq R^{\frac16}$, so that $\tau$ in each subfamily contains $\sim N$ parallel planks $P_I\in\P_A$.
We fix such $N$ and call the associated boxes $\tau$ {\em heavy}. We denote the family of all heavy $\tau$ as $\Tc$.
\\

3. Parameters $Z_1, Z_2$: definition of heavy $\Sigma$, contributing $\Pi_I$ and contributing $I$\\

We also tile each heavy $\Pi_I$ by parallel $(R^{\frac12},R^{\frac56},R)$-boxes $\Sigma$. Of course, each heavy $\tau$ is uniquely contained in some $\Sigma$. Note that the box $\Sigma$ includes at most $\sim R^\frac16$ many $\tau$. For a dyadic parameter $1\leq Z_1 \leq R^{\frac16}$, we partition the family of $\Sigma$, so that $\Sigma$ in each subfamily contains $\sim Z_1$ heavy $\tau$. Fixing the parameter $Z_1$, the associated  boxes $\Sigma$ will be called {\em heavy}.
\smallskip

Note that each heavy $\Pi_I$ can contain at most $\sim R^{\frac13}$ many boxes $\Sigma$. For the dyadic parameter $1\leq Z_2 \leq R^{\frac13}$, we partition the family of heavy $\Pi_I$ so that $\Pi_I$ in each subfamily contains $\sim Z_2$ heavy $\Sigma$. We call such $\Pi_I$ {\em contributing}. Recall that for each $I\in \I_{heavy}$, there can be $\lesssim Y$ contributing fat plates $\Pi_I$. We call $I$  {\em contributing} if there is any contributing fat plate $\Pi_I$. We write the family of contributing intervals $I$ as $\I_{contr}$. So  $|\I_{contr}| \leq |\I_{heavy}|$.\\


\textbf{Figure 2}

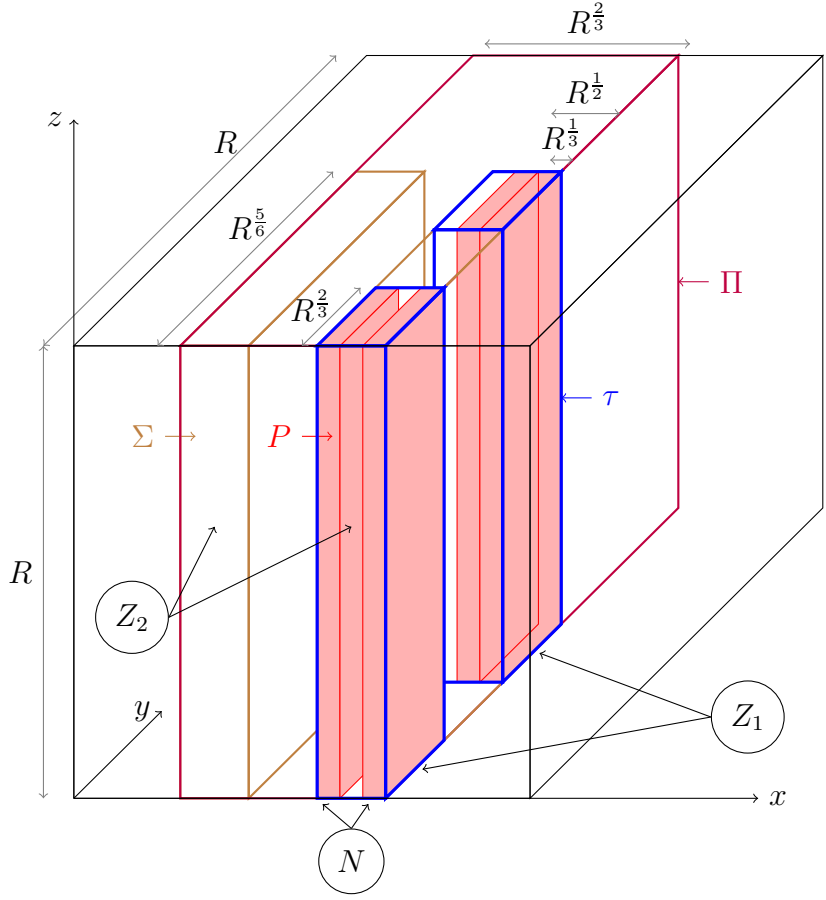
\begin{center}
\begin{tikzpicture}[scale=2,roundnode/.style={circle, draw=black, minimum size=0.1}]

\pgfmathsetmacro{\originonex}{2}
\pgfmathsetmacro{\originoney}{0}
\pgfmathsetmacro{\originonez}{0}

\coordinate (origin) at (-\originonex, -\originoney, 0);

\pgfmathsetmacro{\SSigmax}{0.45}
\pgfmathsetmacro{\SSigmay}{3}
\pgfmathsetmacro{\SSigmaz}{3}
\pgfmathsetmacro{\Pix}{1.35}
\pgfmathsetmacro{\Piy}{3}
\pgfmathsetmacro{\Piz}{5}

\coordinate (SSigmao) at ($(origin)+(\SSigmax,\SSigmay,0)+(-\Pix+\SSigmax,0,0)$); 

\draw[thick, brown] (SSigmao) -- ++(-\SSigmax, 0,0)-- ++ (0,-\SSigmay, 0) -- ++ (\SSigmax, 0, 0)-- cycle;
\draw[thick, brown] (SSigmao) -- ++ (0,0, -\SSigmaz) -- ++ (0,-\SSigmay, 0) -- ++ (0,0, \SSigmaz)--cycle;
\draw[thick, brown]  (SSigmao) -- ++ (-\SSigmax, 0,0) -- ++ (0,0,-\SSigmaz) -- ++ (\SSigmax, 0,0) -- cycle;

\draw[brown , ->] ($(SSigmao)-(\SSigmax,\SSigmay/5,0)-(0.1,0,0)$) --++(0.2,0,0);
\node[left, brown] at ($(SSigmao)-(\SSigmax,\SSigmay/5,0)-(0.1,0,0)$) {$\Sigma$};

\pgfmathsetmacro{\Pix}{1.35}
\pgfmathsetmacro{\Piy}{3}
\pgfmathsetmacro{\Piz}{5}
\pgfmathsetmacro{\Sigmax}{0.45}

\coordinate (Pio) at ($(origin)+(\Sigmax,\Piy,0)$); 

\draw[thick, purple] (Pio) -- ++(-\Pix, 0,0)-- ++ (0,-\Piy, 0) -- ++ (\Pix, 0, 0)-- cycle;
\draw[thick,  purple] (Pio) -- ++ (0,0, -\Piz) -- ++ (0,-\Piy, 0) -- ++ (0,0, \Piz)--cycle;
\draw[thick,  purple]  (Pio) -- ++ (-\Pix, 0,0) -- ++ (0,0,-\Piz) -- ++ (\Pix, 0,0) -- cycle;

\draw[purple , ->] ($(Pio)-(0,\Piy/2,\Piz)+(0.2,0,0)$) --++(-0.2,0,0);
\node[right, purple] at ($(Pio)-(0,\Piy/2,\Piz)+(0.2,0,0)$) {$\Pi$};

\pgfmathsetmacro{\Sx}{0.15}
\pgfmathsetmacro{\Sy}{3}
\pgfmathsetmacro{\Sz}{1}

\coordinate (Sorigin) at ($(origin)+(\Sx,\Sy,0)$);

\draw[red,fill=red!30] (Sorigin) -- ++(-\Sx, 0,0)-- ++ (0,-\Sy, 0) -- ++ (\Sx, 0, 0)-- cycle;
\draw[red,fill=red!30] (Sorigin) -- ++ (0,0, -\Sz) -- ++ (0,-\Sy, 0) -- ++ (0,0, \Sz)--cycle;
\draw[red,fill=red!30] (Sorigin) -- ++ (-\Sx, 0,0) -- ++ (0,0,-\Sz) -- ++ (\Sx, 0,0) -- cycle;

\draw[red, ->] ($(Sorigin)-(\Sx+0.1,\Sy/5,0)$) -- ++(0.2, 0, 0);
\node[left, red] at ($(Sorigin)-(\Sx+0.1,\Sy/5,0)$) {$P$};

\coordinate (SSSorigin) at ($(origin)+(\Sx,\Sy,0)+(2*\Sx,0,-2*\Sz)$);
\pgfmathsetmacro{\Qx}{3}
\pgfmathsetmacro{\Qy}{3}
\pgfmathsetmacro{\Qz}{3}

\draw[red,fill=red!30] (SSSorigin) -- ++(-\Sx, 0,0)-- ++ (0,-\Sy, 0) -- ++ (\Sx, 0, 0)-- cycle;
\draw[red,fill=red!30] (SSSorigin) -- ++ (0,0, -\Sz) -- ++ (0,-\Sy, 0) -- ++ (0,0, \Sz)--cycle;
\draw[red,fill=red!30] (SSSorigin) -- ++ (-\Sx, 0,0) -- ++ (0,0,-\Sz) -- ++ (\Sx, 0,0) -- cycle;

\coordinate (SSSSorigin) at ($(origin)+(\Sx,\Sy,0)+(\Sx,0,-2*\Sz)$);

\draw[red,fill=red!30] (SSSSorigin) -- ++(-\Sx, 0,0)-- ++ (0,-\Sy, 0) -- ++ (\Sx, 0, 0)-- cycle;
\draw[red,fill=red!30] (SSSSorigin) -- ++ (0,0, -\Sz) -- ++ (0,-\Sy, 0) -- ++ (0,0, \Sz)--cycle;
\draw[red,fill=red!30] (SSSSorigin) -- ++ (-\Sx, 0,0) -- ++ (0,0,-\Sz) -- ++ (\Sx, 0,0) -- cycle;

\pgfmathsetmacro{\Sigmax}{0.45}
\pgfmathsetmacro{\Sigmay}{3}
\pgfmathsetmacro{\Sigmaz}{3}

\coordinate (Sigmao) at ($(origin)+(\Sigmax,\Sigmay,0)$);

\draw[brown, thick] (Sigmao) -- ++(-\Sigmax, 0,0)-- ++ (0,-\Sigmay, 0) -- ++ (\Sigmax, 0, 0)-- cycle;
\draw[brown, thick] (Sigmao) -- ++ (0,0, -\Sigmaz) -- ++ (0,-\Sigmay, 0) -- ++ (0,0, \Sigmaz)--cycle;
\draw[brown, thick](Sigmao) -- ++ (-\Sigmax, 0,0) -- ++ (0,0,-\Sigmaz) -- ++ (\Sigmax, 0,0) -- cycle;

\draw[blue, ->] ($(Sigmao)-(0,\Sigmay/2,\Sigmaz)+(0.2,0,0)$) --++(-0.2,0,0);
\node[right, blue] at ($(Sigmao)-(0,\Sigmay/2,\Sigmaz)+(0.2,0,0)$) {$\tau$};

\pgfmathsetmacro{\Tx}{0.45}
\pgfmathsetmacro{\Ty}{3}
\pgfmathsetmacro{\Tz}{1}

\coordinate (TTorigin) at ($(origin)+(\Tx,\Ty,-2*\Tz)$);

\draw[blue, very thick] (TTorigin) -- ++(-\Tx, 0,0)-- ++ (0,-\Ty, 0) -- ++ (\Tx, 0, 0)-- cycle;
\draw[blue, very thick]  (TTorigin) -- ++ (0,0, -\Tz) -- ++ (0,-\Ty, 0) -- ++ (0,0, \Tz)--cycle;
\draw[blue, very thick] (TTorigin) -- ++ (-\Tx, 0,0) -- ++ (0,0,-\Tz) -- ++ (\Tx, 0,0) -- cycle;


\coordinate (SSorigin) at ($(origin)+(\Sx,\Sy,0)+(2*\Sx,0,0)$);

\draw[red,fill=red!30] (SSorigin) -- ++(-\Sx, 0,0)-- ++ (0,-\Sy, 0) -- ++ (\Sx, 0, 0)-- cycle;
\draw[red,fill=red!30] (SSorigin) -- ++ (0,0, -\Sz) -- ++ (0,-\Sy, 0) -- ++ (0,0, \Sz)--cycle;
\draw[red,fill=red!30] (SSorigin) -- ++ (-\Sx, 0,0) -- ++ (0,0,-\Sz) -- ++ (\Sx, 0,0) -- cycle;

\coordinate (Torigin) at ($(origin)+(\Tx,\Ty,0)$);

\draw[blue, very thick] (Torigin) -- ++(-\Tx, 0,0)-- ++ (0,-\Ty, 0) -- ++ (\Tx, 0, 0)-- cycle;
\draw[blue, very thick] (Torigin) -- ++ (0,0, -\Tz) -- ++ (0,-\Ty, 0) -- ++ (0,0, \Tz)--cycle;
\draw[blue, very thick](Torigin) -- ++ (-\Tx, 0,0) -- ++ (0,0,-\Tz) -- ++ (\Tx, 0,0) -- cycle;

\pgfmathsetmacro{\Qx}{3}
\pgfmathsetmacro{\Qy}{3}
\pgfmathsetmacro{\Qz}{5}

\coordinate (Qorigin) at ($(origin)+(2.3-\Pix+\Sigmax,3,0)$);

\draw[black] (Qorigin) -- ++(-\Qx, 0,0)-- ++ (0,-\Qy, 0) -- ++ (\Qx, 0, 0)-- cycle;
\draw[black]  (Qorigin) -- ++ (0,0, -\Qz) -- ++ (0,-\Qy, 0) -- ++ (0,0, \Qz)--cycle;
\draw[black]  (Qorigin) -- ++ (-\Qx, 0,0) -- ++ (0,0,-\Qz) -- ++ (\Qx, 0,0) -- cycle;

\draw[black, ->] ($(Sigmao)-(\Sigmax/2,\Sigmay+0.2,0)$) --($(Sorigin)-(\Sx/2,\Sy,-0.1)$);
\draw[black, ->] ($(Sigmao)-(\Sigmax/2,\Sigmay+0.2,0)$) --($(SSorigin)-(\Sx/2,\Sy,-0.1)$);
\node [roundnode]  [below] at ($(Sigmao)-(\Sigmax/2,\Sigmay+0.2,0)$) {$N$};


\draw[black, ->] ($(Torigin)-(-1.6,\Ty,\Tz+0.4)$) --($(TTorigin)-(-0.05,\Ty,\Tz/2)$);
\draw[black, ->] ($(Torigin)-(-1.6,\Ty,\Tz+0.4)$) --($(Torigin)-(-0.05,\Ty,\Tz/2)$);
\node[roundnode][right] at ($(Torigin)-(-1.6,\Ty,\Tz+0.4)$) {$Z_1$};

\draw[black, ->] ($(SSigmao)-(\Sigmax/2+0.3,3*\Sigmay/5,0)$) --($(Sigmao)-(\Sigmax/2,2*\Sigmay/5,0)$);
\draw[black, ->] ($(SSigmao)-(\Sigmax/2+0.3,3*\Sigmay/5,0)$) --($(SSigmao)-(\Sigmax/2,2*\Sigmay/5,0)$);
\node [roundnode]  [left] at ($(SSigmao)-(\Sigmax/2+0.3,3*\Sigmay/5,0)$) {$Z_2$};

\coordinate(ss) at (0.2, 0, 0);
\coordinate (m1) at ($(Qorigin)+(-\Qx, 0, 0) - (ss)$);
\coordinate (m2) at ($(Qorigin)+(-\Qx, 0, -\Qz) - (ss)$);
\draw[gray, <->] (m1)--(m2);
\node[left] at ($(m1)! 0.7 !(m2) $)  {$R$};

\coordinate(ss) at (0.2, 0, 0);
\coordinate (m3) at ($(Qorigin)+(-\Qx, -\Qy, 0) - (ss)$);
\coordinate (m4) at ($(Qorigin)+(-\Qx, 0, 0) - (ss)$);
\draw[gray, <->] (m3)--(m4);
\node[left] at ($(m3)! 0.5 !(m4) $)  {$R$};

\coordinate(sss) at (0.1, 0, 0);
\coordinate (m5) at ($(Sorigin)+(-\Sx, 0, 0) - (sss)$);
\coordinate (m6) at ($(Sorigin)+(-\Sx, 0, -\Sz) - (sss)$);
\draw[gray, <->] (m5)--(m6);
\node[left] at ($(m5)! 0.7 !(m6) $)  {$R^{\frac23}$};

\coordinate(ssss) at (0, 0, 0.2);
\coordinate (m9) at ($(SSSorigin)+(-\Sx, 0, -\Sz) - (ssss)$);
\coordinate (m10) at ($(SSSorigin)+(0,0, -\Sz) - (ssss)$);
\draw[gray, <->] (m9)--(m10);
\node[above] at ($(m9)! 0.5 !(m10) $)  {$R^{\frac{1}{3}}$};

\coordinate(sssss) at (0, 0, 1);
\coordinate (m11) at ($(TTorigin)+(-\Tx, 0, -\Tz) - (sssss)$);
\coordinate (m12) at ($(TTorigin)+(0, 0, -\Tz) - (sssss)$);
\draw[gray, <->] (m11)--(m12);
\node[above] at ($(m11)! 0.5 !(m12) $)  {$R^{\frac{1}{2}}$};

\coordinate(ssssss) at (0, 0, 0.2);
\coordinate (m13) at ($(Pio)+(-\Pix, 0, -\Piz) - (ssssss)$);
\coordinate (m14) at ($(Pio)+(0, 0, -\Piz) - (ssssss)$);
\draw[gray, <->] (m13)--(m14);
\node[above] at ($(m13)! 0.5 !(m14) $)  {$R^{\frac{2}{3}}$};

\coordinate(sssssss) at (0.15, 0, 0);
\coordinate (m15) at ($(SSigmao)+(-\Sigmax, 0, 0) - (sssssss)$);
\coordinate (m16) at ($(SSigmao)+(-\Sigmax, 0, -\Sigmaz) - (sssssss)$);
\draw[gray, <->] (m15)--(m16);
\node[left] at ($(m15)! 0.7 !(m16) $)  {$R^{\frac56}$};

\coordinate (corigin) at ($(Qorigin)+(-\Qx,-\Qy,0)$);

\draw[->] (corigin)--++(4.5,0,0);
\node[right] at ($(corigin)+(4.5,0,0)$) {$x$};
\draw[->] (corigin) -- ++ (0,4.5, 0);
\node[left] at ($ (corigin) + (0, 4.5, 0)$) {$z$};
\draw[->] (corigin)--++ (0, 0, -1.5);
\node[left] at ($ (corigin)+(0,0,-1.5)$) {$y$};
\end{tikzpicture}
\end{center}
\bigskip

The second pigeonholing sequence  is over. We only keep the  planks $P\in\P_A$ that are contained in some heavy $\tau$, which itself is contained in a heavy $\Sigma$, contained in some contributing $\Pi_I$. We call $\P^{(j)}$ the collection of these planks.\\

We write  $$g^{(j)}=\sum_{P\in\P^{(j)}}g_P$$ so that $g$ is $\sum_j g^{(j)}$, apart from a negligible error. We can decompose each function $g_1$, $g_2$, $g_3$ as a sum of $O((\log R)^C)$ many restricted functions $g_1^{(j_1)}$, $g_2^{(j_2)}$, $g_3^{(j_3)}$ associated with the families of plates $P\in \P^{(j)}(g_1)$,  $\P^{(j)}(g_2)$, $\P^{(j)}(g_3)$ respectively. We have as before

\begin{equation}
\label{eqn:pigeonhole2}
\|(g_1g_2g_3)^{1/3}\|_{L^p(\R^3)}\lessapprox \sup_{j_1,j_2,j_3}\|(g_1^{(j_1)}g_2^{(j_2)}g_3^{(j_3)})^{1/3}\|_{L^p(\R^3)}.
\end{equation}
Again we use the same index $j_1=j_2=j_3=j$ and denote $g^{(j)}$ by $h$, $\P^{(j)}$ by $\P$ and $\P^{(j)}_I$ by $\P_I$.
We also call $A, N, Z_1, Z_2$ the parameters associated with $g$.

From the last pigeonholing sequence we have
\begin{equation}
\label{eqn:contributingP}
 |\P|\lesssim |\I_{contr}|NZ_1Z_2Y\le |\I_{heavy}|NZ_1Z_2Y.
\end{equation}
This will be used in the proof of Theorem \ref{ejfyer7f0-p-i98t90-}.

Note that the function $h$  has  Fourier transform supported on the anisotropic neighborhood of the cubic moment curve $\Gamma(R^{-\frac13})$, due (P1) in Theorem \ref{te:WPD2}. Let us finish wave packet decomposition by writing (recall that only $I\in\I_{contr}$ contribute to the summation)
$$h=\sum_{I\in \I_{R^{-1/3}}}\Pc_{2I} h=\sum_{I\in \I_{R^{-1/3}}}\sum_{P\in \P_I}g_P.$$

We finished introducing and pigeonholing parameters. Combining \eqref{eqn:pigeonhole1} and (\ref{eqn:pigeonhole2}) our main Theorem  \ref{te:mainrecall}  can be reduced to showing
\begin{equation}
\label{eqn:finalreduction}
\|(h_1h_2h_3)^{\frac13}\|_{L^{10}(\R^3)}\lesssim_\epsilon R^{\frac12(\frac12-\frac{1}{10})+\epsilon}(|\I_{heavy}|m w^{10} nX |W|)^{\frac{1}{10}}.
\end{equation}
In the next section, we prove this inequality using a two step decoupling approach.

\subsection{Proof of the main theorem}

The following two propositions make use of the uniformity we obtained from  pigeonholing.

The first one decouples intervals $I$ into smaller intervals $J$. This will be achieved by using $L^2$ orthogonality, $l^4(L^4)$ small cap decoupling and $l^2(L^6)$ canonical scale decoupling for the parabola. This result offers the main connection between parameters from the two pigeonholing sequences.

\begin{pr}
\label{te:246decoupling}
For the parameters $w,n,l,A,N,Z_1,Z_2$ from the two pigeonholing sequences, we have the following inequality
$$A\lesssim_\epsilon R^\epsilon\min(\frac{wl^{\frac12}R^{\frac{1}{12}}}{N^{\frac12}},\frac{wl^{\frac14}R^{\frac{1}{8}}}{(NZ_1)^{\frac14}},\frac{wl^{\frac12}n^{\frac{1}{6}}R^{\frac{1}{12}}}{(NZ_1Z_2)^{\frac16}}).$$
\end{pr}

The second proposition is an incidence estimate for Vinograodov planks $P_I$ under spacing condition from the second pigeonholing sequence. It essentially amounts to decoupling into intervals $I$ of canonical scale.

\begin{pr}
\label{te:Plankinci}
Let $\P$ be the collection of planks obtained at the end of the second pigeonholing sequence. Let $\Qc_r(\P)$ be the collection of $r$-rich $R^{\frac13}$-cubes $q$ with respect to $\P$. Then for each $1\leq r \leq R^{\frac13}$,
\begin{equation*}
    |\Qc_r(\P)|\lessapprox \frac{|\P|N^5Z_1R^{2}}{r^7}.
\end{equation*}
\end{pr}
Note that parameters $X$ and $Y$ do not appear in either proposition. They will however play a role in the proof of Theorem \ref{ejfyer7f0-p-i98t90-} below.

We postpone the proofs of both propositions for later. Let us now prove the reduced version  (\ref{eqn:finalreduction}) of main theorem. We use ``interpolation" (via H\"older's inequality) of the trilinear $L^6$ restriction estimate and the refined $l^{12}(L^{12})$ decoupling, as shown in Proposition 8.7 of \cite{DGH}.

\begin{te}
\label{ejfyer7f0-p-i98t90-}	
We have
$$\|(h_1h_2h_3)^{\frac13}\|_{L^{10}(\R^3)}\lesssim_\epsilon R^{\frac12(\frac12-\frac{1}{10})+\epsilon}(|\I_{heavy}|m w^{10} nX |W|)^{\frac{1}{10}}$$
\end{te}

\begin{proof}  Let us denote $\P(h_1)$, $\P(h_2)$, $\P(h_3)$ by $\P_1$, $\P_2$, $\P_3$ and $\cup_{k=1}^3\P_k$ by $\P$. It suffices to prove that for each $r$
$$\|(h_1h_2h_3)^{\frac13}\|_{L^{10}(\Qc_r(\P))}\lesssim_\epsilon R^{\frac12(\frac12-\frac{1}{10})+\epsilon}(|\I_{heavy}|m w^{10} nX |W|)^{\frac{1}{10}}.$$
This is because cubes $q$ intersecting at least one $P\in\P$ will be $r$-rich for some dyadic parameter $1\le r\le R^{1/3}$, while the contribution from those $q$ not intersecting any planks can be considered negligible.
\smallskip

Let us fix $r\ge 1$.	
Note that the functions $h_1$, $h_2$, $h_3$ have Fourier transform supported in the anisotropic neighborhood $\Gamma(R^{-\frac13})$, which is contained in $\Nc_{\Gamma}(R^{-\frac13})$. Therefore, we can apply the trilinear $L^6$ restriction theorem for curves on each $q$
$$
\|(h_1h_2h_3)^{\frac13}\|_{L^{6}(q)}\lesssim (\|\sum_{P\in \P_1}|g_P|^2\|_{L^3(\chi_q)} \|\sum_{P\in \P_2}|g_P|^2\|_{L^3(\chi_q)} \|\sum_{P\in \P_3}|g_P|^2\|_{L^3(\chi_q)})^{\frac16}.
$$
Summing this up for all  cubes $q\in \Qc_{r}(\P)$, we have
\begin{align*}
 \|(h_1h_2h_3)^{\frac13}\|_{L^{6}(\cup_{q\in\Qc_{r}(\P)}q)}
&\lesssim (\|\sum_{P\in \P_1}|g_P|^2\|_{L^3(\sum\chi_q)} \|\sum_{P\in \P_2}|g_P|^2\|_{L^3(\sum\chi_q)} \|\sum_{P\in \P_3}|g_P|^2\|_{L^3(\sum\chi_q)})^{\frac16}\\
&\lesssim A(\|\sum_{P\in \P_1}|\chi_P|^2\|_{L^3(\sum\chi_q)} \|\sum_{P\in \P_2}|\chi_P|^2\|_{L^3(\sum\chi_q)} \|\sum_{P\in \P_3}|\chi_P|^2\|_{L^3(\sum\chi_q)})^{\frac16}\\
&\lesssim_\epsilon A(|\Qc_r(\P)|R^{1+\epsilon}r^3)^\frac16.
\end{align*}
Next, we apply the $l^{12}(L^{12})$ refined decoupling Theorem \ref{te:L12refined} to each $h_1$, $h_2$ and $h_3$,
\begin{align*}
\|(h_1h_2h_3)^{\frac13}\|_{L^{12}(\cup_{q\in \Qc_r(\P)}q)}& \le (\|h_1\|_{L^{12}(\cup_{q\in \Qc_r(\P)}q)}\|h_2\|_{L^{12}(\cup_{q\in \Qc_r(\P)}q)}\|h_3\|_{L^{12}(\cup_{q\in \Qc_r(\P)}q)})^{1/3}
\\&\lesssim_\epsilon r^{\frac{5}{12}}R^{\epsilon}(\sum_{I\in \I_{R^{-1/3}}}\|\Pc_{2I}h\|_{L^{12}(\R^3)}^{12})^{\frac{1}{12}}\\
&\lesssim r^{\frac{5}{12}}R^{\epsilon} A(|\P|R^2)^{\frac{1}{12}}.
\end{align*}
For the last two sequences of inequalities, we have used that each $q\in\Qc_r(\P)$ intersects at most $r$ planks in each of the families $\P_1,\P_2,\P_3$. This is immediate, since each $\P_i$ is a subset of $\P$.

We combine these two inequalities via H\"older's inequality
\begin{align*}
   \|(g_1g_2g_3)^{\frac13}\|_{L^{10}(\cup_{q\in\Qc_r(\P)}q)}\lesssim_\epsilon AR^{\epsilon}(|\Qc_r(\P)|Rr^3)^\frac{1}{30}(r^{\frac{1}{3}} (|\P|R^2)^{\frac{1}{15}}).
\end{align*}
Therefore, it suffices to show that
$$A(|\Qc_r|Rr^3)^\frac{1}{30}(r^{\frac{1}{3}} (|\P|R^2)^{\frac{1}{15}}) \lesssim_\epsilon R^{\frac12(\frac12-\frac{1}{10})+\epsilon}(|\I_{heavy}|m w^{10} nX |W|)^{\frac{1}{10}}.$$
Let us raise both sides to power 30, plug in $|W|\sim R^{\frac52}$ and prove
\begin{align*}
    A^{30}(|\Qc_r|Rr^3)(r^{10}|\P|^2R^4)\lesssim_\epsilon R^{\epsilon}w^{30} |\I_{heavy}|^3m^3n^3X^3R^{\frac{27}{2}}.
\end{align*}
We apply the $A$-estimate from Proposition \ref{te:246decoupling}, the $|\Qc_r|$-estimate from Proposition \ref{te:Plankinci} and the $|\P|$-estimate  in (\ref{eqn:contributingP})
\begin{align*}
A^{30}(|\Qc_r|Rr^3)(r^{10}|\P|^2R^4)&\lesssim_\epsilon R^{\epsilon}(\frac{wl^{\frac12}R^{\frac{1}{12}}}{N^{\frac12}})^8(\frac{wl^{\frac14}R^{\frac{1}{8}}}{(NZ_1)^{\frac14}})^4(\frac{wl^{\frac12}n^{\frac{1}{6}}R^{\frac{1}{12}}}{(NZ_1Z_2)^{\frac16}})^{18}(\frac{|\P|N^5Z_1R^2}{r^7}r^{13}|\P|^2R^5)\\
&\lesssim R^{\epsilon}w^{30}(\frac{l^{14}n^3R^{\frac83}}{N^8Z_1^4Z_2^3})(|\I_{heavy}|NZ_1Z_2Y)^3N^5Z_1r^6R^7\\
&= R^{\epsilon}w^{30}l^{14}n^3|\I_{heavy}|^3Y^3r^6R^{\frac{29}{3}}.
\end{align*}
Finally, we  use \eqref{fsddppf-=0pg}, $l\leq R^{\frac16}$, $r\leq R^{\frac13}$
\begin{align*}
    w^{30}l^{14}n^3|\I_{heavy}|^3Y^3r^6R^{\frac{29}{3}} &\lesssim w^{30}l^{11}n^3|\I_{heavy}|^3m^3X^3r^6R^{\frac{29}{3}}\\
    &\lesssim w^{30} |\I_{heavy}|^3 m^{3} n^{3} X^3 R^{\frac{27}{2}}.
\end{align*}
This finishes the proof of the main theorem.
\end{proof}

\bigskip

\subsection{Proof of Proposition \ref{te:246decoupling}}
In this section  we prove the inequality
$$A\lesssim_\epsilon R^\epsilon\min(\frac{wl^{\frac12}R^{\frac{1}{12}}}{N^{\frac12}},\frac{wl^{\frac14}R^{\frac{1}{8}}}{(NZ_1)^{\frac14}},\frac{wl^{\frac12}n^{\frac{1}{6}}R^{\frac{1}{12}}}{(NZ_1Z_2)^{\frac16}}).$$

\begin{proof} We prove the  three estimates using different methods.
\\
\\
\textbf{$L^2$ orthogonality on $\Delta$: the proof of}	
$$A\lesssim\frac{wl^{\frac12}R^{\frac{1}{12}}}{N^{\frac12}}.$$
We choose an arbitrary cube $\Delta$ of side length $R^{\frac12}$ inside one of the heavy $(R^{\frac12},R^{\frac23},R)$-boxes $\tau$ selected in the second pigeonholing sequence. The box $\tau$ is associated with some contributing $I\in \I_{contr}$ and $\tau$ is contained in some contributing (in particular also heavy) fat plate $\Pi_I$. Recall that there are $\sim N$ many parallel planks $P_I$ inside the box $\tau$. Comparing dimensions, we see that each such plank must intersect $\Delta$.  Therefore the cube $\Delta$ intersects $\sim N$ many parallel planks $P_I$ and each intersection has the volume $\sim R^\frac43$. Thus
$$\|\Pc_{2I}h\|_{L^2(\Delta)}\gtrsim A(NR^{\frac43})^{\frac12}.$$

 Local $L^2$ orthogonality shows that
$$\|\Pc_{2I}h\|_{L^2(\Delta)}\leq\|\Pc_{2I}g\|_{L^2(\Delta)}\lesssim (\sum_{J\in\I_{R^{-1/2}}(I)}\|\Pc_{2J}g\|_{L^2(w_\Delta)}^2)^{\frac12}$$
Let us recall the wave packet decomposition of the function $\Pc_{2I}g$ is
$$\Pc_{2I}g=\sum_{J\in \I_{R^{-1/2}}(I)} \Pc_{2J}g=\sum_{J\in \I_{R^{-1/2}}(I)}\sum_{W\in \W_J}F_W$$
Since  $\Pi_I$ is also heavy according to the first pigeonholing sequence, it is contributed by $\sim l$ many intervals $J\subset I$, and contains $\sim n$ many parallel plates $W$ for each of the contributing direction. For each direction $J$, the cube $\Delta$ intersects at most $O(1)$ plates $W$ and  $|F_W|$ is essentially constant on the cube $\Delta$. Thus we have,
$$(\sum_{J\in\I_{R^{-1/2}}(I)}\|\Pc_{J}g\|_{L^2(w_\Delta)}^2)^{\frac12}\lesssim wl^{\frac12}|\Delta|^{\frac12}\sim wl^{\frac12}R^{\frac34}.$$

Combining the last three inequalities, we get
$$A(NR^{\frac43})^{\frac12}\lesssim wl^{\frac12}R^{\frac34}.$$
This finishes the proof of the first estimate.
\\
\\
\textbf{Small cap $l^4(L^4)$ decoupling on $\Sigma$: the proof of}
$$A\lesssim_\epsilon R^{\epsilon} \frac{wl^{\frac14}R^{\frac{1}{8}}}{(NZ_1)^{\frac14}}.$$
Let us choose an arbitrary heavy $(R^{\frac12},R^{\frac56},R)$-box $\Sigma$, selected in the second pigeonholing sequence. Let $I=[c,c+R^{-\frac13}]$ be the interval that $\Sigma$ is associated with, and let $\Pi_I$ be the contributing fat plate containing the box $\Sigma$ (so $\Pi_I$ is also heavy). We can translate  $\Pi_I$ to the origin and apply the linear transformation $A_{1,c}$. Then  $I$ becomes $[0,R^{-\frac13}]$ and we can use cylindrical decoupling (Exercise 9.22 of \cite{Book}). The Fourier transform of $\Pc_I g$ is supported in $\Nc_{[0,R^{-1/3}]}(R^{-1})$. This neighborhood is essentially planar, it lies inside $\{(\xi,\xi^2,0):\xi\in [0,R^{-\frac13}]\}+B(0,R^{-1})$. We apply planar parabolic rescaling
$(\xi_1,\xi_2)\mapsto (R^{\frac13}\xi_1,R^{\frac23}\xi_2)$ to stretch the $I$ to $[0,1]$, the intervals $J$ to intervals in $\I_{R^{-1/6}}$, and the $(R^{1/2},R^{5/6})$ horizontal slice of $\Sigma$ into an $R^{1/6}$-square.  We use $l^4(L^4)$ small cap decoupling, see Theorem 2.3 in \cite{DGH} to get
$$\|\Pc_{2I}h\|_{L^{4}(\Sigma)}\lesssim \|\Pc_{2I}g\|_{L^{4}(\Sigma)}\lesssim_\epsilon R^{\frac16(\frac12-\frac14)+\epsilon}(\sum_{J\in \I_{R^{-1/2}}(I)}\|\Pc_{2J}g\|_{L^4(w_{\Sigma})}^4)^{\frac14}.$$

Since  $\Pi_I$ is heavy, there are $\sim l$ contributing intervals $J\subset I$. Each $\Sigma$ intersects at most $O(1)$ plates $W$ from each contributing direction. Since the wave packet $|F_W|$ is essentially- constant on the box $\Sigma$ and since $|\Sigma|\sim R^{7/3}$, we have
$$R^{\frac16(\frac12-\frac14)+\epsilon}(\sum_{J\in \I_{R^{-1/2}}(I)}\|\Pc_{2J}g\|_{L^4(w_{\Sigma})}^4)^{\frac14} \lesssim_\epsilon R^{\frac{1}{24}+\epsilon}w(lR^{\frac73})^\frac14=wl^{\frac14}R^{\frac58+\epsilon}$$
On the other hand, the heavy $\Sigma$ contains $\sim NZ_1$ many parallel planks $P_I$ and each plank $P_I$ has the volume $\sim R^2$. Thus we have
$$\|\Pc_{2I}h\|_{L^{4}(\Sigma)}\gtrsim A(NZ_1R^2)^{1/4}.$$
We combine the last three inequalities and get
$$A(NZ_1R^2)^\frac14\lesssim_\epsilon wl^{\frac14}R^{\frac58+\epsilon}.$$
This proves the second estimate.
\\
\\
\textbf{$l^2(L^6)$ decoupling: the proof of}

$$A\lesssim_\epsilon R^\epsilon\frac{wl^{\frac12}n^{\frac{1}{6}}R^{\frac{1}{12}}}{(NZ_1Z_2)^{\frac16}}.$$
Let us choose a contributing fat plate $\Pi_I$ for some $I=[c,c+R^{-\frac13}]$. We translate the fat plate $\Pi_I$ to the origin and apply the linear transformation $A_{1,c}$. Then the corresponding interval becomes $[0,R^{-\frac13}]$. We apply the parabolic rescaling $(\xi_1,\xi_2)\mapsto (R^{\frac13}\xi_1,R^{\frac23}\xi_2)$ as before, and use $l^2(L^6)$  decoupling to get
$$\|\Pc_{2I} h\|_{L^6(\Pi_I)}\lesssim\|\Pc_{2I} g\|_{L^6(\Pi_I)} \lesssim_\epsilon R^\epsilon(\sum_{J\in \I_{R^{-1/2}}(I)}\|\Pc_{2J} g\|_{L^6(w_{\Pi_I})}^2)^{\frac12}.$$
The fat plate $\Pi_I$ contains $\sim n$ many parallel plates $W\in \W_J$ for each of the contributing $\sim l$ directions $J\subset I$. Since each plate $W$ has a volume $\sim R^\frac52$ we have the following
$$(\sum_{J\in \I_{R^{-1/2}}(I)}\|\Pc_{2J} g\|_{L^6(w_{\Pi_I})}^2)^{\frac12} \lesssim_\epsilon wl^{\frac12}n^{\frac16}R^{\frac{5}{12}+\epsilon}.$$
On the other hand, since the contributing fat plate $\Pi_I$ contains $\sim NZ_1Z_2$ many planks $P_I$, we have
$$\|\Pc_{2I} h\|_{L^6(\Pi_I)}\gtrsim A(NZ_1Z_2R^2)^{\frac16}.$$
Combining the last three inequalities, we have
$$A(NZ_1Z_2R^2)^{\frac16}\lesssim_\epsilon wl^{\frac12}n^{\frac16}R^{\frac{5}{12}+\epsilon}.$$
This finishes the proof of Proposition \ref{te:246decoupling}.
\end{proof}
\bigskip

\subsection{Proof of Proposition \ref{te:Plankinci}}
\label{8.5}
We rewrite  Proposition \ref{te:Plankinci} as follows, spelling out explicitly the spacing conditions satisfied by $\P$. We recall that for each $I\in\I_{R^{-1/3}}$, the $(R^{\frac12},R^{\frac56},R)$-boxes  $\Sigma_I$ tile $[-R,R]^3$. Also, each $\Sigma_I$ is tiled with $(R^{\frac12},R^{\frac23},R)$-boxes $\tau_I$. Finally, each $\tau_I$ is tiled with $(R^{\frac13},R^{\frac23},R)$-Vinogradov planks $P_I$.

\begin{pr}
Let $\P$ be a collection of $(R^{1/3},R^{2/3},R)$-Vinogradov planks in $[-R,R]^3$. Assume the following two spacing conditions
\\
\\
(Sp1) We call $\tau_I$ heavy if it contains $\sim N$ planks $P_I$. We assume that for  each $I$, all planks $P_I$ are subsets of heavy boxes $\tau_I$. That is, each $\tau_I$ is either heavy, or otherwise contains no plank $P_I$.
\\
\\
(Sp2) We call $\Sigma_I$ heavy if it contains $\sim Z_1$ heavy boxes $\tau_I$. We assume that for each $I$, all heavy $\tau_I$ are subsets of heavy boxes $\Sigma_I$. That is, each $\Sigma_I$ is either heavy, or otherwise contains no heavy $\tau_I$, so it also contains no planks $P_I\in\P$.

Let $\Qc_r(\P)$ be the collection of $r$-rich $R^{\frac13}$-cubes $q$ with respect to $\P$. Then for each $1\leq r\leq R^{\frac13}$ we have
\begin{equation*}
    |\Qc_r(\P)|\lessapprox \frac{|\P|N^5Z_1R^{2}}{r^7}.
\end{equation*}
\end{pr}
\bigskip
\begin{proof}
	
In addition to intervals $I\in\I_{R^{-1/3}}$, we will also consider intervals $H\in \I_{R^{-1/6}}$.

We start with  several pigeonholing steps that introduce additional structure to the argument. There will be various boxes, dyadic parameters and Kakeya estimates. The main idea is to replace planks $P$ with smaller  planks $S$, with dimensions $(R^{\frac13},R^{\frac12},R^{\frac23})$,  and to consider their incidences inside smaller cubes $Q$ with side length $R^{2/3}$.\\

We will first need to estimate the number of such planks $S$, in terms of the size $|\P|$ of the collection of large planks. This will be achieved in the first five steps of the argument, by using our earlier results on both tube and plate incidences. These results will be used again inside each $Q$, throughout Steps 6-11 of the argument. There is a finer localization to smaller $R^{1/2}$-cubes $\Delta$. The number of  relevant $\Delta$ is counted using incidences for tubes, while the number of $r$-rich $q$ inside each $\Delta$ is counted using incidences for plates. In Step 12, these estimates are summed over all $\Delta$ and $Q$ inside $[-R,R]^3$.
\\
	
1. Replacing the planks $P$ with smaller planks $S$
\\
	
For each ${H}\in \I_{R^{-1/6}}$, we tile $[-R,R]^3$ with $(R^{\frac23},R^{\frac56},R)$-boxes $B=B_H$ with corresponding axes $(\textbf{t}(H),\textbf{n}(H),\textbf{b}(H))$. Computations similar to those  from Lemma \ref{le:planksimplegeo} and Lemma \ref{le:platesmallangle} show that for each $I\subset H$,  each  $\Sigma_I$, $\tau_I$ and $P_I$ are subsets of exactly  one of the boxes $B_H$. In fact, the dimensions of $B_H$ are the smallest subject to this property. It is easy to check if $\dist(H,H')\gg R^{-1/6}$, then the long side of the almost rectangular box $B_H\cap B_{H'}$ is $\ll R$, so no $P,\tau$ or $\Sigma$ would fit inside both boxes, since they all have long side equal to $R$.

Referring back again to Lemma \ref{le:planksimplegeo} with $\sigma=R^{-1/6}$,
we tile each $B_H$ with smaller planks $S$ with dimension $\sim (R^{\frac13},R^{\frac12},R^{\frac23})$ with the same axes as the box $B_H$. It is immediate that $R^{-1/6}S$ is a spatial Vinogradov plank associated with $H$, and that $R^{1/3}S$ has the same dimensions as $B_H$.
The dimensions and orientation of $S$ guarantees that each $P_I$ with $I\subset H$ can be tiled with translates of $S$. So we think about the intersection of any number of planks $P_I\subset B_H$, for various $I\subset H$, as being the union of pairwise disjoint $S$. This is similar to understanding the intersection of congruent rectangles in $\R^2$ as being essentially a union of squares, whose side length equals the width of the rectangles.

Let us write the set of all small planks $S$ in $[-R,R]^3$ as $\S$, and the set of all  $S$ contained in the box $B$ as $\S_B$.
\\

2. Pigeonholing the parameters $E_1,E_2$
\\

We partition the family of small planks $\S$ according to the dyadic parameter $E_2$: the number of planks $P$ that each $S$ belongs to. We only count the planks $P$ lying inside the same box $B$ that $S$ lies inside of.
Note that $E_2\leq R^{\frac16}$. There are $\lesssim \log R$ such dyadic values of $E_2$, and we pick one value for $E_2$ such that \eqref{jhdhdhuyhhfcyrfui} holds. We will call these small planks  $E_2$-planks and from now on $\S$, $\S_B$ will refer only the $E_2$-planks.

We can transform the original problem ``counting $r$-rich $R^{\frac13}$-cubes $q$ with respect to $\P$" to ``counting $\frac{r}{E_2}$-rich $R^{\frac13}$-cubes $q$ with respect to $\S$". We define $E_1:=\frac{r}{E_2}$ and note that $E_1\leq R^{\frac16}$. Note that $E_1$ measures the number of intervals $H$ that contribute. We write the family of $E_1$-rich $R^{\frac13}$-cubes $q$ as $\Qc_{E_1}(\S)$ and choose $E_2$ (or equivalently, $E_1$) such that
\begin{equation}
\label{jhdhdhuyhhfcyrfui}
|\Qc_r(\P)|\lessapprox |\Qc_{E_1}(\S)|.
\end{equation}
\\
	
3. Pigeonholing the  parameter $M_2$ and the boxes $\Lambda$
\\
	
Recall that each $P_I$ is a subset of a heavy $\tau\in \Tc$ with the same orientation. Thus the number of all heavy  $\tau\subset [-R,R]^3$ is
\begin{equation}
\label{eqn:tauestimate}
    |\Tc|\lesssim\frac{|\P|}{N}.
\end{equation}
Let us tile each $(R^{\frac23},R^{\frac56},R)$-box $B_H$ with $(R^{\frac12},R^{\frac23},R^{\frac56})$-boxes $\Lambda_H$ with the same axes.
Note that $R^{1/6}\Lambda_H$ has the same dimensions as $B_H$. Boxes $\Lambda_H$ arise as intersections of boxes $\tau_I$ associated with various intervals $I\subset H$. This is another computation as in  Lemma \ref{le:planksimplegeo}.

Each $E_2$-plank $S\subset B_H$ can be assumed to be contained inside a unique $\Lambda_H$. Only those $\Lambda_H$ that contain at least one such $S$ will be relevant to us. Note that each such $\Lambda_H$ must be intersected by at least $E_2$ heavy boxes $\tau_I$ with $I\subset H$. This is because $S$ is intersected by $\sim E_2$ planks $P_I$, and each of these planks lives inside a different $\tau_I$. We classify the relevant boxes $\Lambda_H$ according to the dyadic parameter $M_2\in[E_2, R^{1/6}]$ that counts the number of heavy $\tau_I$ intersecting $\Lambda_H$.

We fix $M_2$. We write the collection of $M_2$-rich boxes $\Lambda_H$ as $\Qc_{M_2}(\Tc)$ and the collection of such $\Lambda_H$ living inside the box $B$, as $\Qc_{M_2, B}(\Tc)$.
\\
	
4. Estimates for $\Qc_{M_2, B}(\Tc)$ using Vinogradov tube incidences
\\
	
Let $H=[c,c+R^{-\frac16}]$, $I=[a,a+R^{-\frac13}]$ and $I\subset H$. We fix a  box $B=B_H$, and aim to estimate the number of $M_2$-rich boxes $\Lambda_H$ inside $B.$
We translate $B$ to the origin, and apply the linear transformation $A_{R^{-1/6},c}$, where $A_{R^{-1/6},c}(x,y,z)=(x',y',z')$ is as follows
\begin{equation*}
\begin{cases}
x'=(R^{-\frac16})(x+2cy+3c^2z)\\
y'=(R^{-\frac16})^2(y+3cz)\\
z'=(R^{-\frac16})^3z.
\end{cases}
\end{equation*}

Let us investigate the images of the boxes $B, \Sigma, \tau, \Lambda$ under this linear map. The image of $B$, $\tilde{B}:=A_{R^{-1/6},c}(B)$ is an $R^{\frac12}$-cube. The image of each $\Sigma_I$, $\tilde{\Sigma}_I:=A_{R^{-1/6},c}(\Sigma_I)$ is an $(R^{\frac13},R^{\frac12},R^{\frac12})$-plate with normal vector $\textbf{t}(R^{\frac16}(a-c))$ contained in the cube $\tilde{B}$. The image of $\tau_I$, $\tilde{\tau}_I:=A_{R^{-1/6},c}(\tau_I)$ is a  $(R^{\frac13},R^{\frac13},R^{\frac12})$-Vinograodov tube at scale $R^{1/2}$ (according to Definition \ref{oiufufihvjdkvnjkgbj }) with long edge in the direction  $\textbf{b}(R^{\frac16}(a-c))$, contained in one of $\tilde{\Sigma}_I$. We note that intervals get stretched by $R^{1/6}$, so that $H$ becomes $[0,1]$ and each $I$ becomes an interval $\tilde{I}$ of length $R^{-1/6}$. So $\tilde{\tau}_I$ is a Vinogradov tube associated with the interval $\tilde{I}$, according to Definition \ref{oiufufihvjdkvnjkgbj }.
Lastly the image of $\Lambda$, $\tilde{\Lambda}:=A_{R^{-1/6},c}(\Lambda)$ is an $R^{\frac13}$-cube.\\

\bigskip

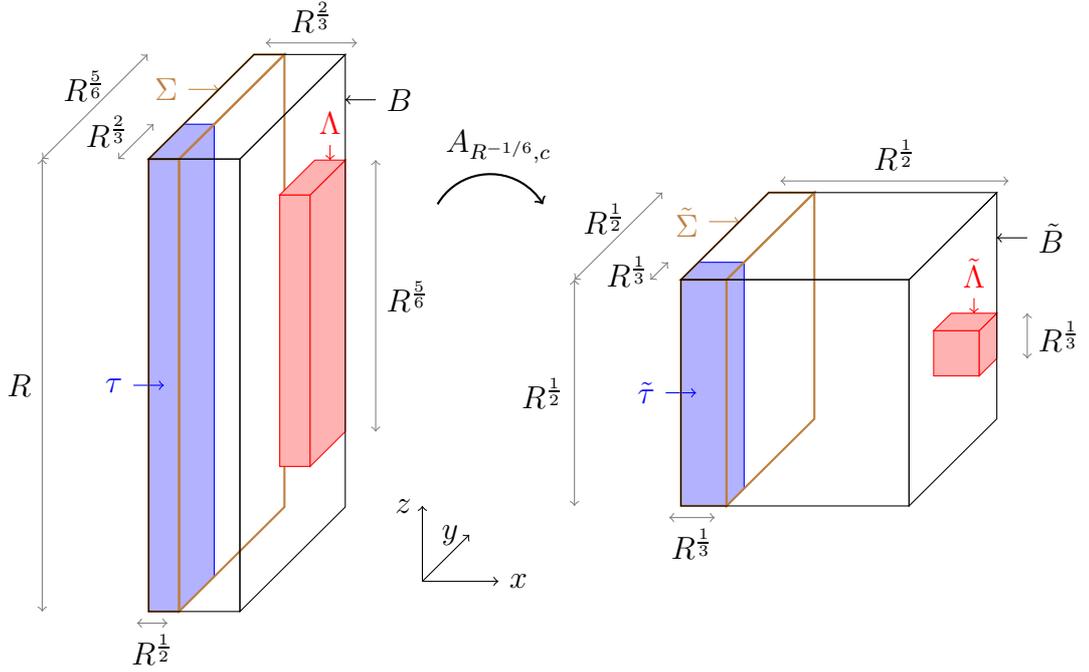
\textbf{Figure 3}
\begin{center}
\begin{tikzpicture}[scale=2]
\pgfmathsetmacro{\originonex}{3.5}
\pgfmathsetmacro{\originoney}{0}
\pgfmathsetmacro{\originonez}{5}

\coordinate (origin1) at (-\originonex, -\originoney, 0);

\pgfmathsetmacro{\taux}{0.2}
\pgfmathsetmacro{\tauy}{3}
\pgfmathsetmacro{\tauz}{0.6}

\coordinate (tauo) at ($(origin1)+(\taux,\tauy,0)$);

\draw[blue, fill=blue!30] (tauo) -- ++(-\taux, 0,0)-- ++ (0,-\tauy, 0) -- ++ (\taux, 0, 0)-- cycle;
\draw[blue, fill=blue!30]  (tauo) -- ++ (0,0, -\tauz) -- ++ (0,-\tauy, 0) -- ++ (0,0, \tauz)--cycle;
\draw[blue, fill=blue!30] (tauo) -- ++ (-\taux, 0,0) -- ++ (0,0,-\tauz) -- ++ (\taux, 0,0) -- cycle;

\draw[blue, ->] ($(tauo)-(\taux+0.1,\tauy/2,0)$) -- ++(0.2, 0, 0);
\node[left, blue] at ($(tauo)-(\taux+0.1,\tauy/2,0)$) {$\tau$};

\pgfmathsetmacro{\Sx}{0.2}
\pgfmathsetmacro{\Sy}{3}
\pgfmathsetmacro{\Sz}{1.8}

\coordinate (Sorigin) at ($(origin1)+(\Sx,\Sy,0)$); 

\draw[thick, brown] (Sorigin) -- ++(-\Sx, 0,0)-- ++ (0,-\Sy, 0) -- ++ (\Sx, 0, 0)-- cycle;
\draw[thick, brown] (Sorigin) -- ++ (0,0, -\Sz) -- ++ (0,-\Sy, 0) -- ++ (0,0, \Sz)--cycle;
\draw[thick, brown]  (Sorigin) -- ++ (-\Sx, 0,0) -- ++ (0,0,-\Sz) -- ++ (\Sx, 0,0) -- cycle;

\draw[brown , ->] ($(Sorigin)-(\Sx,0,2*\Sz/3)-(0.2,0,0)$) --++(0.2,0,0);
\node[left, brown] at ($(Sorigin)-(\Sx,0,2*\Sz/3)-(0.2,0,0)$) {$\Sigma$};

\pgfmathsetmacro{\lambdax}{0.2}
\pgfmathsetmacro{\lambday}{1.8}
\pgfmathsetmacro{\lambdaz}{0.6}

\coordinate (lambdao) at ($(origin1)+(0.6,2.3,-1.2)$);

\draw[red, fill=red!30]  (lambdao) -- ++(-\lambdax, 0,0)-- ++ (0,-\lambday, 0) -- ++ (\lambdax, 0, 0)-- cycle;
\draw[red, fill=red!30]   (lambdao) -- ++ (0,0, -\lambdaz) -- ++ (0,-\lambday, 0) -- ++ (0,0, \lambdaz)--cycle;
\draw[red, fill=red!30]  (lambdao) -- ++ (-\lambdax, 0,0) -- ++ (0,0,-\lambdaz) -- ++ (\lambdax, 0,0) -- cycle;
\draw[red, ->] ($(lambdao) + (-\lambdax/2, 0.1, -\lambdaz)$)--($(lambdao) + (-\lambdax/2, 0, -\lambdaz)$);
\node[red, above] at ($(lambdao) + (-\lambdax/2, 0.1, -\lambdaz)$) {$\Lambda$};

\pgfmathsetmacro{\Bx}{0.6}
\pgfmathsetmacro{\By}{3}
\pgfmathsetmacro{\Bz}{1.8}

\coordinate (Borigin) at ($(origin1)+(0.6,3,0)$);

\draw[black] (Borigin) -- ++(-\Bx, 0,0)-- ++ (0,-\By, 0) -- ++ (\Bx, 0, 0)-- cycle;
\draw[black]  (Borigin) -- ++ (0,0, -\Bz) -- ++ (0,-\By, 0) -- ++ (0,0, \Bz)--cycle;
\draw[black]  (Borigin) -- ++ (-\Bx, 0,0) -- ++ (0,0,-\Bz) -- ++ (\Bx, 0,0) -- cycle;

\draw[black, ->] ($(Borigin)-(0,0.3,\Bz)+(0.2,0,0)$) --++(-0.2,0,0);
\node[right, black] at ($(Borigin)-(0,0.3,\Bz)+(0.2,0,0)$) {$B$};

\coordinate(ss) at (0.7, 0, 0);
\coordinate (m1) at ($(Borigin)+(-\Bx, 0, 0) - (ss)$);
\coordinate (m2) at ($(Borigin)+(-\Bx, 0, -\Bz) - (ss)$);
\draw[gray, <->] (m1)--(m2);
\node[left] at ($(m1)! 0.7 !(m2) $)  {$R^{\frac{5}{6}}$};

\coordinate(ss) at (0.7, 0, 0);
\coordinate (m3) at ($(Borigin)+(-\Bx, -\By, 0) - (ss)$);
\coordinate (m4) at ($(Borigin)+(-\Bx, 0, 0) - (ss)$);
\draw[gray, <->] (m3)--(m4);
\node[left] at ($(m3)! 0.5 !(m4) $)  {$R$};

\coordinate(sss) at (0.2, 0, 0);
\coordinate (m5) at ($(Borigin)+(-\Bx, 0, 0) - (sss)$);
\coordinate (m6) at ($(Borigin)+(-\Bx, 0, -\tauz) - (sss)$);
\draw[gray, <->] (m5)--(m6);
\node[left] at ($(m5)! 0.8 !(m6) +(-0.05,0,0)$)  {$R^{\frac{2}{3}}$};

\coordinate(ssss) at (0, 0, 0.2);
\coordinate (m7) at ($(Borigin)+(-\Bx, 0, -\Bz) - (ssss)$);
\coordinate (m8) at ($(Borigin)+(0, 0, -\Bz) - (ssss)$);
\draw[gray, <->] (m7)--(m8);
\node[above] at ($(m7)! 0.5 !(m8) $)  {$R^{\frac{2}{3}}$};

\coordinate(ssss) at (0, 0, 0.2);
\coordinate (m9) at ($(Borigin)+(-\Bx, -\By, 0) + (ssss)$);
\coordinate (m10) at ($(Borigin)+(-\Bx+\Sx, -\By, 0) + (ssss)$);
\draw[gray, <->] (m9)--(m10);
\node[below] at ($(m9)! 0.5 !(m10) $)  {$R^{\frac{1}{2}}$};

\coordinate(sssss) at (0.2, 0, 0);
\coordinate (m11) at ($(lambdao)+(0, 0, -\lambdaz) + (sssss)$);
\coordinate (m12) at ($(lambdao)+(0, -\lambday, -\lambdaz) + (sssss)$);
\draw[gray, <->] (m11)--(m12);
\node[right] at ($(m11)! 0.5 !(m12) $)  {$R^{\frac{5}{6}}$};

\pgfmathsetmacro{\origintwox}{0}
\pgfmathsetmacro{\origintwoy}{0}
\pgfmathsetmacro{\origintwoz}{5}

\coordinate (origin2) at ($(\origintwox, \origintwoy, 0)+(0,0.7,0)$);

\pgfmathsetmacro{\ttaux}{0.3}
\pgfmathsetmacro{\ttauy}{1.5}
\pgfmathsetmacro{\ttauz}{0.3}

\coordinate (ttauo) at ($(origin2)+(\ttaux,\ttauy,0)$);

\draw[blue,fill=blue!30] (ttauo) -- ++(-\ttaux, 0,0)-- ++ (0,-\ttauy, 0) -- ++ (\ttaux, 0, 0)-- cycle;
\draw[blue,fill=blue!30]  (ttauo) -- ++ (0,0, -\ttauz) -- ++ (0,-\ttauy, 0) -- ++ (0,0, \ttauz)--cycle;
\draw[blue,fill=blue!30] (ttauo) -- ++ (-\ttaux, 0,0) -- ++ (0,0,-\ttauz) -- ++ (\ttaux, 0,0) -- cycle;
\draw[blue, ->] ($(ttauo)-(\ttaux+0.1,\ttauy/2,0)$) -- ++(0.2, 0, 0);
\node[left, blue] at ($(ttauo)-(\ttaux+0.1,\ttauy/2,0)$) {$\tilde{\tau}$};

\pgfmathsetmacro{\tSx}{0.3}
\pgfmathsetmacro{\tSy}{1.5}
\pgfmathsetmacro{\tSz}{1.5}

\coordinate (tSorigin) at ($(origin2)+(\tSx,\tSy,0)$);

\draw[thick, brown] (tSorigin) -- ++(-\tSx, 0,0)-- ++ (0,-\tSy, 0) -- ++ (\tSx, 0, 0)-- cycle;
\draw[thick, brown]  (tSorigin) -- ++ (0,0, -\tSz) -- ++ (0,-\tSy, 0) -- ++ (0,0, \tSz)--cycle;
\draw[thick, brown]  (tSorigin) -- ++ (-\tSx, 0,0) -- ++ (0,0,-\tSz) -- ++ (\tSx, 0,0) -- cycle;
\draw[brown , ->] ($(tSorigin)-(\tSx,0,2*\tSz/3)-(0.2,0,0)$) --++(0.2,0,0);
\node[left, brown] at ($(tSorigin)-(\tSx,0,2*\tSz/3)-(0.2,0,0)$) {$\tilde{\Sigma}$};

\pgfmathsetmacro{\tlambdax}{0.3}
\pgfmathsetmacro{\tlambday}{0.3}
\pgfmathsetmacro{\tlambdaz}{0.3}

\coordinate (tlambdao) at ($(origin2)+(1.5,0.7,-1.2)$);

\draw[red, fill=red!30] (tlambdao) -- ++(-\tlambdax, 0,0)-- ++ (0,-\tlambday, 0) -- ++ (\tlambdax, 0, 0)-- cycle;
\draw[red, fill=red!30]  (tlambdao) -- ++ (0,0, -\tlambdaz) -- ++ (0,-\tlambday, 0) -- ++ (0,0, \tlambdaz)--cycle;
\draw[red, fill=red!30] (tlambdao) -- ++ (-\tlambdax, 0,0) -- ++ (0,0,-\tlambdaz) -- ++ (\tlambdax, 0,0) -- cycle;

\draw[red, ->] ($(tlambdao) + (-\tlambdax/2, 0.1, -\tlambdaz)$)--($(tlambdao) + (-\tlambdax/2, 0, -\tlambdaz)$);
\node[red, above] at ($(tlambdao) + (-\tlambdax/2, 0.1, -\tlambdaz)$) {$\tilde{\Lambda}$};

\pgfmathsetmacro{\tBx}{1.5}
\pgfmathsetmacro{\tBy}{1.5}
\pgfmathsetmacro{\tBz}{1.5}

\coordinate (tBorigin) at ($(origin2)+(\tBx,\tBy,0)$); 

\draw[black] (tBorigin) -- ++(-\tBx, 0,0)-- ++ (0,-\tBy, 0) -- ++ (\tBx, 0, 0)-- cycle;
\draw[black]  (tBorigin) -- ++ (0,0, -\tBz) -- ++ (0,-\tBy, 0) -- ++ (0,0, \tBz)--cycle;
\draw[black]  (tBorigin) -- ++ (-\tBx, 0,0) -- ++ (0,0,-\tBz) -- ++ (\tBx, 0,0) -- cycle;

\draw[black, ->] ($(tBorigin)-(0,0.3,\tBz)+(0.2,0,0)$) --++(-0.2,0,0);
\node[right, black] at ($(tBorigin)-(0,0.3,\tBz)+(0.2,0,0)$) {$\tilde{B}$};

\coordinate(tss) at (0.7, 0, 0);
\coordinate (tm1) at ($(tBorigin)+(-\tBx, 0, 0) - (tss)$);
\coordinate (tm2) at ($(tBorigin)+(-\tBx, 0, -\tBz) - (tss)$);
\draw[gray, <->] (tm1)--(tm2);
\node[left] at ($(tm1)! 0.7 !(tm2) $)  {$R^{\frac{1}{2}}$};

\coordinate(tss) at (0.7, 0, 0);
\coordinate (tm3) at ($(tBorigin)+(-\tBx, -\tBy, 0) - (tss)$);
\coordinate (tm4) at ($(tBorigin)+(-\tBx, 0, 0) - (tss)$);
\draw[gray, <->] (tm3)--(tm4);
\node[left] at ($(tm3)! 0.5 !(tm4) $)  {$R^{\frac{1}{2}}$};

\coordinate(tsss) at (0.2, 0, 0);
\coordinate (tm5) at ($(tBorigin)+(-\tBx, 0, 0) - (tsss)$);
\coordinate (tm6) at ($(tBorigin)+(-\tBx, 0, -\ttauz) - (tsss)$);
\draw[gray, <->] (tm5)--(tm6);
\node[left] at ($(tm5)! 0.5 !(tm6) $)  {$R^{\frac{1}{3}}$};

\coordinate(tssss) at (0, 0, 0.2);
\coordinate (tm7) at ($(tBorigin)+(-\tBx, 0, -\tBz) - (tssss)$);
\coordinate (tm8) at ($(tBorigin)+(0, 0, -\tBz) - (tssss)$);
\draw[gray, <->] (tm7)--(tm8);
\node[above] at ($(tm7)! 0.5 !(tm8) $)  {$R^{\frac{1}{2}}$};

\coordinate(tssss) at (0, 0, 0.2);
\coordinate (tm9) at ($(tBorigin)+(-\tBx, -\tBy, 0) + (tssss)$);
\coordinate (tm10) at ($(tBorigin)+(-\tBx+\tSx, -\tBy, 0) + (ssss)$);
\draw[gray, <->] (tm9)--(tm10);
\node[below] at ($(tm9)! 0.5 !(tm10) $)  {$R^{\frac{1}{3}}$};

\coordinate(tsssss) at (0.2, 0, 0);
\coordinate (tm11) at ($(tlambdao)+(0, 0, -\tlambdaz) + (tsssss)$);
\coordinate (tm12) at ($(tlambdao)+(0, -\tlambday, -\tlambdaz) + (tsssss)$);
\draw[gray, <->] (tm11)--(tm12);
\node[right] at ($(tm11)! 0.5 !(tm12) $)  {$R^{\frac{1}{3}}$};

\draw[thick, ->] ($(origin2)+(-1.6,2,0)$) arc (150:30:0.4);
\node[above] at ($(origin2)+(-1.2,2.2,0)$)  {$A_{R^{-1/6},c}$};

\coordinate (corigin) at ($(origin2)+(-1.7,-0.5,0)$);

\draw[->] (corigin)--++(0.5,0,0);
\node[right] at ($(corigin)+(0.5,0,0)$) {$x$};
\draw[->] (corigin) -- ++ (0,0.5, 0);
\node[left] at ($ (corigin) + (0, 0.5, 0)$) {$z$};
\draw[->] (corigin)--++ (0, 0, -0.8);
\node[left] at ($ (corigin)+(0,0,-0.8)$) {$y$};
\end{tikzpicture}
\end{center}
\bigskip

Using this rescaling, the problem of estimating $|\Qc_{M_2, B}(\Tc)|$ transformed into the problem of counting  $R^{\frac13}$-cubes which are $M_2$-rich with respect to the Vinogradov tubes $\tilde{\tau}$ living inside the $R^{\frac12}$-cube $\tilde{B}$. We note that we can apply (the rescaled $R\mapsto R^{1/2}$ version of the) Corollary \ref{co:halfwellspace}, with $[-R,R]^3$ replaced with $\tilde{B}$, $S_I$ replaced with $\tilde{\Sigma}_I$, $N$ replaced with $Z_1$ and with tubes $\tilde{\tau_I}$.

Thus we estimate the number of $M_2$-rich boxes $\Lambda$ inside the box $B$ as follows
\begin{equation}
\label{eqn:Lambda}
    |\Qc_{M_2, B}(\Tc)|\lessapprox \frac{|\{\tau\text{ heavy }:\tau\subset B\}|Z_1 R^{\frac16}}{M_2^2}.
\end{equation}
\\
	
5. Counting small planks $S$ inside each $\Lambda$ by means of Vinogradov plate incidences
\\

Let $H=[c,c+R^{-\frac16}]$, $I=[a,a+R^{-\frac13}]$ and $I\subset H$.	
Recall that each $\Lambda\in\Qc_{M_2}(\Tc)$ intersects $\sim N$ planks $P_I$ for each of $\sim M_2$ many intervals $I$. This is since $\Lambda$ must (due to dimensional considerations) intersect all $N$ planks $P_I$ contained in each heavy $\tau_I$ that intersects (in fact contains) $\Lambda$.
Let us translate each $\Lambda\in\Qc_{M_2}(\Tc)$ to the origin and apply the same linear transformation $A_{R^{-1/6},c}$ as in Step 4.
First, the image of the box $\Lambda$, $\tilde{\Lambda}=A_{R^{-1/6},c}(\Lambda)$ is an $R^{\frac13}$-cube. Next, the image of $P\cap \Lambda$, which we write as
$\tilde{P}$, is an $(R^{\frac16},R^{\frac13},R^{\frac13})$-Vinogradov plate (see Definition \ref{ fjrgop cpreopit}),  with normal vector $\textbf{t}(R^{\frac16}(a-c))$, contained inside the $R^{\frac13}$ cube $\tilde{\Lambda}$. Lastly, the image of $S$, $\tilde{S}=A_{R^{-1/6},c}(S)$ is an $R^{\frac16}$ cube.
\\

\bigskip
\textbf{Figure 4}

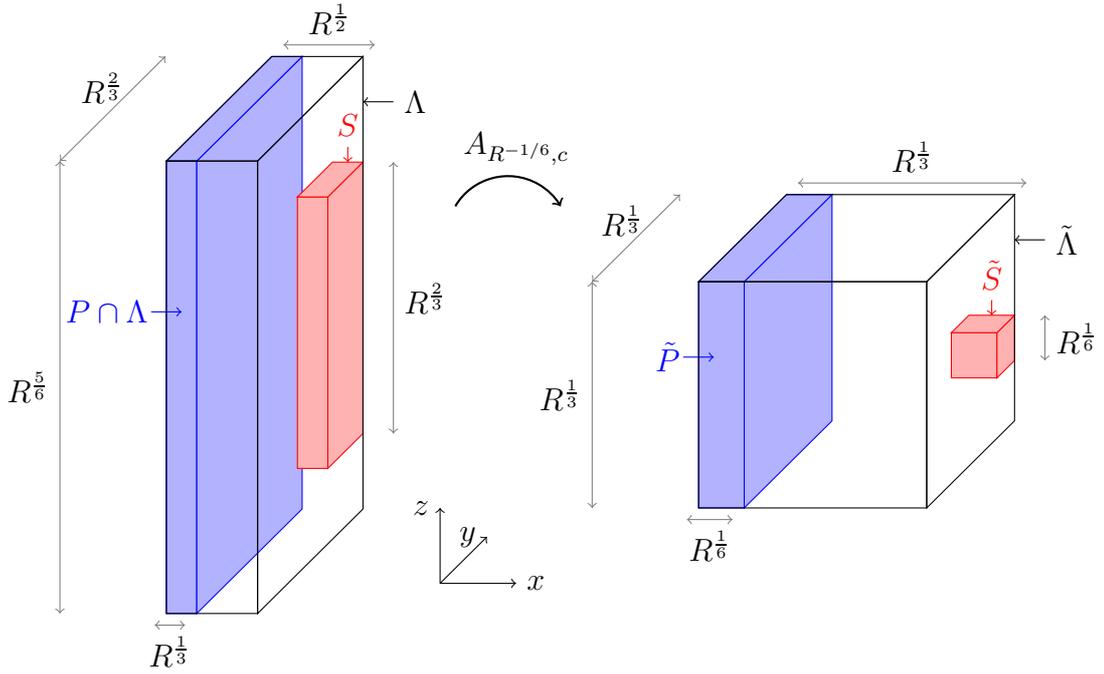
\begin{center}
\begin{tikzpicture}[scale=2]
\pgfmathsetmacro{\originonex}{3.5}
\pgfmathsetmacro{\originoney}{0}
\pgfmathsetmacro{\originonez}{5}

\coordinate (origin1) at (-\originonex, -\originoney, 0);

\pgfmathsetmacro{\Sx}{0.2}
\pgfmathsetmacro{\Sy}{3}
\pgfmathsetmacro{\Sz}{1.8}

\coordinate (Sorigin) at ($(origin1)+(\Sx,\Sy,0)$); 

\draw[blue, fill=blue!30] (Sorigin) -- ++(-\Sx, 0,0)-- ++ (0,-\Sy, 0) -- ++ (\Sx, 0, 0)-- cycle;
\draw[blue, fill=blue!30] (Sorigin) -- ++ (0,0, -\Sz) -- ++ (0,-\Sy, 0) -- ++ (0,0, \Sz)--cycle;
\draw[blue, fill=blue!30]  (Sorigin) -- ++ (-\Sx, 0,0) -- ++ (0,0,-\Sz) -- ++ (\Sx, 0,0) -- cycle;

\draw[blue , ->] ($(Sorigin)-(\Sx+0.1,\Sy/3,0)$) --++(0.2,0,0);
\node[left, blue] at ($(Sorigin)-(\Sx+0.05,\Sy/3,0)$) {$P\cap \Lambda$};

\pgfmathsetmacro{\lambdax}{0.2}
\pgfmathsetmacro{\lambday}{1.8}
\pgfmathsetmacro{\lambdaz}{0.6}

\coordinate (lambdao) at ($(origin1)+(0.6,2.3,-1.2)$);

\draw[red, fill=red!30]  (lambdao) -- ++(-\lambdax, 0,0)-- ++ (0,-\lambday, 0) -- ++ (\lambdax, 0, 0)-- cycle;
\draw[red, fill=red!30]   (lambdao) -- ++ (0,0, -\lambdaz) -- ++ (0,-\lambday, 0) -- ++ (0,0, \lambdaz)--cycle;
\draw[red, fill=red!30]  (lambdao) -- ++ (-\lambdax, 0,0) -- ++ (0,0,-\lambdaz) -- ++ (\lambdax, 0,0) -- cycle;
\draw[red, ->] ($(lambdao) + (-\lambdax/2, 0.1, -\lambdaz)$)--($(lambdao) + (-\lambdax/2, 0, -\lambdaz)$);
\node[red, above] at ($(lambdao) + (-\lambdax/2, 0.1, -\lambdaz)$) {$S$};

\pgfmathsetmacro{\Bx}{0.6}
\pgfmathsetmacro{\By}{3}
\pgfmathsetmacro{\Bz}{1.8}

\coordinate (Borigin) at ($(origin1)+(0.6,3,0)$);

\draw[black] (Borigin) -- ++(-\Bx, 0,0)-- ++ (0,-\By, 0) -- ++ (\Bx, 0, 0)-- cycle;
\draw[black]  (Borigin) -- ++ (0,0, -\Bz) -- ++ (0,-\By, 0) -- ++ (0,0, \Bz)--cycle;
\draw[black]  (Borigin) -- ++ (-\Bx, 0,0) -- ++ (0,0,-\Bz) -- ++ (\Bx, 0,0) -- cycle;

\draw[black, ->] ($(Borigin)-(0,0.3,\Bz)+(0.2,0,0)$) --++(-0.2,0,0);
\node[right, black] at ($(Borigin)-(0,0.3,\Bz)+(0.2,0,0)$) {$\Lambda$};

\coordinate(ss) at (0.7, 0, 0);
\coordinate (m1) at ($(Borigin)+(-\Bx, 0, 0) - (ss)$);
\coordinate (m2) at ($(Borigin)+(-\Bx, 0, -\Bz) - (ss)$);
\draw[gray, <->] (m1)--(m2);
\node[left] at ($(m1)! 0.7 !(m2) $)  {$R^{\frac{2}{3}}$};

\coordinate(ss) at (0.7, 0, 0);
\coordinate (m3) at ($(Borigin)+(-\Bx, -\By, 0) - (ss)$);
\coordinate (m4) at ($(Borigin)+(-\Bx, 0, 0) - (ss)$);
\draw[gray, <->] (m3)--(m4);
\node[left] at ($(m3)! 0.5 !(m4) $)  {$R^{\frac{5}{6}}$};

\coordinate(ssss) at (0, 0, 0.2);
\coordinate (m7) at ($(Borigin)+(-\Bx, 0, -\Bz) - (ssss)$);
\coordinate (m8) at ($(Borigin)+(0, 0, -\Bz) - (ssss)$);
\draw[gray, <->] (m7)--(m8);
\node[above] at ($(m7)! 0.5 !(m8) $)  {$R^{\frac{1}{2}}$};

\coordinate(ssss) at (0, 0, 0.2);
\coordinate (m9) at ($(Borigin)+(-\Bx, -\By, 0) + (ssss)$);
\coordinate (m10) at ($(Borigin)+(-\Bx+\Sx, -\By, 0) + (ssss)$);
\draw[gray, <->] (m9)--(m10);
\node[below] at ($(m9)! 0.5 !(m10) $)  {$R^{\frac{1}{3}}$};

\coordinate(sssss) at (0.2, 0, 0);
\coordinate (m11) at ($(lambdao)+(0, 0, -\lambdaz) + (sssss)$);
\coordinate (m12) at ($(lambdao)+(0, -\lambday, -\lambdaz) + (sssss)$);
\draw[gray, <->] (m11)--(m12);
\node[right] at ($(m11)! 0.5 !(m12) $)  {$R^{\frac{2}{3}}$};

\pgfmathsetmacro{\origintwox}{0}
\pgfmathsetmacro{\origintwoy}{0}
\pgfmathsetmacro{\origintwoz}{5}

\coordinate (origin2) at ($(\origintwox, \origintwoy, 0)+(0,0.7,0)$);

\pgfmathsetmacro{\tSx}{0.3}
\pgfmathsetmacro{\tSy}{1.5}
\pgfmathsetmacro{\tSz}{1.5}

\coordinate (tSorigin) at ($(origin2)+(\tSx,\tSy,0)$);

\draw[blue, fill=blue!30] (tSorigin) -- ++(-\tSx, 0,0)-- ++ (0,-\tSy, 0) -- ++ (\tSx, 0, 0)-- cycle;
\draw[blue, fill=blue!30]  (tSorigin) -- ++ (0,0, -\tSz) -- ++ (0,-\tSy, 0) -- ++ (0,0, \tSz)--cycle;
\draw[blue, fill=blue!30]  (tSorigin) -- ++ (-\tSx, 0,0) -- ++ (0,0,-\tSz) -- ++ (\tSx, 0,0) -- cycle;

\draw[blue , ->] ($(tSorigin)-(\tSx+0.1,\tSy/3,0)$) --++(0.2,0,0);
\node[left, blue] at ($(tSorigin)-(\tSx+0.05,\tSy/3,0)$) {$\tilde{P}$};

\pgfmathsetmacro{\tlambdax}{0.3}
\pgfmathsetmacro{\tlambday}{0.3}
\pgfmathsetmacro{\tlambdaz}{0.3}

\coordinate (tlambdao) at ($(origin2)+(1.5,0.7,-1.2)$);

\draw[red, fill=red!30] (tlambdao) -- ++(-\tlambdax, 0,0)-- ++ (0,-\tlambday, 0) -- ++ (\tlambdax, 0, 0)-- cycle;
\draw[red, fill=red!30]  (tlambdao) -- ++ (0,0, -\tlambdaz) -- ++ (0,-\tlambday, 0) -- ++ (0,0, \tlambdaz)--cycle;
\draw[red, fill=red!30] (tlambdao) -- ++ (-\tlambdax, 0,0) -- ++ (0,0,-\tlambdaz) -- ++ (\tlambdax, 0,0) -- cycle;

\draw[red, ->] ($(tlambdao) + (-\tlambdax/2, 0.1, -\tlambdaz)$)--($(tlambdao) + (-\tlambdax/2, 0, -\tlambdaz)$);
\node[red, above] at ($(tlambdao) + (-\tlambdax/2, 0.1, -\tlambdaz)$) {$\tilde{S}$};

\pgfmathsetmacro{\tBx}{1.5}
\pgfmathsetmacro{\tBy}{1.5}
\pgfmathsetmacro{\tBz}{1.5}

\coordinate (tBorigin) at ($(origin2)+(\tBx,\tBy,0)$); 

\draw[black] (tBorigin) -- ++(-\tBx, 0,0)-- ++ (0,-\tBy, 0) -- ++ (\tBx, 0, 0)-- cycle;
\draw[black]  (tBorigin) -- ++ (0,0, -\tBz) -- ++ (0,-\tBy, 0) -- ++ (0,0, \tBz)--cycle;
\draw[black]  (tBorigin) -- ++ (-\tBx, 0,0) -- ++ (0,0,-\tBz) -- ++ (\tBx, 0,0) -- cycle;

\draw[black, ->] ($(tBorigin)-(0,0.3,\tBz)+(0.2,0,0)$) --++(-0.2,0,0);
\node[right, black] at ($(tBorigin)-(0,0.3,\tBz)+(0.2,0,0)$) {$\tilde{\Lambda}$};

\coordinate(tss) at (0.7, 0, 0);
\coordinate (tm1) at ($(tBorigin)+(-\tBx, 0, 0) - (tss)$);
\coordinate (tm2) at ($(tBorigin)+(-\tBx, 0, -\tBz) - (tss)$);
\draw[gray, <->] (tm1)--(tm2);
\node[left] at ($(tm1)! 0.7 !(tm2) $)  {$R^{\frac{1}{3}}$};

\coordinate(tss) at (0.7, 0, 0);
\coordinate (tm3) at ($(tBorigin)+(-\tBx, -\tBy, 0) - (tss)$);
\coordinate (tm4) at ($(tBorigin)+(-\tBx, 0, 0) - (tss)$);
\draw[gray, <->] (tm3)--(tm4);
\node[left] at ($(tm3)! 0.5 !(tm4) $)  {$R^{\frac{1}{3}}$};

\coordinate(tssss) at (0, 0, 0.2);
\coordinate (tm7) at ($(tBorigin)+(-\tBx, 0, -\tBz) - (tssss)$);
\coordinate (tm8) at ($(tBorigin)+(0, 0, -\tBz) - (tssss)$);
\draw[gray, <->] (tm7)--(tm8);
\node[above] at ($(tm7)! 0.5 !(tm8) $)  {$R^{\frac{1}{3}}$};

\coordinate(tssss) at (0, 0, 0.2);
\coordinate (tm9) at ($(tBorigin)+(-\tBx, -\tBy, 0) + (tssss)$);
\coordinate (tm10) at ($(tBorigin)+(-\tBx+\tSx, -\tBy, 0) + (ssss)$);
\draw[gray, <->] (tm9)--(tm10);
\node[below] at ($(tm9)! 0.5 !(tm10) $)  {$R^{\frac{1}{6}}$};

\coordinate(tsssss) at (0.2, 0, 0);
\coordinate (tm11) at ($(tlambdao)+(0, 0, -\tlambdaz) + (tsssss)$);
\coordinate (tm12) at ($(tlambdao)+(0, -\tlambday, -\tlambdaz) + (tsssss)$);
\draw[gray, <->] (tm11)--(tm12);
\node[right] at ($(tm11)! 0.5 !(tm12) $)  {$R^{\frac{1}{6}}$};

\draw[thick, ->] ($(origin2)+(-1.6,2,0)$) arc (150:30:0.4);
\node[above] at ($(origin2)+(-1.2,2.2,0)$)  {$A_{R^{-1/6},c}$};

\coordinate (corigin) at ($(origin2)+(-1.7,-0.5,0)$);

\draw[->] (corigin)--++(0.5,0,0);
\node[right] at ($(corigin)+(0.5,0,0)$) {$x$};
\draw[->] (corigin) -- ++ (0,0.5, 0);
\node[left] at ($ (corigin) + (0, 0.5, 0)$) {$z$};
\draw[->] (corigin)--++ (0, 0, -0.8);
\node[left] at ($ (corigin)+(0,0,-0.8)$) {$y$};
\end{tikzpicture}
\end{center}
\bigskip

Therefore, the problem of counting $E_2$-planks $S$ inside $\Lambda$ can be transformed into the problem of counting $E_2$-rich cubes $\tilde{S}$ inside the cube $\tilde{\Lambda}$. Note that inside the cube $\tilde{\Lambda}$, there are $\sim N$ many Vinogradov plates $\tilde{P}$ from each of $\sim M_2$ contributing directions. Thus, we can use  Corollary \ref{te:L4plate} with  $\delta=R^{-\frac16}$. We estimate the number of $E_2$-rich planks $S$ inside  each $M_2$-rich $\Lambda$ as follows
\begin{equation}
\label{eqn:SLambda}
|\{S\;\text{is }E_2-rich\;:S\subset \Lambda\}|\lessapprox \frac{N^3M_2(R^{\frac16})^3}{E_2^4}.
\end{equation}
At this point, combining (\ref{eqn:tauestimate}), (\ref{eqn:Lambda}), (\ref{eqn:SLambda}) and  adding up all the $|\S_B|$ estimates, we get the following $|\S|$ estimate in $[-R,R]^3$

\begin{align}
\begin{split}
\label{eqn:Sesitmate}
|\S|&=\sum_{B\subset[-R,R]^3}|\S_B|\\
&\lessapprox\sum_{B\subset[-R,R]^3}\sum_{\Lambda\in \Qc_{M_2,B}(\Tc)}|\{S\;\text{is }E_2-rich\;:S\subset \Lambda\}|\\
&\lessapprox\sum_{B\subset[-R,R]^3}\sum_{\Lambda\in \Qc_{M_2,B}(\Tc)} \frac{N^3M_2(R^{\frac16})^3}{E_2^4}\\
&\lessapprox\sum_{B\subset[-R,R]^3}\frac{|\{\tau\text{ heavy }:\tau\subset B\}|Z_1 R^{\frac16}}{M_2^2} \frac{N^3M_2(R^{\frac16})^3}{E_2^4}\\
&\lesssim\frac{\frac{|\P|}{N}Z_1R^{\frac16}}{M_2^2}\frac{N^3M_2(R^{\frac16})^3}{E_2^4}\\
&=\frac{|\P|N^2Z_1(R^{\frac16})^4}{M_2E_2^4}.
\end{split}
\end{align}
\\

We are half way through the argument. All further boxes will be localized inside smaller cubes $Q$ with side length $R^{2/3}$. Note that each $S$ fits into one of these cubes. The next few steps produce estimates for the number of incidences between small planks $S$ inside such a cube.
\\
	
6. Pigeonholing the parameter $U_1$: the tubes $T$
\\

For each $H\in \I_{R^{-1/6}}$, we tile  $[-R,R]^3$ with $(R^{\frac12},R^{\frac12},R^{\frac23})$-tubes $T$ with direction $\textbf{b}(H)$.
Note that $R^{-1/6}T$ is a Vinogradov tube at scale $R^{1/2}$ associated with $H$, cf. Definition \ref{oiufufihvjdkvnjkgbj }.
We can assume that each small plank $S$ is uniquely contained in one of the tubes. Let us call the set of $T$ as $\T$ and partition $\T$ according to the dyadic parameter $U_1$, such that each tube in the family contains $\sim U_1$ small $E_2$-planks $S\in \S$. We fix the parameter $U_1\leq R^{\frac16}$ and note that $|\T|\lesssim \frac{|\S|}{U_1}$.
\\
\\

\bigskip
\textbf{Figure 5}
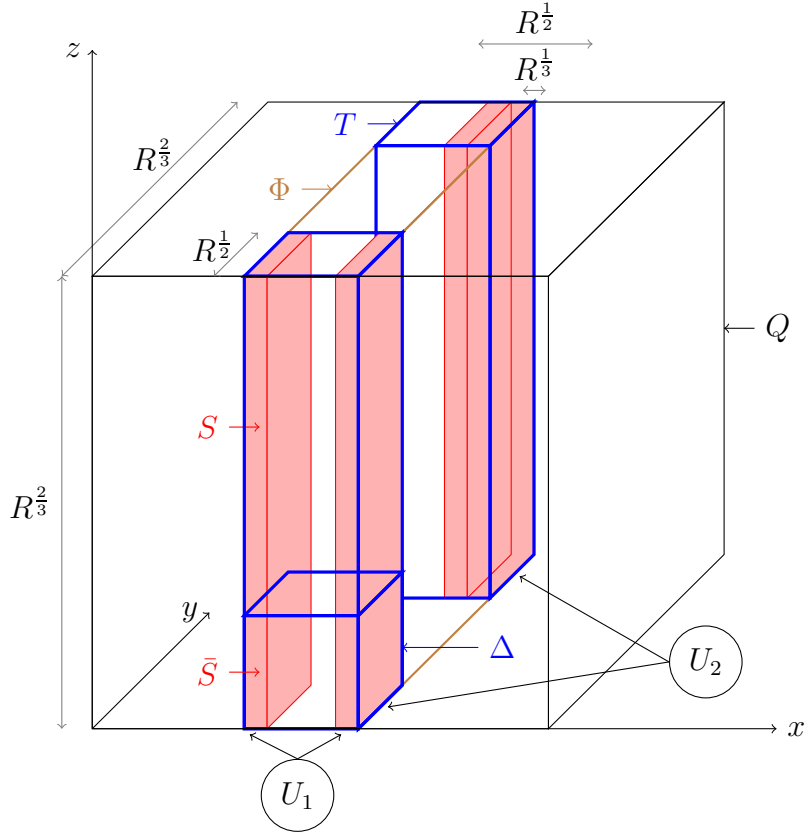
\begin{center}
\begin{tikzpicture}[scale=2,roundnode/.style={circle, draw=black, minimum size=0.1}]

\pgfmathsetmacro{\originonex}{2}
\pgfmathsetmacro{\originoney}{0}
\pgfmathsetmacro{\originonez}{5}

\coordinate (origin) at (-\originonex, -\originoney, 0);

\pgfmathsetmacro{\Sx}{0.15}
\pgfmathsetmacro{\Sy}{3}
\pgfmathsetmacro{\Sz}{0.75}

\coordinate (Sorigin) at ($(origin)+(\Sx,\Sy,0)$);

\draw[red,fill=red!30] (Sorigin) -- ++(-\Sx, 0,0)-- ++ (0,-\Sy, 0) -- ++ (\Sx, 0, 0)-- cycle;
\draw[red,fill=red!30] (Sorigin) -- ++ (0,0, -\Sz) -- ++ (0,-\Sy, 0) -- ++ (0,0, \Sz)--cycle;
\draw[red,fill=red!30] (Sorigin) -- ++ (-\Sx, 0,0) -- ++ (0,0,-\Sz) -- ++ (\Sx, 0,0) -- cycle;

\draw[red, ->] ($(Sorigin)-(\Sx+0.1,\Sy/3,0)$) -- ++(0.2, 0, 0);
\node[left, red] at ($(Sorigin)-(\Sx+0.1,\Sy/3,0)$) {$S$};

\coordinate (SSSorigin) at ($(origin)+(\Sx,\Sy,0)+(4*\Sx,0,-3*\Sz)$);
\pgfmathsetmacro{\Qx}{3}
\pgfmathsetmacro{\Qy}{3}
\pgfmathsetmacro{\Qz}{3}

\draw[red,fill=red!30] (SSSorigin) -- ++(-\Sx, 0,0)-- ++ (0,-\Sy, 0) -- ++ (\Sx, 0, 0)-- cycle;
\draw[red,fill=red!30] (SSSorigin) -- ++ (0,0, -\Sz) -- ++ (0,-\Sy, 0) -- ++ (0,0, \Sz)--cycle;
\draw[red,fill=red!30] (SSSorigin) -- ++ (-\Sx, 0,0) -- ++ (0,0,-\Sz) -- ++ (\Sx, 0,0) -- cycle;

\coordinate (SSSSorigin) at ($(origin)+(\Sx,\Sy,0)+(3*\Sx,0,-3*\Sz)$);

\draw[red,fill=red!30] (SSSSorigin) -- ++(-\Sx, 0,0)-- ++ (0,-\Sy, 0) -- ++ (\Sx, 0, 0)-- cycle;
\draw[red,fill=red!30] (SSSSorigin) -- ++ (0,0, -\Sz) -- ++ (0,-\Sy, 0) -- ++ (0,0, \Sz)--cycle;
\draw[red,fill=red!30] (SSSSorigin) -- ++ (-\Sx, 0,0) -- ++ (0,0,-\Sz) -- ++ (\Sx, 0,0) -- cycle;

\pgfmathsetmacro{\Phix}{0.75}
\pgfmathsetmacro{\Phiy}{3}
\pgfmathsetmacro{\Phiz}{3}

\coordinate (Phio) at ($(origin)+(\Phix,\Phiy,0)$); 

\draw[thick, brown] (Phio) -- ++(-\Phix, 0,0)-- ++ (0,-\Phiy, 0) -- ++ (\Phix, 0, 0)-- cycle;
\draw[thick, brown] (Phio) -- ++ (0,0, -\Phiz) -- ++ (0,-\Phiy, 0) -- ++ (0,0, \Phiz)--cycle;
\draw[thick, brown]  (Phio) -- ++ (-\Phix, 0,0) -- ++ (0,0,-\Phiz) -- ++ (\Phix, 0,0) -- cycle;

\draw[brown , ->] ($(Phio)-(\Phix,0,\Phiz/2)-(0.2,0,0)$) --++(0.2,0,0);
\node[left, brown] at ($(Phio)-(\Phix,0,\Phiz/2)-(0.2,0,0)$) {$\Phi$};

\pgfmathsetmacro{\Tx}{0.75}
\pgfmathsetmacro{\Ty}{3}
\pgfmathsetmacro{\Tz}{0.75}

\coordinate (TTorigin) at ($(origin)+(\Tx,\Ty,-3*\Tz)$);

\draw[blue, very thick] (TTorigin) -- ++(-\Tx, 0,0)-- ++ (0,-\Ty, 0) -- ++ (\Tx, 0, 0)-- cycle;
\draw[blue, very thick]  (TTorigin) -- ++ (0,0, -\Tz) -- ++ (0,-\Ty, 0) -- ++ (0,0, \Tz)--cycle;
\draw[blue, very thick] (TTorigin) -- ++ (-\Tx, 0,0) -- ++ (0,0,-\Tz) -- ++ (\Tx, 0,0) -- cycle;

\draw[blue, ->] ($(TTorigin)-(\Tx+0.2,0,\Tz/2)$) -- ++(0.2, 0, 0);
\node[left, blue] at ($(TTorigin)-(\Tx+0.2,0,\Tz/2)$) {$T$};

\coordinate (SSorigin) at ($(origin)+(\Sx,\Sy,0)+(4*\Sx,0,0)$);

\draw[red,fill=red!30] (SSorigin) -- ++(-\Sx, 0,0)-- ++ (0,-\Sy, 0) -- ++ (\Sx, 0, 0)-- cycle;
\draw[red,fill=red!30] (SSorigin) -- ++ (0,0, -\Sz) -- ++ (0,-\Sy, 0) -- ++ (0,0, \Sz)--cycle;
\draw[red,fill=red!30] (SSorigin) -- ++ (-\Sx, 0,0) -- ++ (0,0,-\Sz) -- ++ (\Sx, 0,0) -- cycle;

\coordinate (Torigin) at ($(origin)+(\Tx,\Ty,0)$);

\draw[blue, very thick] (Torigin) -- ++(-\Tx, 0,0)-- ++ (0,-\Ty, 0) -- ++ (\Tx, 0, 0)-- cycle;
\draw[blue, very thick] (Torigin) -- ++ (0,0, -\Tz) -- ++ (0,-\Ty, 0) -- ++ (0,0, \Tz)--cycle;
\draw[blue, very thick](Torigin) -- ++ (-\Tx, 0,0) -- ++ (0,0,-\Tz) -- ++ (\Tx, 0,0) -- cycle;

\pgfmathsetmacro{\deltax}{0.75}
\pgfmathsetmacro{\deltay}{0.75}
\pgfmathsetmacro{\deltaz}{0.75}

\coordinate (deltao) at ($(origin)+(\deltax,\deltay,0)$);

\draw[blue, very thick] (deltao) -- ++(-\deltax, 0,0)-- ++ (0,-\deltay, 0) -- ++ (\deltax, 0, 0)-- cycle;
\draw[blue, very thick] (deltao) -- ++ (0,0, -\deltaz) -- ++ (0,-\deltay, 0) -- ++ (0,0, \deltaz)--cycle;
\draw[blue, very thick](deltao) -- ++ (-\deltax, 0,0) -- ++ (0,0,-\deltaz) -- ++ (\deltax, 0,0) -- cycle;

\draw[red , ->] ($(Sorigin)-(\Sx,\Sy,0)+(-0.1,\deltay/2,0)$) --++(0.2,0,0);
\node[left, red] at ($(Sorigin)-(\Sx,\Sy,0)+(-0.1,\deltay/2,0)$) {$\bar{S}$};
\draw[blue, ->] ($(deltao)-(0,0.5,\deltaz)+(0.5,0,0)$) --++(-0.5,0,0);
\node[right, blue] at ($(deltao)-(0,0.5,\deltaz)+(0.5,0,0)$) {$\Delta$};

\pgfmathsetmacro{\Qx}{3}
\pgfmathsetmacro{\Qy}{3}
\pgfmathsetmacro{\Qz}{3}

\coordinate (Qorigin) at ($(origin)+(2,3,0)$);

\draw[black] (Qorigin) -- ++(-\Qx, 0,0)-- ++ (0,-\Qy, 0) -- ++ (\Qx, 0, 0)-- cycle;
\draw[black]  (Qorigin) -- ++ (0,0, -\Qz) -- ++ (0,-\Qy, 0) -- ++ (0,0, \Qz)--cycle;
\draw[black]  (Qorigin) -- ++ (-\Qx, 0,0) -- ++ (0,0,-\Qz) -- ++ (\Qx, 0,0) -- cycle;

\draw[black, ->] ($(Qorigin)-(0,\Qy/2,\Qz)+(0.2,0,0)$) --++(-0.2,0,0);
\node[right, black] at ($(Qorigin)-(0,\Qy/2,\Qz)+(0.2,0,0)$) {$Q$};

\draw[black, ->] ($(Sorigin)-(-0.2,\Sy+0.2,0)$) --($(Sorigin)-(\Sx/2,\Sy,-0.1)$);
\draw[black, ->] ($(Sorigin)-(-0.2,\Sy+0.2,0)$) --($(SSorigin)-(\Sx/2,\Sy,-0.1)$);
\node [roundnode]  [below] at ($(Sorigin)-(-0.2,\Sy+0.2,0)$) {$U_1$};


\draw[black, ->] ($(Torigin)-(-1.6,\Ty,\Tz+0.4)$) --($(TTorigin)-(-0.05,\Ty,\Tz/2)$);
\draw[black, ->] ($(Torigin)-(-1.6,\Ty,\Tz+0.4)$) --($(Torigin)-(-0.05,\Ty,\Tz/2)$);
\node[roundnode][right] at ($(Torigin)-(-1.6,\Ty,\Tz+0.4)$) {$U_2$};

\coordinate(ss) at (0.2, 0, 0);
\coordinate (m1) at ($(Qorigin)+(-\Qx, 0, 0) - (ss)$);
\coordinate (m2) at ($(Qorigin)+(-\Qx, 0, -\Qz) - (ss)$);
\draw[gray, <->] (m1)--(m2);
\node[left] at ($(m1)! 0.7 !(m2) $)  {$R^{\frac{2}{3}}$};

\coordinate(ss) at (0.2, 0, 0);
\coordinate (m3) at ($(Qorigin)+(-\Qx, -\Qy, 0) - (ss)$);
\coordinate (m4) at ($(Qorigin)+(-\Qx, 0, 0) - (ss)$);
\draw[gray, <->] (m3)--(m4);
\node[left] at ($(m3)! 0.5 !(m4) $)  {$R^{\frac{2}{3}}$};

\coordinate(sss) at (0.2, 0, 0);
\coordinate (m5) at ($(Sorigin)+(-\Sx, 0, 0) - (sss)$);
\coordinate (m6) at ($(Sorigin)+(-\Sx, 0, -\Sz) - (sss)$);
\draw[gray, <->] (m5)--(m6);
\node[left] at ($(m5)! 0.7 !(m6) $)  {$R^{\frac{1}{2}}$};

\coordinate(ssss) at (0, 0, 0.2);
\coordinate (m9) at ($(SSSorigin)+(-\Sx, 0, -\Sz) - (ssss)$);
\coordinate (m10) at ($(SSSorigin)+(0,0, -\Sz) - (ssss)$);
\draw[gray, <->] (m9)--(m10);
\node[above] at ($(m9)! 0.5 !(m10) $)  {$R^{\frac{1}{3}}$};

\coordinate(sssss) at (0, 0, 1);
\coordinate (m11) at ($(TTorigin)+(-\Tx, 0, -\Tz) - (sssss)$);
\coordinate (m12) at ($(TTorigin)+(0, 0, -\Tz) - (sssss)$);
\draw[gray, <->] (m11)--(m12);
\node[above] at ($(m11)! 0.5 !(m12) $)  {$R^{\frac{1}{2}}$};

\coordinate (corigin) at ($(origin)+(-1,0,0)$);

\draw[->] (corigin)--++(4.5,0,0);
\node[right] at ($(corigin)+(4.5,0,0)$) {$x$};
\draw[->] (corigin) -- ++ (0,4.5, 0);
\node[left] at ($ (corigin) + (0, 4.5, 0)$) {$z$};
\draw[->] (corigin)--++ (0, 0, -2);
\node[left] at ($ (corigin)+(0,0,-2)$) {$y$};
\end{tikzpicture}
\end{center}
\bigskip

7. Plates $\Phi$ and a  bound for $U_2$ via  $L^2$ Kakeya for plates
\\

We tile $[-R,R]^3$ with  cubes $Q$ with side length $R^\frac23$. For each $H\in \I_{R^{-\frac16}}$, we tile each $Q$ with $(R^{\frac12},R^{\frac23},R^{\frac23})$-plates $\Phi_H$ with normal vector $\textbf{t}(H)$. Each tube $T_H\in \T$ associated with $H$ is uniquely contained in one of the plates $\Phi_H$.
For each plate $\Phi_H$, we will bound the number $U_2$ of tubes $T_H$ they contain. Recall that the plate $\Phi_H$ containing $T_H$ should intersect $\sim N$  planks $P_I$, for each of $\sim M_2$ contributing $I\subset H$. The plank $P_I$ truncated to $\Phi$ is an $(R^{1/3},R^{2/3},R^{2/3})$ Vinogradov plate associated with $I$. We denote such plate by $\bar{P}_I$. The intersections of  such plates $\bar{P}_I$ for $I\subset H$ are unions of small planks $S_H$, cf. Lemma \ref{le:platesmallangle}. We combine
Chebyshev's inequality with Theorem \ref{te:L2plate} to get

$$U_1U_2E_2^2|S_H|\le \|\sum_{}1_{\bar{P}_I}\|_2^2\lessapprox N(\sum_{}|\bar{P}_I|).$$
The term $U_1U_2$ on the left represents the number of $E_2$-rich small planks $S_H$ inside $\Phi_H$. We conclude that
\begin{equation}
\label{eqn:linkakeya}
    U_1U_2\lessapprox \frac{N^2M_2R^{\frac16}}{E_2^2}.
\end{equation}
We fix the parameter $U_2$, and the corresponding family of plates $\Phi$.	
\\
	
8. An upper bound for $U_1$ via a  double counting argument
	\\
	
Each plate $\Phi_H$ can intersect $\sim N$ planks $P_I$ from each of $M_2$ contributing directions $I\subset H$. Each tube $T_H$ inside $\Phi_H$ contains $U_1$ many $E_2$-rich planks $S_H$. Therefore, the tube $T_H$ intersects at least $U_1E_2$ many planks $P_I$, since each $P_I$ can intersect at most $O(1)$ smaller planks $S_H\subset T_H$. We conclude that
\begin{equation}
\label{eqn:doublecounting}
    U_1E_2\lesssim NM_2.
\end{equation}
\\
	
9. Pigeonholing the parameter $M_1$
\\
	
Let us tile each $Q$ with $R^{\frac12}$-cubes $\Delta$. Each small $R^{\frac13}$-cube $q$ in $\Qc_{E_1}(\S)$ lies inside one of $\Delta$.
We will only focus on those $\Delta$ containing at least one $q\in \Qc_{E_1}(\S)$.

The cubes $\Delta$ are the typical intersections between the tubes  $T$. Since each $q$ is $E_1$-rich with respect to the family of small planks $S$, and since each $S$ lies in some tube $T$, it follows that  $\Delta$  can be assumed to be $M_1$-rich with respect to the family $\T$, for some $M_1\ge E_1$. There are $\lesssim \log R$ choices of parameter $M_1\leq R^{\frac16}$. For fixed $M_1\geq E_1$, we denote the collection of $M_1$-rich cubes $\Delta$ by $\Qc_{M_1}(\T)$.
\\
	
10. Counting  $M_1$-rich $R^{\frac12}$-cubes $\Delta$ using Vinogradov tube incidences
\\
	
Let us fix an $R^{\frac{2}{3}}$-cube $Q$. For each $H$, each plate $\Phi_H$ inside $Q$ contains at most $U_2$ many  tubes $T_H$. Writing the family of $M_1$-rich $\Delta$ inside $Q$ as $\Qc_{M_1,Q}(\T)$, we get the following estimate by using Corollary \ref{co:halfwellspace}
\begin{equation}
\label{eqn:Delta}
    |\Qc_{M_1, Q}(\T)|\lessapprox \frac{|\{T\in \T:T\subset Q\}|U_2 R^{\frac16}}{M_1^2}.
\end{equation}
\\

11. Counting $E_1$-rich $R^{\frac13}$-cubes $q$ inside each $\Delta$ by Vinogradov plate incidences
\\
	
The cube $\Delta\in\Qc_{M_1, Q}(\T)$ intersects $\sim U_1$ many $E_2$-planks $S$ from each of the $\sim M_1$ contributing directions. Let us call $\bar{S}=S\cap \Delta$. This is a rescaled Vinogradov plate. Therefore. we can use Theorem \ref{te:L4plate} as before, to get
\begin{equation}
\label{eqn:TiiDelta}
|\{q\in \Qc_{E_1}(\S):q\subset \Delta\}|\lessapprox \frac{U_1^3M_1(R^{\frac16})^3}{E_1^4}.
\end{equation}
At this point, combining the estimates (\ref{eqn:Delta}) and (\ref{eqn:TiiDelta}), we sum up $|\{q\in \Qc_{E_1}(\S):q\subset Q\}|$  over the cubes $Q\subset [-R,R]^3$ to get the estimate for $|\Qc_r(\P)|$

\begin{align}
\begin{split}
\label{qestimate}
|\Qc_{r}(\P)|&\lessapprox|\Qc_{E_1}(\S)|=\sum_{Q\subset [-R,R]^3}|\{q\in \Qc_{E_1}(\S):q\subset Q\}|\\
&\lessapprox\sum_{Q\subset [-R,R]^3}\sum_{\Delta\in\Qc_{M_1, Q}(\T)}|\{q\in \Qc_{E_1}(\S):q\subset \Delta\}|\\
&\lessapprox\sum_{Q\subset [-R,R]^3}\sum_{\Delta\in\Qc_{M_1, Q}(\T)}\frac{U_1^3M_1(R^{\frac16})^3}{E_1^4}\\
&\lessapprox\sum_{Q\subset [-R,R]^3}\frac{|\{T\in \T:T\subset Q\}|U_2 R^{\frac16}}{M_1^2}\frac{U_1^3M_1(R^{\frac16})^3}{E_1^4}\\
&=\frac{|\T| R^{\frac16}}{M_1^2}\frac{U_1M_1(R^{\frac16})^3}{E_1^4}U_1(U_1U_2)\\
&\lessapprox\frac{|\T|U_1N^3M_2^2(R^{\frac16})^5}{M_1E_1^4E_2^3}.
\end{split}
\end{align}
The last inequality follows from (\ref{eqn:linkakeya}) and (\ref{eqn:doublecounting}).
\\

12. Reaching the final estimate
	\\

Recall that $|\T|\lesssim \frac{|\S|}{U_1}$. Combining the last inequalities of (\ref{eqn:Sesitmate}) and (\ref{qestimate}), we get the following final estimate
\begin{align*}
|\Qc_r(\P)|&\lessapprox \frac{|\S|N^3M_2^2(R^{\frac16})^5}{M_1E_1^4E_2^3}\\
&\lessapprox \frac{|\P|N^2Z_1(R^{\frac16})^4}{M_2E_2^4}\frac{N^3M_2^2(R^{\frac16})^5}{M_1E_1^4E_2^3}\\
&=\frac{|\P|N^5Z_1R^{\frac32}M_2}{r^4M_1E_2^3}\\
&= \frac{|\P|N^5Z_1R^{\frac32}M_2E_1^3}{r^4M_1E_2^3E_1^3}\\
&\leq \frac{|\P|N^5Z_1R^{2}}{r^7}.\\
\end{align*}
The last inequality follows  from $r=E_1E_2$, $E_1\leq M_1$, $E_1\leq R^{\frac16}$, $M_2\leq R^{\frac16}$. This finishes the proof of Proposition \ref{te:Plankinci}.
\end{proof}

\end{document}